\documentclass{amsart}
\usepackage{xypic}
\usepackage[mathscr]{eucal} 
\usepackage[dvipsnames]{color}
\newcommand{\I}{{(1+I)^*}}
\newcommand{\Xfloc}{X_F^{\textup{loc}}}
\newcommand{\loc}{{\textup{loc}}}
\def\Lmod #1 {\textup{L}_{\textup{mod},#1}}
\newcommand{\Ltwo}{\textup{L}_2}
\def\dirlim_#1{\mathchoice{\underset{#1}{\varinjlim}}{\varinjlim_{#1}}{}{}}
\newcommand{\ppair}[1]{\pair{#1}^\prime}
\newcommand{\pint}{\int^\prime}
\newcommand{\someterm}{\mathcal{T}}
\newcommand{\fbar}{{\bar{f}}}
\newcommand{\gbar}{{\bar{g}}}
\newcommand{\End}{\operatorname{\mathsf{End}}}
\newcommand{\Ccomp}[2]{\mathcal{C}^{#1}(#2)}
\newcommand{\sgn}{\operatorname{sgn}}
\newcommand{\C}{{\mathbb{C}}}
\newcommand{\Cp}{{\mathbb{C}_p}}
\newcommand{\Fpb}{{\overline{\mathbb F}_p}}

\newcommand{\Hdr}{H_{\dR}}
\def\K #1 #2 #3 {K_{#1}^{(#2)}(#3)}
\def\iso{\simeq}
\newcommand{\Ker}{\operatorname{Ker}}
\newcommand{\kres}{{\mathbb{F}}}
\newcommand{\PP}{\mathbb{P}}
\newcommand{\Q}{\mathbb{Q}}
\newcommand{\R}{\mathbb{R}}
\newcommand{\Sym}{\operatorname{Sym}}
\newcommand{\Z}{\mathbb{Z}}

\newcommand{\dpair}[1]{{\ll #1 \gg}}
\newcommand{\canproj}{\mathbf{p}}
\newcommand{\dd}{{\textup d}}
\newcommand{\dR}{{\textup{dR}}}
\newcommand{\dlogf}{\dlog \phi^\ast}
\newcommand{\dlog}{\operatorname{d\!\log}}
\newcommand{\fdiv}[1]{\frac{\phi^\ast}{q^{#1}}}

\newcommand{\FQ}{F_\Q^*}
\newcommand{\FpQ}{F_\Q^{\prime*}}

\newcommand{\hr}{H_{\rig}}
\newcommand{\hsyn}{H_{\syn}}
\newcommand{\htms}{\tilde{H}_{\ms}}
\newcommand{\id}{\operatorname{id}}
\newcommand{\inject}{\hookrightarrow}
\newcommand{\isom}{\cong}
\newcommand{\logf}{\log \phi^\ast}
\newcommand{\ms}{{\textup{ms}}}
\newcommand{\ord}{\operatorname{ord}}
\newcommand{\ovqs}{\frac{1}{q^2}}
\newcommand{\ovq}{\frac{1}{q}}
\newcommand{\reg}{\operatorname{reg}_p}
\newcommand{\res}{\operatorname{Res}}
\newcommand{\rig}{{\textup{rig}}}
\def\scriptM #1 #2 {{\mathcal M}_{(#1)}(#2)}
\newcommand{\spec}{\operatorname{Spec}}
\newcommand{\syn}{{\textup{syn}}}
\newcommand{\tensor}{\otimes}
\def\tildescriptM #1 #2 {\widetilde{\mathcal M}_{(#1)}(#2)}
\newcommand{\tM}{\widetilde M}
\renewcommand{\O}{\mathcal{O}}

\newcommand{\alog}{A_{\textup{log}}}
\newcommand{\acol}{A_{\textup{col}}}
\newcommand{\aloga}{A_{\textup{log},1}}
\newcommand{\acola}{A_{\textup{col},1}}
\newcommand{\ocol}{\Omega_{\textup{col}}}
\newcommand{\Ocola}{\Omega_{\textup{col},1}^1}
\newcommand{\mer}{\mathsf{Mer}}
\newcommand{\gl}{{\textup{gl}}}
\newcommand{\pair}[1]{{\left\langle #1 \right\rangle}}
\newcommand{\bbox}{\openbox}
\newcommand{\Li}{\operatorname{Li}}
\newcommand{\symb}[2]{[#1]_2\cup (#2)}
\newcommand{\kbb}{K_2^{(2)}}

\newcommand{\CC}{{\mathscr{C}}}
\newcommand{\Csp}{\CC_{\kappa}}
\newcommand{\Cspff}{{\kappa(\CC_\kappa)}}
\newcommand{\Cspnf}{\CC_{\kres}'}
\newcommand{\Cspnfff}{\kres(\CC_{\kres}')}
\newcommand{\tmg}{\frac{t-g}{t-1}}
\newcommand{\omg}{\dlog (\tmg)\wedge \dlog (1-g)}
\newcommand{\ovqt}{\frac{1}{q^3}}
\newcommand{\epsg}{\ovq \log(1-g)_0 \dlog\left(\tmg\right) - \ovqs \log \left(\tmg\right)_0
  \dlog \phi^\ast (1-g)}
\newcommand{\epsgi}{\ovq \log(1-g)_0 \dlog g - \ovqs \log g_0 \dlog \phi^\ast (1-g)} 

\newcommand{\regmapa}{\Psi_{p,\omega}}
\newcommand{\regmapb}{\Psi_{p,\omega}^\prime}
\newcommand{\regmapc}{\Psi_{p,\omega}^{\prime\prime}}
\newcommand{\regmapd}{\Psi_{p,\omega}^{\prime\prime\prime}}

\newtheorem{theorem}[equation]{Theorem}
\newtheorem{conjecture}[equation]{Conjecture}
\newtheorem{notation}[equation]{Notation}
\newtheorem{assumption}[equation]{Assumption}
\newtheorem{proposition}[equation]{Proposition}
\newtheorem{prop}[equation]{Proposition}
\newtheorem{lemma}[equation]{Lemma}
\newtheorem{corollary}[equation]{Corollary}
\theoremstyle{definition}
\newtheorem{definition}[equation]{Definition}
\newtheorem{remark}[equation]{Remark}

\newtheorem{wish}[equation]{Wish}

\theoremstyle{plain}
\numberwithin{equation}{section}
\numberwithin{figure}{section}

\def\Pmod #1 {P_{#1,\textup{mod}}}
\def\Pzag #1 {P_{#1,\textup{Zag}}}
\def\Can{{C_{\textup{an}}}}

\def\myphi{\Psi_{\infty,\omega}^{\prime\prime}}

\def\tmyphi{\Psi_{\infty,\omega}^{\prime\prime\prime}}

\newcommand{\alt}{{\textup{alt}}}
\newcommand{\IO}{(1+I)_\O^*}
\newcommand{\kk}{\kappa}
\def\Kab #1 #2 #3 #4 {K_{#1, #4}^{(#2),#3}}
\newcommand{\leftiso}{\,{\overset \simeq \leftarrow}\,}

\newcommand{\SPkk}{{\kk^\flat}}
\newcommand{\OQ}{{\O_{\Q}^*}}
\newcommand{\OpQ}{{\O_{\Q}^{\prime*}}}
\def\regC{\operatorname{reg}_{\C}}
\newcommand{\rightiso}{\,{\overset \simeq \rightarrow}\,}
\newcommand{\SPF}{{F^\flat}}
\newcommand{\SPO}{{\O^\flat}}
\newcommand{\Symb}{{\textup {Symb}}}
\newcommand{\tcup}{\tilde\cup}
\newcommand{\XFLOC}{X_F^{'\loc}}
\newcommand{\Xoloc}{X_\O^{\loc}}

\def\reliso{T_0^\infty}
\long\def\comment#1\endcomment{}

\begin{document}

\title{The syntomic regulator for $K_4$ of curves}

\author{Amnon Besser}
\address{
Department of Mathematics\\
Ben-Gurion University of the Negev\\
P.O.B. 653\\
Be'er-Sheva 84105\\
Israel
}

\author{Rob de Jeu}
\address{Faculteit der Exacte Wetenschappen\\Afdeling Wiskunde\\VU University Amsterdam\\De Boelelaan 1081a\\1081 HV Amsterdam\\The Netherlands}

\begin{abstract}
Let $ C $ be a  curve defined over a discrete valuation field of characteristic zero where
the residue field has positive characteristic.
Assuming that $ C $ has good reduction over the residue field, 
we compute the syntomic regulator on a certain part
of $ \K 4 3 C $.
The result can be expressed in terms of $ p $-adic polylogarithms and Coleman integration.
We also compute the syntomic regulator on a certain part of $ \K 4 3 F $
for the function field $ F $ of $ C $. The result can be expressed
in terms of $ p $-adic polylogarithms and Coleman integration,
or by using a trilinear map (``triple index'') on certain functions.
\end{abstract}



\maketitle

\dedicatory{Dedicated to the Memory of Jon Rogawski}

\section{Introduction}
\label{sec:intro}

Let $ K $ be a complete discrete valuation field of
characteristic zero, $ R $ its valuation ring,
and $ \kappa $ its residue field. Assume $ \kappa $ is of positive
characteristic $ p $.
If $ \mathscr{X}/R $ is a scheme, smooth and of finite type, 
then, after tensoring with $ \Q $, one can decompose the $ K $-theory 
of $ \mathscr{X} $ according to the Adams weight eigenspaces, i.e.,
\[ 
K_n(\mathscr{X}) \tensor_\Z \Q = \sum_j \K n j {\mathscr{X}}
\,,
\]
where $ \K n j {\mathscr{X}} $  consists of those $ \alpha $ in $ K_n(\mathscr{X}) \tensor_\Z \Q $
such that $ \psi^k(\alpha) = k^j \alpha $ for all Adams operators $ \psi^k $;
see~\cite[Proposition~5]{Sou85}.
There is a regulator map
\[
\K n j \mathscr{X} \to \hsyn^{2j-n}(\mathscr{X}, j)
\]
(see~\cite{Bes98a}). In many interesting cases
the target group of the regulator is isomorphic to the
rigid cohomology group of the special fiber $\mathscr{X}_\kappa$, in
the sense of Berthelot, $\hr^{2j-n-1}(\mathscr{X}_\kappa/K)$.
We will be interested in the situation where $ \mathscr{X} $
is a proper, irreducible, smooth curve $ \CC $
over $ R $ with a geometrically irreducible generic fiber $ C $, 
and the $ K $-group is $ \K 4 3 {\CC} $.  $ \K 4 3 C $ is known
to be isomorphic to $ \K 4 3 {\CC} $ under localization; see
Section~\ref{prelimk}.
The target group 
for the regulator in this case is $ \hr^1(\CC_\kappa/K)  \isom
\Hdr^1(C/K)$. The cup product gives a pairing
$$
\Hdr^1(C/K) \times \Hdr^1(C/K) \to \Hdr^2(C/K) \isom K
$$
where the last isomorphism is given by the trace map. We will denote
this pairing by $\cup$ as well.
If $ \alpha $ is an element of $ \K 4 3 C $ and $ \omega \in
\Hdr^1(C/K) $, we want to compute $ \reg(\alpha)\cup \omega \in  K $.

To achieve this goal, we first of all need to be able to write
elements in the above mentioned $ K $-group. We do this using an integral
version of the motivic complexes introduced by the second named
author. The complex $ \scriptM 3 F $ was defined in \cite[Section~3]{Jeu95}
for any field $ F $ of characteristic zero.  It consists
of three terms in cohomological degrees 1, 2 and 3:
\begin{equation}
M_3(F) \to M_2(F) \tensor \FQ \to \FQ \tensor \bigwedge^2 \FQ 
\,,
\end{equation}
with $ \FQ = F^* \tensor_\Z \Q $, and
 $ M_n(F) $ a $ \Q $-vector space on symbols $ [x]_n $
for $ x $ in $ F \setminus \{ 0,1\} $, modulo non-explicit relations depending
on $ n $.  The maps in the complex are given by
\begin{equation}
  \label{complex-maps}
  \begin{split}
    \dd [x]_3 &= [x]_2 \tensor x \\
     \dd [x]_2 \tensor y &= (1-x) \tensor (x \wedge y)
  \end{split}
\end{equation}
There are maps
\begin{equation*}
     H^i(\scriptM 3 F ) \to \K 6-i 3 F 
\end{equation*}
for $ i=2,3 $, and for $ i=3 $ this is an isomorphism.
Quotienting out by a suitable subcomplex (see Section~\ref{M3Fcomplexes})
one obtains the complex
\begin{equation} \label{eq:tildedF}
  \tildescriptM 3 {F} :\; \tM_3(F) \to \tM_2(F) \tensor \FQ \to  \bigwedge^3 \FQ 
\,,
\end{equation}
which is quasi-isomorphic to $ \scriptM 3 F $ in degrees~2 and~3.
Its shape is more in line with conjectures (see, e.g., \cite[Conjecture~2.1]{gonXpam})
and it is easier to work with for explicit examples.

We can apply this with $ F $ the function field $ K(C) $ of $ C $, but
as the syntomic regulator needs some information over the residue
field, we have to use an analogous complex.

\begin{notation}\label{case1}
  For the curve $ \CC $ as above with generic fiber $C/K$, we let
$ \O \subset F $ be the local ring consisting of functions that are generically defined on the special fiber $ \CC_\kappa $.
\end{notation}

In Section~\ref{M3O} we shall construct a complex
\begin{equation}\label{integralcomplex}
\scriptM 3 {\O} :\; M_3(\O) \to M_2(\O) \tensor \OQ \to \OQ \tensor \bigwedge^2 \OQ 
\,,
\end{equation}
with in this case 
$ M_n(\O) $ a $ \Q $-vector space on symbols $ [u]_n $
for $ u $ in $ \SPO $, the {\it special units \/} of $ \O $, namely
those $ u $ in $ \O^* $ for which $ 1-u $ is also in $ \O^* $,
again modulo non-explicit relations depending
on $ n $, and $ \OQ = \O^* \otimes_\Z \Q $.
The maps in the complex are given by~\eqref{complex-maps} as before.
In fact, one may view $ M_2(\O) \subseteq M_2(F) $; see Remark~\ref{m2includes}.
The complex comes with maps
\begin{equation}\label{kmapsi}
  H^i( \scriptM 3 {\O} ) \to \K 6-i 3 {\O} 
\end{equation}
for $ i=2 $ and 3.

Similar constructions, satisfying in particular~\eqref{kmapsi} can be
made in the following situation.

\begin{notation}\label{case2}
  Suppose $ k \subset K $ is a number field and let $ R' $ be the local ring $ R \cap k $.
  For $ \CC' $ a smooth, proper, geometrically irreducible
  curve over
  $ R' $, let $ \O' $ denote the local ring of rational functions on $ \CC' $ that
  are generically defined on the special fiber above the maximal ideal of $ R' $.
\end{notation}

In this case one has an additional map
\begin{equation*}
   M_2 (\O') \tensor_\Q \OpQ \xrightarrow{\partial_1}  \coprod_x \tM_2 (k(x))
\end{equation*}
where $ M_2(\O') $ is now a $ \Q $-vector space on symbols $ [u]_2 $
with $ u $ in $ \O'^* $ such that $ 1-u $ is also in $ \O'^* $,
the coproduct is over 
all closed points of the generic fibre $ C' = \CC'\otimes_{R'}k $, given by
$$ \partial_{1,x} ([g]_2 \tensor f) = \ord_x(f) \cdot [g(x)]_2 \,, $$
with the convention that $ [0]_2 = [1]_2 = [\infty]_2 = 0 $.

To explain the terms in which the formula for the regulator will be
expressed, we need to introduce Coleman integration theory (see Section~\ref{sec:coleman}).
Coleman~\cite{Col82,Col-de88} defined an integration theory on curves
over $\Cp$ with good reduction and on certain rigid analytic subdomains of these, which
he termed ``Wide open spaces''. One
first needs to choose a branch of the $p$-adic logarithm, i.e., a group
homomorphism 
$ \log : \C_p^* \to \Cp $, such that around $ z=1 $, it is given
by the usual power series expansions for $ \log(1+y) $ (this
amounts to specifying $ \log(p) $ in $ \Cp $). Once this is done, the
theory includes single valued iterated integrals on the appropriate
domain, called ``Coleman functions''. In particular, we have the
following functions
\begin{equation} \label{dilogfuncs}
\begin{split}
   \Li_2(z) &= - \int_0^z \log(1-x) \dd \log x\\
   \Ltwo(z)&=\Li_2(z)+\log(z)\log(1-z)\\
    \Lmod 2 (z) &= \Li_2(z)+ \frac12 \log(z)\log(1-z)
\end{split}
\end{equation}

The function $ \Li_2(z) $ is defined on $ \Cp\setminus\{1\} $;
see the beginning of Section~VI in~\cite{Col82}. Consequently, all 3
functions are defined everywhere except possibly $0,1,\infty$. They,
and other Coleman functions, can be assigned a value at these points
as follows.

For every point $y\in \PP^1(\Cp) $, the residue disc $U_y$ is the
collection of points reducing to the same point as $y$. For each such
$y$, and in terms of a local parameter $ z = z_y $ on $U_y$, a Coleman
function $G$ can be expanded as $G(z) = \sum_i f_i(z) \log^i(z) $, 
where all $ f_i(z) $ are in $ \Cp[[z,z^{-1}]] $. 
We define the \emph{constant term} $ c_z(G) $
at $ y $ with respect to the parameter $ z $ as the constant term of $ f_0 $;
see Definition~\ref{constanttermdef}.   In general the constant
term will depend on the choice of the local parameter $ z $, but 
there are many Coleman functions for which the constant term is
independent of this choice. In such a case we will write $G(y)$ for
this constant term. 
In particular, this is the case at all points for $ \Lmod 2 (z) $
and $\int \Ltwo(g) \omega $ for any rational function $g$, as
well as for, for $\Li_2(z) $ and $\Ltwo(z)$ at all points except
$\infty$ 
(see Lemmas~\ref{L2const} and~\ref{czindependence} as well as Corollary~\ref{lmodindep}).
We further define all three
functions from~\eqref{dilogfuncs} to be $0$ at $0$ and $\infty$ (this is
the constant term with respect to the standard parameter).
For any Coleman function $G$, which is the integral of a form $\eta$,
and divisor $D=\sum n_i y_i$ we will
define
\begin{equation*}
  G(D) = \int_D \eta := \sum n_i G(y_i)
\end{equation*}
where we will be assuming that either $G$ is defined at each $y_i$, or
its constant term there is independent of the parameter.

We note that  $ \Lmod 2 (z)+ \Lmod 2 (z^{-1}) = 0 $ for $ z $ in 
$ \Cp\setminus\{0,1\} $~\cite[Proposition~6.4(ii)]{Col82}, and that this
extends to all values using constant terms. Similarly we have
$ \Ltwo(z)+\Ltwo(z^{-1}) = \frac12 \log^2(z) $.

We shall state the theorems in the introduction
in a way that allows comparison with similar results in the classical case over $ \C $.
The formulae in that case can be easily
transformed by using Stokes theorem, whereas it seems the formulae
in the syntomic case are not as flexible. Consequently, in the
syntomic case we have to state a larger number of theorems.
In order to enable a comparison in Remark~\ref{compremark}
of the syntomic formulae below
(especially those in Theorems~\ref{main-thm3} and~\ref{main-thm4})
with those in the classical case, we recall and reformulate some
of the classical results in Section~\ref{sec:classical}.

We are now ready to state the first main theorem.
In it, and the remaining theorems in the introduction, we evaluate
Coleman functions at closed points of $ C $ by working over 
a finite extension of $ K $ over which all such points are rational.
The result will be in $ K $ by Galois equivariance of Coleman integration.

\begin{theorem} \label{main-thm1}
Suppose, in the situation of Notation~\ref{case1}, that $\omega$ is a
holomorphic form on $C$.
\begin{enumerate}
\item The assignment
  \begin{equation*}
    [g]_2 \tensor f \mapsto 2 \int_{(f)} \Ltwo(g)\omega
  \end{equation*}
  gives a well-defined map 
 $ \regmapa: M_2(\O) \tensor \OQ \to K$, and this induces a map
   $\regmapa: H^2( \scriptM 3 {\O} )\to K $.
\item
  Suppose $ k \subset K $ is a number field, and 
$ \CC' $ is a smooth, proper, geometrically connected
curve over the local ring $ R' = R \cap k $. Let $\O'$ be as in Notation~\ref{case2}
and put $\CC = \CC' \otimes_{R'} R$.
 Let $ \alpha'$ in $ H^2(\scriptM 3 {\O'} ) $ be such that
 $\partial_1(\alpha')  =0$. Then there exists a unique $\beta'$ in
$ \K 4 3 {\CC'} $ whose image in $ \K 4 3 {\O'} $ under localization
  equals the image of
 $\alpha'$ under~\eqref{kmapsi} modulo $ \K 3 2 k \cup \O_{\Q}^{\prime\ast} $.
 If $\beta$ is the image of $\beta'$ under $ \K 4
 3 {\CC'} \to  \K 4 3 {\CC} $, then we have
 \begin{equation*}
   \reg(\beta)\cup \omega = \regmapa(\alpha)
 \end{equation*}
where $\alpha$ is the image of $\alpha'$ in $  H^2(\scriptM 3 {\O} ) $.
\end{enumerate}
\end{theorem}

\begin{remark}
The reader should compare
the above formula for the regulator with the
formula obtained by Coleman and de Shalit~\cite{Col-de88}, which
is known to be the syntomic regulator by~\cite{Bes98b}. There, the
regulator is obtained by sending the symbol $\{ f,g\} $ in $ K_2(F) $
to $\int_{(f)}\log(g) \omega $. The similarity with the present
formula should be clear.
\end{remark}

The rest of our results concern the $ K $-theory of open curves over $R$ and
not over a number field. Thus, they are more general on the one hand,
but progressively harder to state. Indeed, the first theorem is
special because we are able to simplify matters by taking account of
boundary terms over number fields.

As we are now computing on an open scheme, we no longer have a
non-trivial cup product pairing, so we first need to explain what it
is that we are computing.
Under the regulator, each element of $ \K 4 3 {\O} $  maps to $
\Hdr^1(U/K) $ for
some wide open space $ U $ in $ C $ in the terminology of Coleman.
There exists a canonical projection $ \Hdr^1(U/K) \to \Hdr^1(C/K) $,
compatible with restriction to a smaller $ U $;
see~\cite[Proposition~4.8]{Bes98b} and~\eqref{eq:projform} below.  We still denote the
composition $$ \K 4 3 {\O} \to \Hdr^1(U/K) \to \Hdr^1(C/k) $$
by $ \reg $.

\begin{theorem} \label{main-thm2}
Suppose $\omega$ is a holomorphic form on $C$.
The assignment
  \begin{equation*}
    [g]_2 \tensor f \mapsto 2 \int_{(f)} \Ltwo(g) \omega - 2 \sum_y \ord_y (f) F_\omega(y) \Lmod 2  (g(y)) 
  \,,
  \end{equation*}
  where in the sum $ y $ runs through the closed points of $ C $,
  gives a well-defined map $ \regmapb: M_2(\O) \tensor \OQ \to K$.
  It induces a map   $\regmapb: H^2( \scriptM 3  {\O} )\to K $,
  which coincides with the composition
$$
\xymatrix{
H^2(\scriptM 3 {\O} ) \ar[r] & \K 4 3 {\O} \ar[r]^-{\reg} & \Hdr^1(C/K) \ar[r]^-{\cup \omega} &  K.
}
$$
\end{theorem}

Over the
complex numbers it is known that computing the cup product
of the regulator with holomorphic forms suffices to describe it
completely in the case we are considering because those linear
maps surject onto the dual of the target space of the regulator
(see the beginning of Section~4 of \cite{Jeu96}, especially
Proposition~4.1). This is not true over the $p$-adics.
It is therefore important to have 
formulas for the cup product of the regulator with general cohomology
class (such a class can be represented by a form of the second kind on $
C $, i.e., a 
meromorphic form all of whose residues are $0$).
This can be done at the cost of introducing further machinary -
the notion of the triple
index. It is a generalization of the ``local index'' which was
introduced in~\cite[Section~4]{Bes98b}.

Informally speaking, working on an annulus 
$ e $ over $\C_p$, $e\isom \{r< |z|< 1\}$,
the triple index associates to the integrals $F$, $G$ and $H$ of
three rigid analytic 1-forms on $e$ (in this case these forms are simply
Laurent series converging on $e$ multiplied by $dz$)
together
with choices of integrals for $ F \dd G $, $ F \dd H $ and $ G \dd F $,
a number $\pair{F,G;H}_e $ in $ \C_p$ that is supposed to be a
generalization of $\res_e FGdH$. Note that the integrals appearing in
the data for the triple index make perfect sense once one admits a log
function to correspond to the integral of $dz/z$, and are determined
up to a constant by the form they integrate. Suppose now that $C/\C_p$
is a curve with good reduction and that $C$ contains discs $D_i\isom
\{|z|< 1\}$. The rigid analytic domain $U=C \setminus \cup_i
(D_i-e_i)$ where $e_i\subset D_i$ is the annulus 
corresponding to $\{r< |z|< 1\}$ is called a wide open space by
Coleman. The $e_i\subset U$ are called the annuli ends of $U$.
Suppose that $F$, $G$ and $H$ are Coleman functions defined on $U$
such that restricted to the $e_i$'s they are of the type allowing us
to compute the triple indices
$\pair{F|_{e_i},G|_{e_i};H|_{e_i}}_{e_i}$. We may use auxiliary
data composed of 
Coleman integrals restricted to $e_i$ for computing these. It
sometimes turns out that the
sum of triple indices over all the $e_i$ depends only on $F$, $G$, and
$H$ and not on the auxiliary data. This applies in particular to the
sum of triple indices in the two theorems below. It is further known that
this sum of triple indices behaves well with respect to shrinking the
wide open space $U$. Finally, if everything is defined over a complete
subfield $K$ of $\C_p$ then this sum of triple indices is in $K$.

\begin{theorem} \label{main-thm3}
  Let $ \omega $ be a form of the second kind on $ C $.
  The assignment
  \begin{equation*}
    [g]_2 \tensor f \mapsto \sum_e
    \pair{\log(f),\log(g);\int F_\omega \dlog(1-g)}_e
  \end{equation*}
  where $F_\omega$ is any Coleman integral of $\omega$ and the sum of
  triple indices is over all annuli ends $e$ of a wide open space $U$ on which
  all $f$, $g$ and $ 1-g $ are invertible, and $\omega$ is holomorphic,
  gives a well-defined map $
  \regmapc: M_2(\O) \tensor \OQ \to K$. It induces a map
  $\regmapc: H^2( \scriptM 3  {\O} )\to K $, which coincides with
  the composition
$$
\xymatrix{
H^2(\scriptM 3 {\O} ) \ar[r] & \K 4 3 {\O} \ar[r]^-{\reg} & \Hdr^1(C/K) \ar[r]^-{\cup \omega} &  K
\,.
}
$$
\end{theorem}

The complex $\tildescriptM {3} {F} $ defined in~\eqref{eq:tildedF} is easier to work with in
explicit computations then the complex $\scriptM {3} {F} $. Therefore, just as in \cite[Remark~4.5]{Jeu96}
it is desirable to
have a formula for the regulator using this complex. With that in
mind, we define in Section~\ref{tM2O} a complex
\begin{equation*}
 \tildescriptM 3 {\O} : \tM_3(\O) \to \tM_2(\O) \otimes \OQ \to \bigwedge^3\OQ
\end{equation*}
such that its cohomology in degrees 2 and 3 is isomorphic to that of the
complex $ \scriptM 3 {\O} $ in~\eqref{integralcomplex}.
Corresponding to the statements in Theorems~\ref{main-thm2} and~\ref{main-thm3}
for $ \scriptM 3 {\O} $,
we have the following two expressions for the regulator in this case.

\begin{theorem} \label{main-thm4}
1.  Let $ \omega $ be a form of the second kind on $ C $. The assignment
  \begin{align*}
    [g]_2 \tensor f & \mapsto 
    \frac{2}{3} \sum_e \pair{\log(f),\log(g);\int F_\omega  \dlog(1-g)}_e
\\ 
&\qquad
  - \frac{2}{3} \sum_e \pair{\log(f),\log(1-g);\int F_\omega \dlog(g)}_e 
  \end{align*}
  gives a well-defined map $ \regmapd: \tM_2(\O)\otimes  \OQ \to K$. 
  It induces a map
  $ \regmapd : H^2(\tildescriptM 3 {\O} ) \to K$, which coincides with the 
  composition of maps
  \begin{equation*}
  \xymatrix{
   H^2(\tildescriptM 3 {\O} ) \ar[r]^-{\simeq} &  H^2(\scriptM 3 {\O} ) \ar[r] & \K 4 3 {\O} \ar[r]^-{\reg} & \Hdr^1(C/K)  \ar[r]^-{\cup \omega} &  K
  }
  \end{equation*}
  with the leftmost map being the isomorphism alluded to before.\\

2.  If $\omega$ is a holomorphic form on $ C $, then the same holds
  for the assignment
   \begin{equation*}
      [g]_2 \tensor f \mapsto 
      \frac{2}{3}  \left(\int_{(1-g)} \log(g) F_\omega \dlog
        (f)
      -\int_{(g)} \log(1-g) F_\omega \dlog (f)\right)
\,.
   \end{equation*}
\end{theorem}

\begin{remark}
The careful reader will notice that the last formula above does not
make sense as written, because when $g(y)=\infty$ we also have
$1-g(y)=\infty$ so the integral is singular at the point of the
divisor where it is evaluated. This can be resolved either by using
constant terms or by evaluating at such a point the
difference
$\int \log(g) F_\omega \dlog(f) - \int \log(1-g) F_\omega \dlog(f) $, which does make sense.
\end{remark}

A key complex for doing computations is the complex
\begin{equation*}
\Ccomp \bullet \O : \Ccomp 1 \O \to \Ccomp 2 \O
\end{equation*}
in cohomological degrees 1 and 2,
which we will construct in Section~\ref{CcompOconstruction}. The
theorems in this introduction admit analogous results expressed in
terms of this complex. We avoided these results for clarity in the
introduction. 
 However, they are very useful in
applications since it is easier to find explicit examples to
which these results apply, e.g., for
certain elliptic curves; see \cite[Section~6]{Jeu96}.

We end the introduction with a conjecture. The regulator formulae
that we obtain do not depend on any integrality assumptions. This is
only required because the syntomic regulator is a map from the
$ K $-theory of an integral model. Thus we conjecture the following.

\begin{conjecture}
Theorems~\ref{main-thm1}, \ref{main-thm2}, \ref{main-thm3} and~\ref{main-thm4} hold, with the same formulas, with $\O$
replaced by $F$ and $\CC$ replaced by $C$.
\end{conjecture}

The authors would like to thank the European
Community for support through the RTN network \emph{Arithmetic Algebraic
Geometry}, which enabled them to meet on various occasions
during the long gestation period of this paper.
Rob de Jeu would like to thank the Newton Institute in Cambridge,
where part of this paper was written, for a conducive
atmosphere.

This paper is dedicated to the memory of Jon Rogawski.
One of us (AB) still remembers Jon's help and advise as a young postdoc 
at UCLA. He was a great mentor with his calm and assured guidance. He will be greatly missed.

\begin{notation} \label{intronotation}
Unless stated otherwise, throughout the paper, we will be working with the following notation.

$ K $ will be a discrete valuation field of characteristic zero,
with valuation ring $ R $, and residue field $ \kk $ of positive
characteristic $ p $.  We shall assume that $ \kk $ is a subfield
of $ \Fpb $. In various places, $ k $ will be a number field inside $ K $.
In that case we denote by $ \kres \subseteq \kappa $ the residue field of
the local ring $ R'= k \cap R $.

$ \CC $ will be a smooth, proper, geometrically irreducible curve
over $ R $.
The generic fiber is denoted $ C $, the special fiber is denoted $ \Csp $.
We let $ F = K(C) $, and $ \O \subset F $ will be the valuation
ring for the valuation on $ F $ corresponding to the generic point of $ \Csp $,
which consists of those elements in $ F $ that are generically
defined on $ \Csp $.

If $ k \subset K $ is a number field, and
$ \CC' $ is a smooth, proper, geometrically irreducible curve
over $ R' = R \cap k $, then the generic fiber is denoted $ C' $,
the special fiber is denoted $ \Cspnf $.
We let $ F' = k(C') $, and $ \O' \subset F' $ will be the valuation
ring for the valuation on $ F' $ corresponding to the generic
point of $ \Cspnf $. In particular, if $ \CC = \CC'\otimes_{R'}R $,
then $ \O' = \O \cap F' $.

If $ S $ is a subset of a group, then we denote by $ {<} S {>} $
the subgroup generated by $ S $, and if $ S $ is a subset of
a $ \Q $-vector space, we denote by $ {<} S {>}_\Q $ the $ \Q $-vector
subspace generated by $ S $.

All tensor products will be over $ \Q $, unless specified otherwise.

\end{notation}

For the convenience of the reader, we give a commutative diagram,
which plays the role of \lq\lq{\it Leitteppich\/}\rq\rq\ for the proofs in this
paper. In the left lower square we may also use $ \O' $ instead
of $ \O $, in which case $ C = C' \otimes_{R'} K $.

\begin{equation} \label{MAP}
\begin{split}
\xymatrix{
 H^2(\scriptM 3 {\CC'} ) \ar[r]\ar[d]
&
 \K 4 3 {\CC'} \oplus \K 3 2 k \cup \O_\Q^{\prime*} \ar[d]
&
\\
 H^2(\scriptM 3 {\O} ) \ar[r]\ar[d]
&
\K 4 3 {\O} \ar[r]^-{\reg} \ar[d]
&
 \Hdr^1(C)\ar[d]^-{ \cdot \cup\omega }
\\
 H^1(\Ccomp \bullet {\O} ) \ar[r]
&
\K 4 3 {\O} / \K 3 2 {\O} \cup \OQ \ar@{.>}[r]
&
 K
}
\end{split}
\end{equation}

The constructions in algebraic $ K $-theory will be carried out
in Section~\ref{k-theory}.
The top left square comes from the natural map
$ \scriptM 3 {\CC'} \to \scriptM 3 {\O'} $ (see Section~\ref{M3CCconstruction}),
and is justified by~\eqref{FOCcd},
whereas the bottom left square is~\eqref{MCO}.
For $ \omega $ in $ \Hdr^1(C) $ the map
\[
\xymatrix{
\K 4 3 {\O} \ar[r]^-{\reg} & \Hdr^1(C)
}
\]
factorizes through the quotient map
$ \K 4 3 {\O} \to \K 4 3 {\O} / \K 3 2 {\O} \cup \OQ $
(see Corollary~\ref{factorcor}).
The resulting composition in the bottom line of~\eqref{MAP}
is then computed in Section~\ref{sec:begin}, using the techniques developed in
the preceeding sections. 
In Section~\ref{sec:end} we then finish the proofs of the theorems above,
based on this calculation.

\section{$ K $-theory} \label{k-theory}

\subsection{Introduction}

Consider a proper, smooth, geometrically irreducible curve $ \CC $
over $ R $ as in Notation~\ref{case1},
or $ \CC' $ over $ R' $ as in Notation~\ref{case2}.
We shall construct various cohomological complexes whose cohomologies
are related to that of $ F $, $ \O $, $ F' $ or $ \O' $.
The main idea is the same as in \cite{Jeu96}, but the fact that
we shall be working with a discrete valuation ring rather
than a field gives rise to some complications.  In order to highlight
the idea we start with a more gentle exposition.  For the proofs of
the statements that are used in the construction, we refer the
reader to \cite{Jeu95}, especially Sections~2.1 through~2.3, and~3.
There most of the work was done over $ \Q $,
but in fact the proofs hold over our base $ \O $ , a discrete
valuation ring of characteristic zero, 
without any change.

It will be clear from the constructions that the complexes are
natural in terms of $ F $, $ F' $, $ \O $ and $ \O' $, which
we shall use later in this paper. In particular,
if we start with $ \CC' $ over $ R' $ and let $ \CC  = \CC' \otimes_{R'} R $,
then there are natural
maps from the complexes for $ F' $ to those for $ F $, and from
those for $ \O' $ to those for $ \O $.

If $ B $ is a Noetherian scheme of finite Krull dimension $ d $, then
according to~\cite[Proposition~5]{Sou85},
one can write 
\begin{equation}
\label{weightdecomposition}
K_n(B) \tensor_\Z \Q = \oplus_{j=\min{2,n}}^{n+d} \K n j B
\end{equation}
where $ \K n j B $ consists of all $ \alpha $ in $ K_n(B) \tensor_\Z \Q $
such that $ \psi^k(\alpha) = k^j \alpha $ for all Adams operators $ \psi^k $.
(The regularity assumption at the beginning of Section~4 of loc.~cit.\
is not necessary, see~\cite[Proposition~8]{Gil-Sou99}.)
If in addition $ B $ is separated and regular, then the pullback
$ K_*(B) \to K_*({\mathbb A}_B^1) $ 
is an isomorphism, see \cite[\S7]{Qui67}.  
The weight behaves naturally
with respect to pullback, also giving us
$ \K m j B \iso \K m j {\mathbb A}_B^1  $ under pullback. 
And under suitable hypotheses for a
closed embedding, there
is a pushforward Gysin map with a shift in weights corresponding
to the codimension (see, e.g., \cite[Proposition~2.3]{Jeu95}).

Let $ X_B = \PP_B^1\setminus\{t=1\} $ with $ t $ the standard
affine coordinate on $ \PP_B^1 $.  Write $ \bbox_B^1 $ for the
closed subset $ \{t=0,\infty\} $ in $ \PP_B^1 $.  Then the relative
exact sequence for the couple $ (X_B; \bbox_B^1) $
gives us
$$
\dots\to
K_{n+1}(X_B) \to K_{n+1}(\bbox_B^1) \to K_{n}(X_B;\bbox_B^1) \to K_{n}(X_B) \to K_{n}(\bbox_B^1)
\to\cdots
$$
for $ n \geq 0 $.  Because the map pullback $ K_{n+1}(B) \to K_{n+1}(X_B) $
is an isomorphism, combining it with the pullback 
$  K_{n+1}(X_B) \to K_{n+1}(\bbox_B^1) = K_{n+1}(B)^2 $
shows that the map $ K_{n+1}(X_B) \to K_{n+1}(\bbox_B^1) $ corresponds
to the diagonal embedding $ K_{n+1}(B) \to K_{n+1}(B)^2 $.  As this holds
for all $ n \geq 0 $, we get that we have an isomorphism
$ K_{n}(X_B;\bbox_B^1) \iso K_{n+1}(B) $ for $ n \geq 0 $.  Note
that we have a choice of sign here in the isomorphism of the
cokernel of $ K_n(B) \to K_n(B)^2 $ with $ K_n(B) $.  This results
in similar choices of signs in the maps $ H^i(\scriptM n \O ) \to \K 2n-i n {\O} $
(resp. $ H^i(\tildescriptM n \O) \to \K 2n-i n {\O} $)
later on in this section.

We will have to go up one level in the relativity.
If we let $ \bbox_B^2 $ be shorthand for
$$
 \{t_1=0,\infty\};\{t_2=0,\infty\}
\,,
$$
then we can get a long exact sequence
\begin{align*}
& \cdots \to
K_{n+1}(X_B^2;\{t_1=0,\infty\}) \to  K_{n+1}(\{t_2=0,\infty\};   \{t_1=0,\infty\} )
\to
\cr
& 
\to K_{n}(X_B^2;\bbox_B^2) 
\to
K_{n}(X_B^2;\{t_1=0,\infty\})
\to K_{n}(\{t_2=0,\infty\} ;  \{t_1=0,\infty\} )
\to\cdots
.
\end{align*}
The composition
\begin{align*}
&
K_{n+1}(X_B;\{t_1=0,\infty\}) \rightiso K_{n+1}(X_B^2;\{t_1=0,\infty\})
\to 
\cr
& \qquad\qquad \to
K_{n+1}(\{t_2=0,\infty\};   \{t_1=0,\infty\} ) \iso K_{n+1}(X_B;\{t_1=0,\infty\})^2
\end{align*}
(with the first the pullback along the projection $ (t_1,t_2) \mapsto t_2 $)
is the diagonal embedding, hence we obtain an isomorphism
$ K_{n}(X_B^2;\bbox_B^2) \iso K_{n+1}(X_B;\bbox_B^1) $
for $ n \geq 0 $.
Therefore we get
$ K_n(X_B^2;\bbox_B^2) \iso K_{n+1}(X_B;\bbox_B)  \iso K_{n+2}(B)  $ 
for $ n \geq 0 $.
A similar argument with weights gives us an isomorphism
$ \K n j X_B^2;\bbox_B^2  \iso \K n+2 j B  $ for $ n \geq 0 $.

In order to get elements in $ K_{n+2}(X_B^2;\bbox_B^2) $,
we use localization sequences.  We first explain the idea for
$ K_{n+1}(X_B;\bbox_B) $, because for $ K_{n+2}(X_B^2;\bbox_B^2) $
the process involves a spectral sequence.
If $ u $ is an element in our discrete valuation ring
$ \O $ such that both $ u $ and $ 1-u $
are units, then we get an exact localization sequence
$$
\dots\to
K_{m}(\O) \to
K_m(X_\O;\bbox_{\O}^1)
\to K_m(X_{\O,\loc} ;\bbox_\O^1 ) \to K_{m-1}(\O) 
\to\cdots
$$
where $ X_{\O,\loc}
 = X_\O \setminus \{ t=u \} $ and we identified
$ \{ t=u \} \subset X_\O $ with $ \O $ (or rather $ \spec(\O) $).
We used here that $ u $ and $ 1-u $ are units in $ \O $ so that
$ \{ t = u \} $ does not meet $ \bbox_\O^1 $ or $ \{ t = 1 \} $,
and that $ \O $ is regular in order to identify $ K_m(\O) $ with
$ K_m'(\O) $.
(If we want to leave out $ \{ t = u \} $ and $ \{ t = v \} $ simultaneously
for two distinct elements $ u $ and $ v $ in $ \O $ such that all of
$ u $, $ v $, $ 1-u $ and $ 1-v $ are units, which we shall do
below, this already becomes far more complicated and one is force
to use a spectral sequence.)
The image of $ K_2(\O) \to K_2(X_{\O};\bbox_{\O}^1) $ can be controlled by
looking at the weights, which for the bit that we are interested in
gives us
$$
\dots\to
\K 2 1 {\O} \to
\K 2 2 {X_\O;\bbox_\O^1}
 \to \K 2 2 {X_{\O,\loc};\bbox_\O^1}  \to \K 1 1 {\O} 
\to\cdots
\,,
$$
so that $ \K 3 2 {\O} \iso \Ker\left(  \K 2 2 {X_{\O,\loc};\bbox_\O^1}  \to \K 1 1 {\O} \right) $.
Because of weights in $ K $-theory, one knows that $ \K 2 1 {\O} = 0 $,
so we can analyze $ \K 2 2 {X_\O;\bbox_\O^1} $
 as subgroup of
$ \K 2 2 {X_{\O,\loc};\bbox_\O^1} $.  In \cite[Section~3.2]{Jeu95}
universal elements $ [S]_n $ were constructed, of which we want
to use $ [S]_2 $ here.  It gives rise to an element $ [u]_2 $
in $  \K 2 2 {X_{\O,\loc};\bbox_\O^1} $ with boundary $ (1-u)^{-1} $
in $ \K 1 1 {\O} $.  If we use this for various $ u $ (suitably
modifiying the localization sequence above into a spectral sequence)
and also consider elements coming from the cup product 
$$
\K 1 1 {X_{\O,\loc};\bbox_\O^1} \times \K 1 1 {\O} \to \K 2 2 {X_{\O,\loc};\bbox_\O^1} 
$$
we can get part of $ \K 2 2 {X_{\O};\bbox_\O^1} \iso \K 3 2 {\O} $
by intersecting the kernel of the map corresponding to
$ \K 2 2 {X_{\O,\loc};\bbox_\O^1} \to \K 1 1 {\O} $
with the space generated by the symbols $ [u]_2 $ and the image
$  \K 1 1 {X_{\O,\loc};\bbox_\O^1} \cup \K 1 1 {\O}  $ of the cup product.

\subsection{Preliminary material.}
\label{prelimk}
We describe some basic facts about
the various $ K $-groups of $ F $, $ \O $, $ C $ and $ \CC $,
or $ F' $, $ \O' $, $ C' $ and $ \CC' $,
including those mentioned in the introduction. The two cases
are very similar so we shall treat them together.

We shall first consider the case where
$ F = k(C') $ for a smooth, projective curve $ \CC' $ over $ R' $ with
geometrically irreducible generic fiber $ C' $.
Let $ \Cspnf $ be the special fibre of $ \CC' $, which is a smooth,
projective curve over the finite field $ \kres $.
Because $ \Cspnf $ is regular, there is an exact localization sequence
\begin{equation}
\label{kloc}
\xymatrix{
\dots \ar[r] 
&
\K 4 2 {\Cspnfff} \ar[r]
&
\K 4 3 \O' \ar[r]
&
\K 4 3 F' \ar[r]
&
\K 3 2 {\Cspnfff} \ar[r]
&
\dots
}
\,.
\end{equation}
By \cite[Korollar~2.3.2]{Har77}),  $ K_n(L) $ is torsion for $ n \geq 2 $ 
for all function fields $ L $ of curves over finite fields, so in particular,
$ \K 4 3 {\O'} \rightiso \K 4 3 F' $.
If $ F=K(C) $, then we get
\[
\xymatrix{
\dots \ar[r] &
\K 4 2 {\Cspff} \ar[r] &
\K 4 3 {\O} \ar[r] &
\K 3 2 {F} \ar[r] &
\K 3 2 {\Cspff} \ar[r]
&
\dots
}
\]
By our assumptions (see Notation~\ref{intronotation}),
$ \kk \subseteq \Fpb $. 
According to \cite[Proposition~2.2]{Qui67} or \cite[Lemma~5.9]{Sri96}),
$ K_n(\kappa(\CC_\kappa)) $ is the direct limit of $ K_n $ of function fields of curves
over finite fields, hence is torsion as well, and we find $ \K 4 3 {\O} \iso \K 4 3 F $.

From the exact localization sequence
\[
\xymatrix{
\dots \ar[r]
&
\coprod_{x \in \Cspnf^{(1)}} \K n 1 \kres(x) \ar[r]
&
\K n 2 {\Cspnf} \ar[r]
&
\K n 2 {\Cspnfff} \ar[r]
&
\dots
}
\]
and the fact that $ \K n 1 L $ is zero for any field $ L $
for $ n \geq 2 $, we see that $ \K n 2 {\Cspnf} $ is trivial for $ n \geq 2 $.
From the exact localization sequence
\begin{equation*}
\xymatrix{
\dots \ar[r] 
&
\K 4 2 {\Cspnf} \ar[r]
&
\K 4 3 {\CC'} \ar[r]
&
\K 4 3 C' \ar[r]
&
\K 3 2 {\Cspnf} \ar[r]
&
\dots
}
\end{equation*}
we see that $ \K n 2 {\Cspnf} $ is trivial for $ n \geq 2 $,
hence $ \K 4 3 \CC' \simeq \K 4 3 C' $.
Using a direct limit argument as before,  we then see that
$ \K 4 3 {\CC} \iso \K 4 3 C $ as well.

\begin{remark}
\label{K4CtoK4Finjective}
We now have two identifications fitting into a commutative diagram
\begin{equation*}
\xymatrix
{
\K 4 3 \CC' \ar[r]\ar@{=}[d]
&
\K 4 3 {\O'} \ar@{=}[d]
\\
\K 4 3 C' \ar[r]
&
\K 4 3 F'
\,,
}
\end{equation*}
and similarly for $ F $, $ \O $, $ \CC $ and $ C $.
From the exact localization sequence
\begin{equation*}
\xymatrix
@C=11pt{
\dots \ar[r]
&
\coprod_{x \in C'^{(1)}} \K 4 2 k(x) \ar[r]
&
\K 4 3 C' \ar[r]
&
\K 4 3 F' \ar[r]^-{\partial}
&
\coprod_{x \in C'^{(1)}} \K 3 2 k(x) \ar[r]
&
\dots
}
\end{equation*}
we see that the map $ \K 4 3 F' \to \K 4 3 C' $ is injective
because $ \K 4 2 L = 0 $ for any number field $ L $.
Hence the map $ \K 4 3 {\CC'} \to \K 4 3 \O' $ is also injective.
\end{remark}

\begin{remark} \label{directsum}
We have $ \K 4 3 C' \oplus \K 3 2 k \cup \FpQ $ inside $ \K 4 3 F' $.
(This makes sense because $ \FpQ = \K 1 1 F' $.)
Namely, $ \K 4 3 C = \Ker(\partial) $ in the localization
sequence in Remark~\ref{K4CtoK4Finjective}.
On the other hand, for $ f $ in $ \FQ $ and $ \alpha $ in $ \K 3 2 k $,
$ \partial (\alpha \cup f) = \alpha \cup \textup{div}(f) $
in $ \coprod_{x \in C^{(1)}} k(x)_\Q^* $, hence this is trivial
only if $ f $ is in $ k_\Q^* $. But
$ \K 3 2 k \cup k_\Q^* \subseteq \K 4 3 k $, which
is zero as $ k $ is a number field.
Therefore $ \K 3 2 F \cup \FQ $
injects into $ \coprod_{x \in C^{(1)}} k(x)_\Q^* $ under $ \partial $.
\end{remark}

\begin{remark} \label{sameproducts}
Note that a local parameter of $ R' $ is also a local parameter
for $ \O' $, so $ F'^* $ is generated by $ \O'^* $ and that local
parameter. This implies that $ \K 3 2 k \cup \OpQ = \K 3 2 k \cup \FpQ $,
again because $ \K 3 2 k \cup k_\Q^* $ is trivial.
\end{remark}

We shall need the following result at several places later on.

\begin{prop} \label{milnork}
For a discrete valuation ring $ \O $, with residue field $ \kk $ and field of
fractions $ F $,  for all $ n \geq 1 $, the sequence
\[
\xymatrix{
\OQ^{\otimes n} \ar[r]
&
\K n n F \ar[r]
&
\K n-1 n-1 {\kk} \ar[r]
& 
0
}
\]
is exact.
\end{prop}

\begin{proof}
Since $ \K n n L \iso K_n^M(L)_\Q $
for any field $ L $ by~\cite[Th\'eor\`eme~2]{Sou85}, with
$ K_n^M(L) $ the Milnor $ K $-theory of $ L $,
it suffices to show that 
$ \O^{\otimes_\Z n} \to K_n^M(F) \to K_{n-1}^M(\kk) \to 0 $
is exact.
If $ \pi $ is a uniformizer of $ \O $, then $ K_n^M(F) $ is generated
by symbols 
$ \{ u_1,\dots,u_n \} $ and 
$ \{ u_1,\dots,u_{n-1}, \pi \} $, with all $ u_j $ in $ \O^*  $.
The map $ K_n^M(F) \to K_{n-1}^M(\kk) $ is the tame symbol, which
is trivial on the first type of generator, and maps the second
to $ \{ \bar u_1,\dots, \bar u_{n-1} \} $. It is clearly surjective.
So we only have to show that if $ \alpha $
in $ (\O^*)^{\otimes_\Z(n-1)} $ maps to the trivial element
under the composition $ (\O^*)^{\otimes_\Z (n-1)} \to (\kappa^*)^{\otimes_\Z (n-1)} \to K_{n-1}^M(\kappa) $,
then the image of $ \alpha\otimes\pi $ in $ K_n^M(F) $ 
is in the image of $ (\O^*)^{\otimes_\Z n} $.
Noticing that the Steinberg relations
$ \cdots \tensor x \tensor \cdots \tensor (1-x) \tensor \cdots $
in $ (\O^*)^{\tensor_\Z (n-1)} $ surject
onto those in $ (\kappa^*)^{\tensor_\Z (n-1)} $,
we see that we may assume that $ \alpha $ is in the kernel 
of the map $ (\O^*)^{\otimes_\Z (n-1)} \to (\kappa^*)^{\otimes_\Z (n-1)} $.
From the exact sequence
\[
 1 \to 1 + \O \pi \to \O^* \to \kappa \to 1
\]
and the fact that, if we have exact sequences
$  0 \to A_i \to B_i \to C_i \to 0 $ ($ i=1,\dots,m $) of Abelian groups, then
the kernel of $ B_1 \tensor_\Z \dots \tensor_\Z B_m \to C_1 \tensor_\Z \dots \tensor_\Z C_m $
is the image of
$ A_1 \tensor_\Z B_2 \tensor_\Z \dots \tensor_\Z B_m + B_1 \tensor_\Z A_2 \tensor_\Z B_3 \tensor_\Z \dots \tensor_\Z B_m + \dots $,
we see $ \alpha $ lies in the image of
$ (1+ \O \pi) \tensor_\Z \O^* \tensor_\Z \cdots \tensor_\Z \O^* + \O^* \tensor_\Z (1+\O \pi) \tensor_\Z \cdots \tensor_\Z \O^* + \cdots $.
But each element $ \{u_1,\dots,u_{n-1}, \pi\} $ with all
$ u_i $ in $ \O^* $ and at least one of them in $ 1 + \O \pi$
lies in the image of $ \O^{* \tensor_\Z n} $.  Namely, an element
in $ 1 + \O \pi $ is of the form $ 1-\pi^d u $ for some $ u $ in $ \O^* $,
$ d > 0 $. If $ d=1 $ we can rewrite $ \{\dots, 1-\pi u,\dots, \pi \} = -\{ \dots,1-\pi u,\dots, u \} $.
If $ d > 1 $, then using that $ \frac{1-\pi^d u}{1 - \pi} = 1 - \pi \frac{\pi^{d-1}u-1}{1-\pi} $,
we find that
$ \{\dots, 1-\pi^d u,\dots, \pi \} = \{ \dots, 1 - \pi \frac{\pi^{d-1}u-1}{1-\pi},\dots , \pi \} $,
which reduces to the case $ d = 1 $ as $ \frac{\pi^{d-1}u-1}{1-\pi} $
is in $ \O^* $.
\end{proof}

\begin{assumption} \label{sit}
Throughout the construction of the complexes in the various subsections below, we let $ F $
be a field of characteristic zero.  In the constructions for
complexes for $ \O $, $ \O $ will be a discrete valuation
ring $ \O $, with residue field $ \kk $ and field of fractions
$ F $, which we assume to be of characteristic zero.  We shall
always assume that $  | \kk | > 2 $, so that $ \SPO $ is non-empty and $ \langle \SPO \rangle = \O^* $.
\end{assumption}

\subsection{A few more preliminaries}

It will be convenient to introduce the notation $ \SPF = F^*\setminus\{1\} $,
as well as $ \SPO  = \{u \text{ in } \O^* \text{ such that } 1-u \text{ is in } \O^* \} $,
and $ \SPkk = \kk^*\setminus\{1\} $.

Throughout the remainder of Section~\ref{k-theory}, we shall let $ \Xfloc $
be the scheme obtained from $ X_F = \PP_F^1 \setminus \{ t = 1 \} $
by removing all points $ t = u $ with $ u $ in $ \SPF $.
We write $ X_{F}^{2,\loc} $ for $ (\Xfloc)^2 $.
Similarly, we let 
$ X_{\O} = \PP_{\O}^1 \setminus\{t=1\} $,
we write
$ \Xoloc $ for the scheme obtained from $ X_{\O} $ by removing
all subschemes $ t = u $ with $ u $ in $ \SPO $, and we write
$ X_{\O}^{2,\loc} $ for $ (\Xoloc)^2 $.
Finally, for $ \kk $, we let $ K_\kk = \PP_\kk^1\setminus\{t=1\} $,
we write $ X_{\kk}^\loc $ for the scheme obtained from $ X_{\kk} $ by removing
all subschemes $ t = u $ with $ u $ in $ \SPkk $, and we write
$ X_{\kk}^{2,\loc} $ for $ (X_\kk^\loc)^2 $.
(Of course, we would have to remove such a closed subscheme for
only a finite set of $ u $'s first, and then take a direct limit.
But by~\cite[Proposition~2.4]{Qui67} and some exact sequences
in relative $ K $-theory
this will give us the $ K $-theory of $ X_\kk^\loc $ anyway.
Moreover, as such a direct limit over finite subsets of $ \SPO $
or $ \SPF $ is clearly filtered, hence exact, this procedure
will commute with taking spectral sequences etc.\ below, so that
we work directly in the direct limit.)

Since writing $ \{t=0,\infty\} $ or $ \{t_1=0,\infty\}; \{t_2=0,\infty\} $
can be rather too long in places, we often abbreviate the first
by writing $ \bbox $, and the second by writing $ \bbox^2 $.

Let $ \I = \K 1 1 {\Xfloc;\bbox} $.
From the exact sequence
\[
\dots \to \K 2 1 {\bbox} \to \K 1 1 {\Xfloc;\bbox} \to \K 1 1 {\Xfloc} \to \K 1 1 {\bbox}
\to \dots
\]
we see that $ \I \subset \K 1 1 {\Xfloc} $ as $ \K 2 1 {\bbox} \iso \K 2 1 F ^ {\oplus 2} = 0 $.
So we can describe $ \I $ explicitly as those elements in $ \K 1 1 {\Xfloc} $
that restrict to 1 at $ t = 0 $ and $ t = \infty $.  
Because $ K_1(\Xfloc) $ is given by the units in the ring corresponding
to a localization of the affine line, we find that
\begin{equation*}
\I =
  \left\{
  \prod_j \left(\frac{t-u_j}{t-1}\right)^{n_j} \text{ with $ u_j $ in $ \SPF $ , $ n_j $ in $ \Z $, such that }
  \prod_j u_j^{n_j} =1
  \right\} \tensor_\Z \Q
\,.
\end{equation*}
Note that in particular the divisor map
\begin{equation}
\label{Iinjects}
\I \to \coprod_{t \in \SPF} \K 0 0 F
\end{equation}
is an injection.

Note that, if $ A $ is any $ \Q $-subspace of $ \K n l {\Xfloc;\bbox} $,
and we use the cup product $ \I \cup A \to \K n+1 l+1 {X_{F}^{2,\loc};\bbox^2} $
by pulling $ \I $ back along the first projection, and $ A $
along the second, then $ \dd (\I \cup A ) = (\dd \I) \cup A - \I \cup (\dd A) $,
and  $ \coprod_{t_1 \in \SPF} A / (\dd \I) \cup A  \iso A \tensor \FQ $
because $ \SPF $ generates $ F^* $, and the functions in $ \I $
(without $ \dots \tensor_\Z \Q $) give exactly the multiplicative
relations among the elements in $ \SPF $.
Of course, by reversing the role of the projections  we can do
this with $ t_2 $ instead of $ t_1 $ instead.
This will be used in order to change $ \coprod_{t \in \SPF} \dots $
into $ \dots \tensor_\Q \FQ $ in localization sequences or spectral sequences
below.

Under Assumption~\ref{sit}, 
we can do the same for $ \O $.
Namely, define $ \IO = \K 1 1 {\Xoloc;\bbox} $.
Because $ \K 2 1 {\O} = 0 $, and $ \K 1 1 {\O} = \OQ $, one see by exactly the same argument as for $ \I $
that
\begin{equation}
\label{IOdef}
\IO =
  \left\{
  \prod_j \left(\frac{t-u_j}{t-1}\right)^{n_j} \text{ with $ u_j $ in $ \SPO $ , $ n_j $ in $ \Z $, such that }
  \prod_j u_j^{n_j} =1
  \right\} \tensor_\Z \Q
\,.
\end{equation}
In particular, we have $ \IO \subseteq \I $ under localization
of the base from $ \O $ to $ F $.
Note that we used here that $ \IO $ gives us exactly the relations
needed  to turn $ \coprod_{t \in \SPO} \dots $ into $ \dots \tensor \OQ $,
as $ \IO $ (without $ \dots \tensor_\Z \Q $) gives the multiplicative relations among elements in $ \SPO $,
and $ \SPO $ generates $ \O^* $.

Finally, we like to mention that for $ x $ in $ F $,
under the map $ \K 0 0 F _{| t = x}  \to \K 0 1 {X_F;\bbox} \iso \FQ $,
$ 1 $ is mapped to $ x^{\pm 1} $, see~\cite[Lemma~3.14]{Jeu95}.
The same holds for $ \O $ instead of $ F $, and this is compatible with products.

\subsection{Construction of the complexes for $ F $ and $ C' $.} \label{Fsection}

Several parts of the constructions of the complexes 
in this section and in Section~\ref{Ocomplexes} below were
carried out in earlier papers \cite{Jeu95,Jeu96,Bes-deJ98}, but we review them so that we can refer to
the relevant details in some new constructions for $ \O $ and in the calculations
relating to regulators in later sections.  Also, in various cases the constructions
were carried out more generally, in which case they tend to become 
dependent on assumptions on weights in $ K $-theory, and our
exposition below will avoid such assumptions.

\subsubsection{Construction of the complexes $ \scriptM 2 F $ and $ \tildescriptM 2 F $.}
\label{scriptM2Fconstruction}

The principle of the construction of the complex $ \scriptM 2 F $ was first used in Bloch's
Irvine notes (finally published as \cite{Blo00}).
The construction of $ \scriptM 2 F $ and $ \tildescriptM 2 F $
can be found in \cite[Section~3]{Jeu95}.

We start with the localization sequence
\begin{equation}
\label{basicFlocalization}
\begin{split}
\xymatrix{
 \dots \ar[r]
&
 \coprod_{t \in \SPF} \K 2 1 F \ar[r]
&
 \K 2 2 {X_F;\bbox} \ar[r]
&
 \K 2 2 {\Xfloc;\bbox} \ar[r]
&{}
}
\\
\xymatrix{
&
 \coprod_{t \in \SPF} \K 1 1 F \ar[r]
&
\K 1 2 {X_F;\bbox} \ar[r]
&
\dots
\,.
}
\end{split}
\end{equation}
Because $ \K 2 1 F = 0 $ for any field $ F $
by \eqref{weightdecomposition},
this means that the cohomological complex (in degrees 1 and 2)
\begin{equation}
\label{genF2complex}
RC_{(2)}(F) :
 \K 2 2 {\Xfloc;\bbox}
\to
 \coprod_{t \in \SPF} \K 1 1 F 
\end{equation}
has cohomology groups
$ H^1(RC_{(2)}(F)) \iso \K 3 2 F $ and
$ H^2(RC_{(2)}(F)) \iso \K 2 2 F $.

In \cite[Section~3.2]{Jeu95}, see also~\cite{Blo90}, 
for every $ x $ in $ \SPF $ an element $ [x]_2 $ was
constructed in $ \K 2 2 {\Xfloc;\bbox} $
with the property that
its boundary in $ \coprod \K 1 1 F $ is $ (1-x)^{-1}_{|t=x} $.
Let
\[
\Symb_1(F) = \K 1 1 F = \FQ
\,,
\]
and 
\[
\Symb_2(F) = \langle [x]_2 \text{ with } x \text{ in } \SPF \rangle_\Q + \I \cup \Symb_1(F)
\,.
\]
Then we get a subcomplex of \eqref{genF2complex}
\begin{equation} \label{F2symbcomplex}
Symb_2(F) :
\Symb_2(F)  \to \coprod_{t \in \SPF} \Symb_1(F)
\,.
\end{equation}
Letting $ \FQ $ act on the right in \eqref{Iinjects} gives the subcomplex
\begin{equation} \label{I2subcomplex}
\I \cup \FQ \to \dd(\dots)
\,,
\end{equation} 
which is acyclic by \cite[Lemma~3.7]{Jeu95}.
Taking the quotient of \eqref{F2symbcomplex} by \eqref{I2subcomplex},
we obtain the complex
\begin{equation*}
\scriptM 2 F :M_2(F) \to \FQ \tensor \FQ
\,,
\end{equation*}
where we used that $ \dd \I $ gives exactly the right
relations to turn $ \coprod_{t \in \SPF} \cdots $ into
$ \cdots \tensor \FQ $, as $ \SPF $ generates $ F^* $,
and  $ M_2(F) = \Symb_2(F) / \I \cup \Symb_1(F) = \Symb_2(F) / \I \cup \FQ $.
Then $ M_2(F) $ is a $ \Q $-vector space generated by the $ [x]_2 $,
$ x $ in $ \SPF $, and the boundary of $ [x]_2 $ is $ (1-x) \tensor x $.

Note that from the maps
\[
\scriptM 2 F \leftarrow Symb_2(F) \to RC_{(2)}(F)
\]
with the left one a quasi isomorphism, we obtain maps
\[ 
H^i(\scriptM 2 F ) \to \K 4-i 2 F
\]
for $ i=1 $ and 2.  The map for $ i=1 $ is an injection 
as the corresponding statement holds for $ RC_{(2)}(F) $ and
$ Symb_2(F) $ is a subcomplex, and we are in the lowest degree.
For $ i=2 $ the map is an isomorphism because $ \K 2 2 F $ is
the quotient of $ \FQ \tensor \FQ $ by $ \langle x \tensor (1-x) \text{ with $ x $ in $ \SPF $}\rangle $.

We shall quotient out the complex $ \scriptM 2 F $ 
in order to end up with a second term $ \bigwedge^2 \FQ $
rather than $ \FQ \tensor \FQ $.  
The shape of the quotient complexes $ \tildescriptM 2 F $ here and
$ \tildescriptM 3 F $ in Section~\ref{M3Fcomplexes} is more in
line with conjectures (see, e.g., \cite[Conjecture~2.1]{gonXpam}).
Besides, the definition of complex $ \scriptM 3 C' $ depends
on the complexes $ \tildescriptM 2 L $ for number fields $ L $.

Namely, consider the subcomplex
of $ \scriptM 2 F $
\begin{equation} \label{N2Fcomplex}
 N_2(F) \to \dd(\dots)
\end{equation}
with
\begin{equation} \label{N2Fdef}
 N_2(F) = \langle [u]_2 + [u^{-1}]_2 \text{ with } u \text{ in } \SPF \rangle_\Q \subseteq M_2(F) 
\,.
\end{equation}
As $ \dd ([x]_2 + [x^{-1}]_2 ) = x \tensor x $
the second term is in fact $ \Sym^2(\FQ) $.  
By the proof of~\cite[Corollary~3.22]{Jeu95}~\eqref{N2Fcomplex}
 is acyclic. Taking the quotient complex we get
\begin{equation} \label{tildescriptM2Fcomplex}
\tildescriptM 2 F : \tM_2(F) \to \bigwedge^2 \FQ
\,,
\end{equation}
with $ \tM_2(F) = M_2(F) / N_2(F) $,
and $ \dd [x]_2 = (1-x) \wedge x $.

Because $ \tildescriptM 2 F $ is quasi isomorphic to $ \scriptM 2 F $
we have maps
\begin{equation} \label{tM2maps}
H^i(\tildescriptM 2 F ) \to \K 4-i 2 F
\,.
\end{equation}
Again this maps is an injection for $ i=1 $ and an isomorphism for $ i=2 $.

There are essentially two ways of generalizing the complex $ \scriptM 2 F $.
The first one is to look at another part of the localization
sequence \eqref{basicFlocalization}, the other to replace $ X_F $
by $ X_F^n $ for $ n \geq 2 $, and use localization there, which will give a spectral
sequence.  The first will be used to construct the complex $ \Ccomp \bullet F $
in Section~\ref{CcompFconstruction} below, the second (with $ n=2 $) will be used for constructing the complex
$ \scriptM 3 F $ below.

\subsubsection{Construction of the complexes $ \scriptM 3 F $ and $ \tildescriptM 3 F $.} \label{M3Fcomplexes}

Those complexes were also defined in \cite[Section~3]{Jeu95}.
The complex $ \scriptM 3 F $ consists
of three terms in cohomological degrees 1, 2 and 3:
\begin{equation}
M_3(F) \to M_2(F) \tensor \FQ \to \FQ \tensor \bigwedge^2 \FQ 
\end{equation}
and comes equipped with maps
$ H^2(\scriptM 3 F ) \to \K 4 3 F $ and
$ H^3(\scriptM 3 F ) \to \K 3 3 F $.
The last of those two maps is in fact an isomorphism.

Although we shall need a similar complex $ \scriptM 3 {\O} $ in order
to have information about the special fiber,
we describe the complex $ \scriptM 3 F $ first, as it is notationally
easier.  Moreover, in the part of the complex we are interested
in, we can view $ \scriptM 3 {\O} $ as a subcomplex of $ \scriptM 3 F $
(see Remark~\ref{m2includes}).

Consider the divisors on $ X_F^2 $ defined by putting $ t_i =u_j $
for some $ u_j $ in $\SPF $ for $ i = 1 $ or 2.
Then there is a spectral sequence (see \cite[page~257]{Jeu96}
or \cite[Page~221]{Jeu95})
\begin{alignat}{3}
\label{Fss}
& \qquad\quad\vdots & \vdots\quad\qquad\qquad\qquad\qquad & \quad\qquad\vdots
\\
\notag
& \K 2 3 {X_F^{2,\loc};\bbox^2}
& 
\quad\coprod_{t_1 \in \SPF} \K 1 2 {\Xfloc; \bbox}
\,{\textstyle\coprod}
\coprod_{t_2 \in \SPF } \K 1 2 {\Xfloc; \bbox}
&
\quad
\coprod_{t_1, t_2  \in \SPF } \K 0 1 F
\\
\notag
& \K 3 3 {X_F^{2,\loc};\bbox^2}
& 
\quad\coprod_{t_1 \in \SPF } \K 2 2 {\Xfloc; \bbox}
\,{\textstyle\coprod}
\coprod_{t_2 \in \SPF } \K 2 2 {\Xfloc; \bbox}
&
\quad
\coprod_{t_1, t_2 \in \SPF } \K 1 1 F
\\
\notag
& \K 4 3 {X_F^{2,\loc};\bbox^2}
& 
\quad\coprod_{t_1 \in \SPF } \K 3 2 {\Xfloc; \bbox}
\,{\textstyle\coprod}
\coprod_{t_2 \in \SPF } \K 3 2 {\Xfloc; \bbox}
&
\quad
\coprod_{t_1, t_2 \in \SPF } \K 2 1 F
\\
\notag
& \qquad\quad\vdots & \vdots\quad\qquad\qquad\qquad\qquad & \quad\qquad\vdots
\end{alignat}
converging to $ \K * 3 {X_F^2;\bbox^2} \iso \K *+2 3 F $.  The only terms in it that contribute to $ \K 4 3 F $
are $ \K 2 3 {X_F^{2,\loc};\bbox^2} $ and
$ \coprod_{t_1 \in \SPF } \K 2 2 {\Xfloc; \bbox} \,{\textstyle\coprod} \coprod_{t_2 \in \SPF } \K 2 2 {\Xfloc; \bbox} $
because $ \coprod_{t_1,t_2 \in \SPF } \K 1 2 F $ is trivial.
Let $ RC_{(3)}(F) $ be the cohomological complex in degrees 1, 2 and 3, consisting of the
row in~\eqref{Fss} that begins with $ \K 3 3 {X_F^{2,\loc};\bbox^2} $:
\begin{equation} \label{gencomplex}
\begin{split}
RC_{(3)}(F) :
\K 3 3 {X_F^{2,\loc};\bbox^2} \to 
\qquad\qquad\qquad\qquad\qquad\qquad\qquad\qquad\qquad\qquad
\\
\qquad\qquad\quad\coprod_{t_1 \in \SPF } \K 2 2 {\Xfloc; \bbox}
\,{\textstyle\coprod}
\coprod_{t_2 \in \SPF } \K 2 2 {\Xfloc; \bbox}
\to 
\coprod_{t_1, t_2 \in \SPF } \K 1 1 F
\,.
\end{split}
\end{equation}
This complex was denoted $ C_{(3)} $ in \cite[Section~3.1]{Jeu95},
but considering the notational overload of the letter $ C $ in this paper,
we prefer to think of it as a Row Complex rather than just a
Complex.

Note that $ \K 1 2 F $ equals zero, so for $ i=2 $ and 3 there is a map
\begin{equation}
\label{genmap}
H^i( RC_{(3)}(F) ) \to \K 6-i 3 F
\,.
\end{equation}

For $ x $ in $ \SPF $, in addition to the element $ [x]_2 $ in $ \K 2 2 {\Xfloc; \bbox} $
of Section~\ref{scriptM2Fconstruction}, there is also
an element $ [x]_3 $ in $ \K 3 3 {X_F^{2,\loc};\bbox^2} $
(see \cite[Section~3.2]{Jeu95})
with boundary $ - [x]_{2 | t_1 = x} +[x]_{2 | t_2 = x}  $ in $ \coprod_{t_1 \in \SPF } \K 2 2 {\Xfloc; \bbox}
\,{\textstyle\coprod}
\coprod_{t_2 \in \SPF } \K 2 2 {\Xfloc; \bbox} $
in~\eqref{Fss}.
Let us define $ \Symb_n(F) \subseteq \K n n {X_F^{n-1,\loc};\bbox^{n-1}} $
for $ n=1 $, 2 and 3 by setting
\[
\Symb_1(F) = \FQ
\,,
\]
\[
\Symb_2(F) = \langle [u]_2 \text{ with } u \text{ in }  \SPF \rangle_\Q + \I \cup \Symb_1(F) 
\,,
\]
and 
\[
\Symb_3(F) = \langle [u]_3 \text{ with } u \text{ in }  \SPF \rangle_\Q + \I \, \tcup \, \Symb_2(F) 
\,.
\]
For $ n=2 $, those are the definitions given in Section~\ref{scriptM2Fconstruction},
and for $ n = 3 $, by $ \tcup $ we mean the following. In the projection $ X_F^2 $ to $ X_F $,
we
can use one of the factors to pull back $ \I $, the other to pull
back $ \Symb_2(F) $ and then take the product to land in $ \Symb_3(F) $,
giving us two cup products.
The $ \tcup $ indicates that we take the sum of the images of
both possibilities for those cup products.

Because, in \eqref{gencomplex}, $ \dd [u]_2 = (1-u)^{-1}_{|t=u} $, and
$ \dd [u]_3 = -[u]_{2|t_1 = u} + [u]_{2|t_2 = u} $, it follows that
\begin{equation}
\label{symbcomplex}
Symb_{(3)}(F) :
\Symb_3(F) \to \coprod_{t_1\in\SPF}\Symb_2(F) \,{\textstyle \coprod} \coprod_{t_2\in\SPF} \Symb_2(F)
        \to \coprod_{t_1,t_2\in\SPF} \Symb_1(F)
\end{equation}
is a subcomplex of~\eqref{gencomplex}.
It is shown in \cite[Lemma~3.9 and Remark~3.10]{Jeu95} that the subcomplex 
\begin{equation}
\label{IFcomplex}
\I \, \tcup \, \Symb_2(F) \to
  \coprod_{t_1 \in \SPF } \I \cup \FQ \,{\textstyle \coprod} \coprod_{t_2\in\SPF}  \I \cup \FQ + \dd (\dots) \to \dd (...)
\end{equation}
of~\eqref{symbcomplex} is acyclic.

$ S_2 $ acts on the spectral sequence~\eqref{Fss} by swapping
$ t_1 $ and $ t_2 $.  It therefore also acts on the complex~\eqref{gencomplex}
above.  Because the symbol $ [x]_3 $ is alternating by construction
(see \cite[Section~3.2]{Jeu95}), we can take the alternating parts 
of~\eqref{symbcomplex} and~\eqref{IFcomplex}, and form the quotient complex
\begin{equation*}
\scriptM 3 F : M_3(F) \to M_2(F) \tensor \FQ \to \FQ \tensor
\bigwedge^2 \FQ 
\,,
\end{equation*}
where
\[
M_3(F) = \Symb_3(F) / \left(\I \, \tcup \, \Symb_2(F) \right)^\alt
\,,
\]
and 
\[
M_2(F) = \Symb_2(F)/ \I \cup \FQ 
\]
as before in Section~\ref{scriptM2Fconstruction}.
Note that, for $ n=2 $ and 3, $ M_n(F) $ is a $ \Q $-vector space on symbols $ [x]_n $
for $ x $ in $ \SPF $, modulo non-explicit relations depending
on $ n $.  The maps in the complex are given by
\[
\dd [x]_3 = [x]_2 \tensor x
\]
and
\[
\dd [x]_2 \tensor y = (1-x) \tensor (x \wedge y)
.
\]
As before, we used here that $ \dd \I $ gives exactly the right
relations to turn $ \coprod_{t \in \SPF} \dots $ into
$ \dots \tensor \FQ $, as $ \SPF $ generates $ F^* $.
As $ Symb_{(3)}(F) $ is a subcomplex of $ RC_{(3)}(F) $, this
gives us maps
\[
\scriptM 3 F \leftarrow Symb_{(3)}(F)^\alt \to RC_{(3)}(F)^\alt
\to RC_{(3)}(F)
\]
with the left map a quasi isomorphism.
Combining this with~\eqref{genmap} gives us a map 
\begin{equation} \label{HM3Fmap}
H^i( \scriptM 3 F ) \to \K 6-i 3 F 
\end{equation}
for $ i=2 $ and 3.
(For $ i=1 $, starting with
$ H^1(RC_{(3)}(F) ) \to \K 5 3 F / \K 4 2 F \cup \FQ $, we still
obtain a map $ H^1(\scriptM 3 F ) \to \K 5 3 F / \K 4 2 F \cup \FQ $.)

Finally, we quotient out $ \scriptM 3 F $ in order to obtain
$ \tildescriptM 3 F $, as follows. Let
\begin{equation*}
  N_3(F) = \langle [u]_3 - [u^{-1}]_3 \text{ with } u \text{ in } \SPF \rangle_\Q \subseteq M_3(F) 
\end{equation*}
(cf.~\eqref{N2Fdef}; in general $ N_n(F) $ is generated by the
$ [u]_n +(-1)^n [u^{-1}]_n $) and consider the subcomplex
\begin{equation} \label{N3Fcomplex}
\xymatrix{
 N_3(F) \ar[r]
&
 N_2(F) \otimes \FQ \ar[r]
&
\dd(\dots)
}
\end{equation}
of $ \scriptM 3 F $. By the proofs of \cite[Proposition~3.20,
Corollary~3.22]{Jeu95} it is acyclic in degrees~2 and~3, hence
for the quotient complex
\begin{equation*}
 \tildescriptM 3 {F} : \tM_3(F) \to \tM_2(F) \otimes \FQ \to \bigwedge^3\FQ
\,,
\end{equation*}
where $ \tM_3(F) = M_3(F) / N_3(F) $,
we get a map 
\begin{equation} \label{HtildeM3Fmap}
 H^i(\tildescriptM 3 F ) \leftiso H^i(\scriptM 3 F ) \to \K 6-i 3 F
\,.
\end{equation}
In $ \tM_3(F) $ we still denote the class of $ [x]_i $ with $ [x]_i $, so that
the maps
are now given by 
$ \dd [u]_3 = [u]_2 \tensor u $
and
$ \dd [u]_2 \tensor v = (1-u) \wedge u \wedge v $.

\subsubsection{Construction of the complex $ \scriptM 3 C' $.}
\label{scriptM3Cconstruction}

In this section we consider the situation where we have 
 smooth, projective, geometrically irreducible
curve $ C' $ over a number field $ k $ with function field $ F' = k(C') $.

Because we are interested in finding elements in $ \K 4 3 C' $,
we introduce yet another complex, $ \scriptM 3 C' $, which
is the total complex associated to the double complex
\begin{equation*}
\begin{split}
\xymatrix{
M_3 (F') \ar[r]^-\dd \ar[d] & M_2 (F') \tensor_\Q \FpQ \ar[r]^-\dd \ar[d]_{\partial_1} & \FpQ\tensor \bigwedge^2 \FpQ \ar[d] _{\partial_2}\cr
0 \ar[r] & \coprod_x \tM_2 (k(x)) \ar[r] ^-\dd& \coprod_x \bigwedge^2 k(x)_\Q^*
\,. 
}
\end{split}
\end{equation*}
(Although not needed in this paper, one could define the complex
$ \tildescriptM 3 C' $ by using $ \tildescriptM 3 F' $ in the top row.)
Here the coproducts are over all closed points $ x $ of $ C' $.
The boundary maps are as follows.
The $ \dd $'s in the top row are as in $ \scriptM 3 F' $.
In the bottom row, $ \dd [z]_2 = (1-z) \wedge z $.
For the vertical maps, $ \partial_{1,x} ([g]_2 \tensor f) = \ord_x(f) \cdot [g(x)]_2 $,
with the convention that $ [0]_2 = [1]_2 = [\infty]_2 = 0 $.
Finally, 
$ \partial_{2,x} $ described as follows.  Let $ \pi $ be a uniformizer at $ x $,
$ u_j $ units at $ x $.  Then $ \partial_{2,x} $ is determined
by 
\[
 \pi \wedge u_1 \wedge u_2  \mapsto u_1(x) \wedge u_2(x) 
\text{ and }
 u_1 \wedge u _2 \wedge u_3  \mapsto 0 
\,.
\]
Therefore, an element
$ \sum_i [g_i]_2 \tensor f_i $ in $ H^2(\scriptM 3 F' ) $ 
satisfies
\begin{equation*}
  \sum_i (1-g_i) \tensor ( g_i \wedge f_i) = 0
\end{equation*}
in $ \FpQ \tensor \bigwedge^2 \FpQ $.  The additional condition
for it to lie in  $ H^2(\scriptM 3 C' ) $ is that
\begin{equation*}
\sum_i \ord_x(f_i) [g_i(x)]_2 = 0
\end{equation*}
in $ \tM_2(k(x)) $ for all closed points $ x $ in $ C' $, with the convention
that $ [0]_2 = [1]_2 = [\infty]_2 = 0 $.

We have an obvious map $ \scriptM 3 C' \to \scriptM 3 F' $, corresponding
to the localization map in~\eqref{kloc}.  In \cite[Theorem~5.2]{Jeu96},
it is shown that this induces a commutative diagram
\begin{equation} \label{Fboundcd}
\begin{split}
\xymatrix{
 H^2(\scriptM 3 C' ) \ar[d] \ar[r] & H^2(\scriptM 3 F' )  \ar[d] 
\cr
\K 4 3 C' \oplus \K 3 2 k \cup \FpQ \ar[r] & \K 4 3 F' 
\,.
}
\end{split}
\end{equation}
Note that it was shown in Remark~\ref{directsum} that
$ \K 4 3 C' \oplus \K 3 2 k \cup \FpQ $ is indeed a direct sum,
and that the lower horizontal map is an injection.

\begin{remark} \label{Fprojection}
If $ k $ is totally real then $ \K 3 2 k $ is zero.  But in general
we can use the projection 
\begin{equation*}
\K 4 3 C' \oplus \K 3 2 k \cup \FpQ \to \K 4 3 C' 
\end{equation*}
to get a map 
$ H^2(\scriptM 3 C' ) \to \K 4 3 C' $
as the composition
\[
H^2(\scriptM 3 C' ) \to \K 4 3 C' \oplus \K 3 2 k \cup \FpQ \to \K 4 3 C' 
\,.
\]
\end{remark}

\subsubsection{Construction of the complex $ \Ccomp \bullet F $.}
\label{CcompFconstruction}

The complex $ \Ccomp \bullet F $ is described in \cite[Section~3]{Jeu96},
but it was first constructed in \cite{Blo90}.
We recall its construction in order to clarify the construction
of the corresponding complex for $ \O $ in Section~\ref{CcompOconstruction}.

One starts with another part of the exact localization sequence~\eqref{basicFlocalization} in relative
$ K $-theory.
\begin{equation}
\label{secondbasicFlocalization}
\begin{split}
\xymatrix{
  \dots \ar[r]
&
 \coprod_{t \in \SPF} \K 3 2 F \ar[r]
&
 \K 3 3 {X_F;\bbox} \ar[r]
&
 \K 3 3 {\Xfloc;\bbox} \ar[r]
&{}
}
\\
\xymatrix{
 \coprod_{t \in \SPF} \K 2 2 F \ar[r]
&
\K 2 3 {X_F;\bbox} \ar[r]
&
\dots
}
\,.
\end{split}
\end{equation}
Because $ \K 2 3 (X_F;\bbox) \iso \K 3 3 F \iso K_3^M(F)_\Q  $,
so that the map $  \coprod_{t \in \SPF} \K 2 2 F \to \K 2 3 {X_F;\bbox} $
is surjective,
this shows that the cohomological complex in degrees 1 and 2
\begin{equation*}
AC_{(3)}(F) :
 \K 3 3 {\Xfloc;\bbox} 
\to
 \coprod_{t \in \SPF} \K 2 2 F 
\end{equation*}
has maps 
\[
H^1(AC_{(3)}(F)) \iso \K 4 3 F / \K 3 2 F \cup \FQ 
\]
and
\[
H^2(AC_{(3)}(F)) \iso \K 3 3 F 
\,.
\]
(Here $ AC $ stands for Auxiliary Complex.)

Again we have an acyclic subcomplex
\begin{equation*}
\I \cup \K 2 2 F \to \dd(\dots)
\,,
\end{equation*}
and therefore the quotient complex $ \Ccomp \bullet F $ is a
cohomological complex in degree 1 and 2,
\[
\Ccomp \bullet F :
\Ccomp 1 F \to \Ccomp 2 F
\,,
\]
with 
\[
\Ccomp 1 F = \frac{\K 3 3 {\Xfloc;\bbox} }{\I \cup \K 2 2 F }
\]
and
\[
\Ccomp 2 F = \K 2 2 F \tensor \FQ
\,.
\]
It comes with maps 
\[
H^1(\Ccomp \bullet F ) \iso \K 4 3 F / \K 3 2 F \cup \FQ
\]
and
\[
H^2(\Ccomp \bullet F ) \iso \K 3 3 F 
\,.
\]

Note that if $ g $ is in $ \SPF $, and $ f $ is in $ F^* $, then
$ [g]_2 \cup f $ lies in $ \K 3 3 {\Xfloc;\bbox} $.  In fact,
if we take the class of $ [g]_2 $ in $ M_2(F) $ instead, then
we do get a well-defined class in $ \Ccomp 1 F $, as $ \I \cup \FQ \cup f $
goes to zero in $ \Ccomp 1 F $ by definition.
Under the differential in the complex, $ \symb g f $ is mapped
to $ \{ (1-g)^{-1}, f \} \tensor g = - \{ 1-g , f \} \tensor g $,
so the condition for an element $ \sum_i \symb {g_i} {f_i} $
to be in $ H^1(\Ccomp \bullet F ) $ is that
\begin{equation*}
\sum_i \{ 1-g_i, f_i \} \tensor g_i = 0
\end{equation*}
in $ \K 2 2 F \tensor \FQ $.

The map
\[
M_{(2)}(F) \tensor \FQ \to \Ccomp 1 F
\]
given by 
\[
[g]_2 \tensor f \mapsto [g]_2 \cup f
\]
fits into a commutative diagram
\begin{equation}
\label{shiftedcomplexmap}
\begin{split}
\xymatrix{
M_3(F) \ar[r]\ar[d]
&
M_2(F) \tensor \FQ \ar[r]\ar[d]
&
\FQ \tensor \bigwedge^2\FQ\ar[d]
\\
0 \ar[r]
&
\Ccomp 1 F \ar[r]
&
\Ccomp 2 F 
}
\end{split}
\end{equation}
where we map $ f\tensor g\wedge h $ to $ \{f, g \} \tensor h - \{f, h\} \tensor g $.
Multiplying the map $ H^2(\scriptM 3 F ) \to \K 4 3 F $ by $ -1 $
if necessary, we obtain a commutative diagram
\begin{equation}
\label{auxdiagram}
\begin{split}
\xymatrix{
H^2(\scriptM 3 F ) \ar[r]\ar[d] & \K 4 3 F \ar[d]
\\
H^1(\Ccomp \bullet F ) \ar[r] & \K 4 3 F / \K 3 2 F \cup \FQ 
}
\end{split}
\end{equation}
(see \cite[Proposition~3.2]{Jeu96}).

\subsection{Construction of the complexes for $ \O $ and $ \CC' $.}
\label{Ocomplexes}

\begin{remark}
At various stages
there will be some properties of the complexes
for $ \O $ that depend on
$ \K 3 2 {\kk} $
being trivial.  Clearly,
this applies to $ \O $ as in Section~\ref{sec:intro} by our remarks
about the $ K $-groups of $ \Cspff $ and $ \Cspnfff $
in Section~\ref{prelimk}.
\end{remark}

\subsubsection{Construction of the complex $ \scriptM 2 {\O} $.}

When we try to imitate the localization sequence~\eqref{basicFlocalization}
for $ \O $ rather than $ F $, we are dealing with the two dimensional
scheme $ X_\O $, and we end up with a spectral sequence instead,
\begin{alignat}{3}
\notag
\qquad\vdots\qquad\qquad & \qquad\qquad\vdots\qquad &
\\
\notag
\K 1 2 {\Xoloc;\bbox}  &\qquad \coprod_{t \in \SPO} \K 0 1 F &
\\
\label{firstO-ss}
\K 2 2 {\Xoloc;\bbox}  &\qquad \coprod_{t \in \SPO} \K 1 1 F &\qquad\coprod_{t \in \SPkk} \K 0 0 {\kk}
\\
\notag
\K 3 2 {\Xoloc;\bbox}  &\qquad \coprod_{t \in \SPO} \K 2 1 F &\qquad\coprod_{t \in \SPkk} \K 1 0 {\kk}
\\
\notag
\qquad\vdots\qquad\qquad & \qquad\qquad\vdots\qquad & \vdots\qquad\qquad
\end{alignat}
which converges to $ \K * 2 {\Xoloc;\bbox} \iso \K *+1 2 {\O} $.

Because $ \K 2 1 F  $, $ \K 1 0 {\kappa} $ and $ \K 2 0 {\kappa} $
are all trivial, if we let $ RC_{(2)}(\O) $ be the cohomological
complex in degrees 1, 2 and 3, given by
\begin{equation}
\label{RC2Ocomplex}
\K 2 2 {\Xoloc;\bbox}  \to \coprod_{t \in \SPO} \K 1 1 F \to \coprod_{t \in \SPkk} \K 0 0 {\kk}
\,,
\end{equation}
then there are maps
$ H^1(RC_{(2)}(\O)) \iso \K 3 2 {\O} $ and
$ H^2(RC_{(2)}(\O)) \to \K 2 2 {\O} $.
The last map is surjective by Proposition~\ref{milnork} and the exact
sequence
\[
\dots \to \K 2 1 {\kappa} \to \K 2 2 {\O} \to \K 2 2 F \to \K 1 1 {\kappa} \to \dots 
\]
as $ \K 2 1 {\kappa} = 0 $.
Note that the map $ \K 1 1 F \to \K 0 0 {\kappa} $ is surjective,
so that $ H^3(RC_{(2)}(\O)) $ is zero, as is $ \K 1 2 {\O} $.

Now let $ A \subseteq \K 2 2 {\Xoloc;\bbox} $ be the inverse
image of $ \coprod_{t \in \SPO} \OQ $ in $ \coprod_{t \in \SPO} \K 1 1 F $.
Because $ \K 1 1 {\O} = \OQ $ is equal to  $ \ker\left( \K 1 1 F \to \K 0 0 {\kappa} \right) $,
this means that the subcomplex 
\begin{equation}
\label{firstOsubcomplex}
RC_{(2)}(\O) :
A \to  \coprod_{t \in \SPO} \OQ
\end{equation}
of~\eqref{RC2Ocomplex} has maps 
$ H^1(RC_{(2)}(\O)) \to \K 3 2 {\O} $ and
$ H^2(RC_{(2)}(\O)) \to \K 2 2 {\O} $.

We again use the element $ [u]_2 $ in $ \K 2 2 {\Xoloc;\bbox} $
for every $ u $ in $ \SPO $, and put
\[
\Symb_1(\O) = \K 1 1 {\O} = \OQ
\,,
\]
and 
\[
\Symb_2(\O) = \langle [u]_2 \text{ with } u \text{ in } \SPO \rangle_\Q + \IO \cup \OQ
\,.
\]
(See~\eqref{IOdef} for the definition of $ \IO $.)
Observe that, if $ u $ is in $ \SPO $ and $ v $ is in $ \OQ $,
then $ [u]_2 $ and $ \IO \cup v $ are in $ A $, so
we get a subcomplex of~\eqref{firstOsubcomplex}
\begin{equation}
\label{O2symbcomplex}
Symb_2(\O) :
\Symb_2(\O)  \to \coprod_{t \in \SPO} \OQ
\,,
\end{equation}
containing the acyclic subcomplex 
\begin{equation}
\label{IO2subcomplex}
\IO \cup \OQ \to \dd(\dots)
\,.
\end{equation}
We take the quotient complex of~\eqref{O2symbcomplex} by~\eqref{IO2subcomplex},
to obtain the complex
\begin{equation} \label{scriptM2O}
\scriptM 2 {\O} : M_2(\O) \to \OQ \tensor \OQ
\,,
\end{equation}
with $ M_2(\O) = \Sym_2(\O) / \I \cup \OQ $.
Then $ M_2(\O) $ is a $ \Q $-vector space generated by the $ [u]_2 $,
$ u $ in $ \SPO $, and $ \dd [u]_2 = (1-u) \tensor u $.
(Again, we used that $ \dd \IO \cup \OQ $ gives us exactly the
right relations to change $ \coprod_{t \in \SPO} \OQ$ into $ \OQ \tensor \OQ $
because $ \SPO $ generates $ \O^* $.)
Note that we now have maps
\[
\scriptM 2 {\O} \leftarrow Symb_2(\O) \to RC_{(2)}(\O)
\,,
\]
with the left one a quasi isomorphism, so we obtain maps
\begin{equation} \label{scriptM2maps}
H^i(\scriptM 2 {\O} ) \to \K 4-i 2 {\O}
\end{equation}
for $ i=1 $ and 2.  Again the map for $ i=1 $ is an injection 
(cf.~\eqref{tM2maps}).
For $ i=2 $ the map is a surjection by Proposition~\ref{milnork} because
$ \K 2 2 {\O} = \ker\left( \K 2 2 F \to \K 1 1 {\kappa} \right) $.

Localizing the base from $ \O $ to $ F $ in~\eqref{firstO-ss}
gives us~\eqref{Fss}, so that we get a map of complexes
$ M_2(\O) \to M_2(F) $
since the various steps in the constructions of the two complexes
are compatible.

\begin{remark} \label{m2includes}
The map $ M_2(\O) \to M_2(F) $
is injective.  Namely, because the construction of the complexes
for $ \scriptM 2 {\O} $ and $ \scriptM 2 F $ is compatible with
the localization from $ \O $ to $ F $ in ~\eqref{firstO-ss},
we have a commutative diagram
\[
\xymatrix{
0 \ar[r] & H^1(\scriptM 2 {\O} ) \ar[r]\ar[d] & M_2(\O) \ar[r]\ar[d] & \OQ \tensor \OQ \ar[d]
\\
0 \ar[r] & H^1(\scriptM 2 {F} ) \ar[r] & M_2(F) \ar[r] & \FQ \tensor \FQ
\,,
}
\]
with $ H^1(\scriptM 2 {\O}) \subseteq \K 3 2 {\O} $ 
and $ H^1(\scriptM 2 F ) \subseteq \K 3 2 F $.
From the exact localization sequence
\[
\dots \to \K 3 1 {\kk} \to \K 3 2 {\O} \to \K 3 2 F \to \K 2 1 {\kk} \to \dots
\]
we see that $ \K 3 2 {\O} \iso \K 3 2 F $, so that the map on
$ H^1 $'s must be injective.  As $ \OQ \tensor \OQ \to \FQ \tensor \FQ $
is clearly injective, $ M_2(\O) \to M_2(F) $ must be injective
as well. So we may think of $ M_2(\O) $ as the subspace of $ M_2(F) $
generated
by the $ [u]_2 $ with $ u $ in $ \SPO \subset \SPF $.  
\end{remark}

\subsubsection{Construction of the complex $ \scriptM 3 {\O} $.} \label{M3O}

In this subsection, we shall be making Assumption~\ref{sit}.

If we now try to imitate the construction of $ \scriptM 3 F $ using $ \O $ instead of $ F $,
see some differences.   For example, in the construction of the
spectral sequence,
in codimension one, we shall end up with copies of $ \{ t_i = u \} $
for $ u  $ in $ \SPO $, 
which look like $ X_\O $, out of which we have to remove the
intersections with all other such pieces of codimension one of
the form $ \{ t_i = v \} $ for $ i =1 $ and 2, and $ v $ in $ \SPO $.
Note that, in particular, we also cut out $ t_i = v $ with $ u $ and $ v $
different elements in $ \SPO $, but reducing to the same in the
residue field.  Then $ t_i=v $ cuts out the bit in the special
fibre in $ t_i = u $.  
We therefore end up with copies of $ \XFLOC =  X_F \setminus \{ t = u \text{ with } u \text{ in } \SPO \}$.

So if we do this for $ \O $, we end up with the following spectral
sequence, converging to $ \K * 3 {X_{\O}^2;\bbox^2} \iso \K *+2 3 {\O} $
(see \cite[(3.7)]{Bes-deJ98}). For typographical reasons, let
us abbreviate $ \K n j {X_\O^m;\bbox^m} $ to $ \Kab n j m {\O} $,
$ \K n j {\XFLOC;\bbox} $ to $ \Kab n j 1 F $,
and $ \K n j {X_\kk;\bbox} $ to $ \Kab n j 1 {\kk} $. Then the
spectral sequence is

\begin{alignat}{4}
\label{Oss}
&
\quad\vdots
&
\vdots\qquad\quad
&
\quad\qquad\qquad\qquad\qquad\vdots
\\
\notag
&
\Kab 2 3 2 {\O} 
& 
\quad 
\left( \coprod_{t \in \SPO } \Kab 1 2 1 F \right)^2
&
\quad
\coprod_{t_1, t_2 \in \SPO } \K 0 1 F
\,{\textstyle\coprod}
\left( \coprod_{t \in \SPkk } \Kab 0 1 1 {\kk} \right)^2
\\
\notag
&
\Kab 3 3 2 {\O}
& 
\quad
\left(\coprod_{t \in \SPO } \Kab 2 2 1 F \right)^2
&
\quad
\coprod_{t_1, t_2 \in \SPO } \K 1 1 F
\,{\textstyle\coprod}
\left( \coprod_{t \in \SPkk } \Kab 1 1 1 {\kk} \right)^2
&
\quad
\coprod_{t_1, t_2 \in \SPkk } \K 0 0 {\kk}
\\
\notag
&
\Kab 4 3 2 {\O}
& 
\quad
\left( \coprod_{t \in \SPO } \Kab 3 2 1 F \right)^2
&
\quad
\coprod_{t_1, t_2 \in \SPO } \K 2 1 F
\,{\textstyle\coprod}
\left( \coprod_{t \in \SPkk } \Kab 2 1 1 {\kk} \right)^2
&
\quad
\coprod_{t_1, t_2 \in \SPkk } \K 1 0 {\kk}
\\
\notag
&
\quad\vdots
&
\vdots\qquad\quad
&
\quad\qquad\qquad\qquad\qquad\vdots
&
\vdots\qquad\quad
\end{alignat}
Here the $ (\dots)^2 $ corresponds to two copies, corresponding
to a coproduct over $ t_1 $ in $ \SPO $ or $ \SPkk $, and $ t_2 $ in $ \SPO $ or $ \SPkk $.
As explained before, in order to obtain $ \XFLOC $ out of $ X_F $,
we only remove $ t_i = u_j $ with $ u_j $ in $ \SPO $.

Now notice that all $ \K j 0 {\kk} $ are zero for $ j \geq 1 $,
that $ \K j 1 F  $ is zero for $ j \geq 2 $, and finally that
$ \K j 1 {X_{\kk}^\loc;\bbox} $ is zero as well for
$ j \geq 2 $: we consider the exact localization sequence
\[
 \dots \to
\K j 1 {X_{\kk}^1;\bbox} \to
\K j 1 {X_{\kk}^\loc;\bbox} \to
\coprod \K j-1 0 {\kk} \to \dots
\,,
\]
and use that $ \K j 1 {X_{\kk}^1;\bbox} \iso \K j+1 1 {\kk} $,
which is zero as $ \K m 1 L = 0 $ for $ m \geq 2 $ for any field $ L $,
as well as that $ \K j-1 0 {\kappa} = 0 $ because $ j-1 \geq 1 $.
Therefore, with $ RC_{(3)}(\O) $ the following cohomological complex
in degrees 1 through 4 (corresponding to the row in~\eqref{Oss}
starting with $ \K 3 3 {X_{\O,\loc}^2; \bbox^2} $):
\begin{equation}
\label{Ogencomplex}
\begin{split}
RC_{(3)}(\O) :
\K 3 3 {X_\O^{2,\loc};\bbox^2} 
\to
\left( \coprod_{t \in \SPO } \K 2 2 {\Xfloc; \bbox} \right)^2
\to
\qquad\qquad\qquad
\\
\coprod_{t_1, t_2 \in \SPO } \K 1 1 F
\;{\textstyle\coprod}
\left(\coprod_{t \in \SPkk } \K 1 1 {X_{\kk}^\loc;\bbox} \right)^2
\to
\coprod_{t_1, t_2 \in \SPkk } \K 0 0 {\kk}
\end{split}
\end{equation}
has maps
\begin{equation}
\label{Ogenmap}
H^i( RC_{(3)}(\O) ) \to \K 6-i 3 {\O}
\end{equation}
for $ i = 2 $, 3 and 4.

\begin{remark}
Note that for $ i=4 $ this statement is vacuous since from the
localization sequence
\[
\dots \to \K 3 3 F \to \K 2 2 {\kk} \to \K 2 3 {\O} \to \K 2 3 F \to \dots
\]
and the facts that $ \K 2 3 F $ is trivial, and $ \K 3 3 F \to \K 2 2 {\kk} $ is surjective
(see Proposition~\ref{milnork}), it follows that $ \K 2 3 {\O} $ is zero.
\end{remark}

\begin{remark} \label{Oinjective}
The map
$\K 2 2 {\Xoloc;\bbox} \to \K 2 2 {\XFLOC;\bbox} \to \K 2 2 {\Xfloc;\bbox} $
is injective.
Namely, we have an exact localization sequence
\[
\dots \to \K 2 1 {X_{\kk}^\loc;\bbox} \to \K 2 2 {\Xoloc;\bbox} \to \K 2 2 {\XFLOC;\bbox}
\to \cdots
\,,
\]
and $ \K 2 1 {X_{\kk}^\loc;\bbox} $ equals zero, as seen above.
Also, we have an exact localization sequence
\[
\dots \to
\coprod_{t \in F^* \setminus \SPF \bigcup \{1\}} \K 2 1 F 
\to
\K 2 2 {\XFLOC; \bbox} 
\to
\K 2 2 {\Xfloc;\bbox}
\to \cdots
\,,
\]
and again $ \K 2 1 F $ is zero.
\end{remark}

\begin{remark}
\label{K33OFnotinjective}
Note that, because we can localize $ \O $ to $ F $, we have
a natural map of the spectral sequence in~\eqref{Oss} to the
one in~\eqref{Fss},
which, at the level of the complexes~\eqref{gencomplex} and~\eqref{Ogencomplex},
simply forgets the terms over $ \kk $, includes a coproduct over $ \SPO $
into the corresponding coproduct over $ \SPF $, and uses the
maps $ \K 2 2 {\Xoloc;\bbox} \to \K 2 2 {\Xfloc;\bbox} $
and  $ \K 3 3 {X_{\O}^{2,\loc};\bbox^2} \to \K 3 3 {X_F^{'2,\loc};\bbox^2} $.
By Remark~\ref{Oinjective}, the first one is always injective, and the second
is injective if $ \K 5 2 {\kappa} $ and $ \K 4 2 F $ are zero.
\end{remark}

Let us try to create a jewel in the crown of the scary notation
in~\eqref{Ogencomplex}.  Define
$ \Symb_n(\O) \subseteq \K n n {X_{\O}^{n-1,\loc};\bbox^{n-1}} $
for $ n=1 $, 2 and 3 by setting
\[
\Symb_1(\O) = \OQ
\,,
\]
\[
\Symb_2(\O) = \langle [u]_2 \text{ with } u \text{ in }  \SPO \rangle_\Q + \IO \cup \Symb_1(\O) 
\,,
\]
as before, and 
\[
\Symb_3(\O) = \langle [u]_3 \text{ with } u \text{ in }  \SPO \rangle_\Q + \IO \, \tcup \, \Symb_2(\O) 
\,.
\]
Again, by $ \tcup $ we denote that we use both products, coming
from the two ways of projecting $ X_\O^2 $ to $ X_\O $.

Note that for $ n=1 $, $ \Symb_1(\O) = \OQ \subseteq \Symb_1(F) = \FQ $,
and that for $ n=2 $, we can view $ \Symb_2(\O) \subseteq  \Symb_2(F) $
inside $ \K 2 2 {\Xfloc;\bbox} $ by Remark~\ref{Oinjective},
as $ \K 2 2 {\Xoloc;\bbox}  \subseteq \K 2 2 {\Xfloc;\bbox} $.

Because $ \dd [u]_2 = (1-u)^{-1}_{|t=u} $, and
$ \dd [u]_3 = -[u]_{2|t_1 = u} + [u]_{2|t_2 = u} $ (where both
terms lie in a copy of $ \K 2 2 {\Xoloc;\bbox} $ inside $ \K 2 2 {\Xfloc} $,
again by Remark~\ref{Oinjective}), it follows that
\begin{equation}
\label{Osymbcomplex}
Symb_{(3)}(\O) :
\Symb_3(\O) \to \left( \coprod_{t \in \SPO } \Symb_2(\O) \right)^2
   \to \coprod_{t_1,t_2 \in \SPO } \OQ
\end{equation}
is a subcomplex (in degrees 1, 2 and 3) of~\eqref{Ogencomplex}.  Note that we used here
that elements in $ \SPO $ never give rise to a pole or zero over
$ \kk $, so the map to $ \coprod \K 0 0 {\kk} $ is zero.
Also, we used
that an element $ [u]_2 $ with $ u $ in $ \SPO $ under the localization
(of its construction),
\[
\K 2 2 {X_\O \setminus\{t=u\};\bbox} \to
\K 1 1 {\O} \to \dots
\]
maps to $ (1-u)^{-1} $, so under the boundary in~\eqref{Oss} it never
hits the $ \K 1 1 {X_{\kk}^\loc;\bbox} $ component.  Similarly,
the elements in $ \IO \cup \OQ $ never hit the $ \K 1 1 {X_{\kk}^\loc;\bbox} $.

Again, one shows
that the subcomplex of~\eqref{Osymbcomplex} given by
\[
\IO \, \tcup \, \Symb_2(\O) \to
 \left( \coprod_{t} \IO \cup \OQ \right)^2
  + (\dots)
 \to   \dd (\dots)
\]
is acyclic; see~\cite[Lemma~3.7 and Remark~3.10]{Jeu95}.

Taking the quotient complex, and the alternating part for the
action of $ S_2 $ under swapping the coordinates, we finally
get a complex
\begin{equation*}
M_3(\O) \to M_2(\O) \to \OQ \tensor \bigwedge^2 \OQ
\,.
\end{equation*}
Here 
\[
M_3(\O) = \Symb_3(\O) / \left( \IO \, \tcup \,  \Symb_2(\O) \right)^\alt
\]
and, as before,
\[
M_2(\O) = \Symb_2(\O)/ \IO \cup \OQ 
\,.
\]
Note that  $ M_n(\O) $ is a $ \Q $-vector space on symbols $ [u]_n $
for $ u $ in $ \SPO $, modulo non-explicit relations depending
on $ n $.  The maps in the complex are given by
\[
\dd [u]_3 = [u]_2 \tensor u
\]
and
\[
\dd [u]_2 \tensor v = (1-u) \tensor (u \wedge v)
.
\]

In particular, the condition for an element $ \sum_i [u_i] \tensor v_i $
in $ M_2(\O) \tensor \OQ $ to lie in 
$ H^2(\scriptM 3 {\O} ) $ is that
\begin{equation} \label{Ocond}
\sum_i (1-u_i) \tensor (u_i \wedge v_i ) = 0
\end{equation}
in  $ \OQ \tensor \bigwedge^2 \OQ $.

Again  $ S_2 $ acts on the various complexes by swapping the
coordinates, and we get maps
\[
\scriptM 3 {\O} \leftarrow Symb_{(3)}(\O)^\alt \to RC_{(3)}(\O)^\alt \to RC_{(3)}(\O)
\]
with the left map a quasi isomorphism.
Combining this with~\eqref{Ogenmap} gives us a map 
\begin{equation} \label{scriptMOmaps}
H^i( \scriptM 3 {\O} ) \to \K 6-i 3 {\O} 
\end{equation}
for $ i=2 $ and 3, where the map for $ i=3 $ is a surjection
if $ \K 3 2 {\kappa} = 0 $ by Proposition~\ref{milnork} and the localization
sequence 
\[
\dots \to \K 3 2 {\kappa} \to \K 3 3 {\O} \to \K 3 3 F \to \K 2 2 {\kappa} \to \cdots 
\,.
\]

\begin{remark}
Notice that by construction (i.e., by compatibility of everything
we did with the localization of $ \O $ to $ F $), these maps
for $ i=2 $ or 3 fit into a commutative diagram
\begin{equation} \label{FOM3cd}
\begin{split}
\xymatrix{
H^i( \scriptM 3 {\O}) \ar[r] \ar[d] & \K 6-i 3 {\O} \ar[d]
\\
H^i( \scriptM 3 F ) \ar[r] & \K 6-i 3 F 
\,.
}
\end{split}
\end{equation}
We also note that it was proved in Remark~\ref{m2includes} that
the map $ M_2(\O) \to M_2(F) $ is injective.  Because we clearly
have that $ \OQ \to \FQ $ is an injection, this means that, in
degrees 2 and 3, $ \scriptM 3 {\O} $ injects into $ \scriptM 3 F $.
\end{remark}

\subsubsection{Construction of the complex $  \scriptM 3 {\CC'} $.} \label{M3CCconstruction}

In this subsection we imitate the definition of the complex $ \scriptM 3 C' $
in Section~\ref{scriptM3Cconstruction},
but using the complex $ \scriptM 3 {\O'} $ rather than $ \scriptM 3 F' $
in the top row.  
The advantage of using the complex $ \scriptM 3 {\CC'} $ (just like the advantage
of using any $ \O' $-complex over the corresponding $ F' $-complex)
is that the syntomic regulator gets the input it needs on the special
fibre of $ \CC' $.

We therefore put ourselves in the situation of Notation~\ref{case2},
so assume we have a number field $ k \subset K $,
a proper, smooth, irreducible curve $ \CC' $ over $ R'= \O \cap k $,
and that the generic fiber $ C' = \CC'\otimes_{R'} k $ is geometrically irreducible.
We put $ F' = k(C') $, and $ \O' $ the discrete valuation ring
in $ F' $ corresponding to the generic point of the special fibre
of $ \CC' $.  We have a commutative diagram as follows.
\begin{equation*}
\begin{split}
\xymatrix{
M_3 (\O') \ar[r]^-\dd \ar[d] & M_2 (\O') \tensor_\Q \OpQ \ar[r]^-\dd \ar[d]_{\partial_1} & \OpQ\tensor \bigwedge^2 \OpQ \ar[d] _{\partial_2}\cr
0 \ar[r] & \coprod_x \tM_2 (k(x)) \ar[r] ^-\dd& \coprod_x \bigwedge^2 k(x)_\Q^*
\,.
}
\end{split}
\end{equation*}
The $ \dd $'s in the top row are as in $ \scriptM 3 {\O'} $.
The vertical maps, and the map in the bottom row, are given by
the same formulae as before (see~\eqref{scriptM3Cconstruction}),
via the natural map $ \scriptM 3 {\O'} \to \scriptM 3 F' $
corresponding to the localization from $ \O' $ to $ F' $.

We let $ \scriptM 3 {\CC'} $ be the cohomological complex in degrees
1 through 4, given by the total complex associated to the double
complex in the commutative diagram above.
Note that therefore in particular, an element
$ \sum_i [u_i]_2 \tensor v_i $ in $ M_2(\O') \tensor \OpQ $ is
in $ H^2(\scriptM 3 {\CC'} ) $ if and only if it satisfies~\eqref{Ocond}
as well as, for every closed point $ x $ in $ C' $,
\begin{equation} \label{CCcond}
\sum_i \ord_x(v_i) [u_i(x)]_2 = 0
\end{equation}
in $ \tM_2(k(x)) $, with the convention that $ [0]_2 = [1]_2 = [\infty]_2 = 0 $.

The map to $ K $-theory is similar to the map for $ \scriptM 3 F' $,
but now we get
\[
H^2(\scriptM 3 {\CC'} ) \to H^2(\scriptM 3 {\O'} ) \to \K 4 3 {\O'}
\,,
\]
where the first arrow corresponds to forgetting the bottom row
in $ \scriptM 3 {\CC'} $.
In fact, because this is compatible with the localization to $ F' $ (i.e,
with the map $ \scriptM 3 {\O'} \to \scriptM 3 F' $), from~\eqref{Fboundcd}
we find that we have a commutative diagram
\begin{equation} \label{FOCcd}
\begin{split}
\xymatrix{
H^2(\scriptM 3 {\CC'} ) \ar[r] \ar[d] 
&
\K 4 3 {\CC'} \oplus \K 3 2 k \cup \OpQ \ar@{=}[d]
\\
H^2(\scriptM 3 C' ) \ar[r]
&
\K 4 3 C' \oplus \K 3 2 k \cup \FpQ 
\,,
}
\end{split}
\end{equation}
where the group on the right is contained in $ \K 4 3 {\O'} = \K 4 3 F' $,
and we used that
$  \K 4 3 {\CC'} \oplus  \K 3 2 k \cup \FpQ  = \K 4 3 C' \oplus  \K 3 2 k \cup \OpQ $
by Remarks~\ref{K4CtoK4Finjective} and~\ref{sameproducts}.
This proves that the top square in~\eqref{MAP} exists and commutes.

Note that in Theorem~\ref{main-thm1}(2),
the condition $ \partial_1(\alpha') = 0 $ on $ \alpha' $
in $ H^2(\scriptM 3 {\O'} ) $ is exactly that $ \alpha' $ satisfies~\eqref{CCcond},
hence lies in the subspace $ H^2(\scriptM 3 {\CC'} ) $. Therefore we have
proved the existence of $ \beta' $ in the theorem. Its uniqueness
is clear because the direct sum above gives an injection
$ \K 4 3 {\CC'} \to \K 4 3 {\O'} / \K 3 2 k \cup \OpQ $.

\begin{remark}
Just as in Remark~\ref{Fprojection}, we can consider
the projection 
\[
\K 4 3 {\CC'} \oplus \K 3 2 k \cup \OpQ \to \K 4 3 {\CC'}
\]
to get a map 
$ H^2(\scriptM 3 {\CC'} ) \to \K 4 3 {\CC'} $
as the composition
\[
H^2(\scriptM 3 {\CC'} ) \to \K 4 3 {\CC'} \oplus \K 3 2 k \cup \OpQ \to \K 4 3 {\CC'}
\,.
\]
\end{remark}

\subsubsection{Construction of the complex $ \Ccomp \bullet \O $.} \label{CcompOconstruction}
The remainder of the theorems in the introduction will be proved
in Section~\ref{sec:end}. The necessary calculations will in
fact depend heavily on the analogue of $ \Ccomp \bullet F $ for $ \O $,
$ \Ccomp \bullet \O $.

Because we are dealing with the two dimensional
scheme $ X_{\O} $, the localization sequence~\eqref{secondbasicFlocalization}
becomes a spectral sequence (cf.~\eqref{firstO-ss}):
\begin{alignat}{3}
\notag
\vdots\qquad\quad &\qquad\qquad\quad\vdots & \vdots\qquad
\\
\notag
\K 2 3 {\Xoloc;\bbox}  &\qquad \coprod_{t \in \SPO} \K 1 2 F &\qquad\coprod_{t \in \SPkk} \K 0 1 {\kk}
\\
\label{of-ss}
\K 3 3 {\Xoloc;\bbox}  &\qquad \coprod_{t \in \SPO} \K 2 2 F &\qquad\coprod_{t \in \SPkk} \K 1 1 {\kk}
\\
\notag
\K 4 3 {\Xoloc;\bbox}  &\qquad \coprod_{t \in \SPO} \K 3 2 F &\qquad\coprod_{t \in \SPkk} \K 2 1 {\kk}
\\
\notag
\vdots\qquad\quad &\qquad\qquad\quad\vdots & \vdots\qquad
\end{alignat}
converging to $ \K * 3 {X_\O;\bbox} \iso \K *+1 3 {\O} $.
Let us notice that $ \K 2 1 {\kk} $ and $ \K 3 1 {\kappa} $ are zero, and that the
exact localization sequence
\[
\cdots \to \K 3 1 {\kk} \to \K 3 2 {\O} \to \K 3 2 F \to \K 2 1 {\kk} 
\to \K 2 2 {\O} \to \K 2 2 F 
\to \dots
\]
tells us that $ \K 2 2 {\O} \subseteq \K 2 2 F $ and $ \K 3 2 {\O} \iso \K 3 2 F $.
Therefore we get an exact sequence
\[
0 \to \frac{\K 4 3 {\O} }{\K 3 2 {\O} \cup \OQ} \to
\K 3 3 {\Xoloc;\bbox} \to
\ker\left( \coprod_{t \in \SPO} \K 2 2 F \to \coprod_{t \in \SPkk} \K 1 1 {\kk} \right)
\,.
\]
In the middle row of the spectral sequence~\eqref{of-ss} above,
let $ B \subseteq \K 3 3 {\Xoloc;\bbox} $ be the inverse image of $ \coprod \K 2 2 {\O} $ (with
the coproduct over all of $ \SPO $).
Then we have a cohomological complex in degrees 1 and 2,
\begin{equation}
\label{Bcomplex}
AC_{(3)}(\O) : B \to \coprod_{t \in \SPO} \K 2 2 {\O} 
\,,
\end{equation}
and an isomorphism
\[
H^1(AC_{(3)}(\O)) \iso \frac{\K 4 3 {\O} }{\K 3 2 {\O} \cup \OQ } 
\]
and a map
\[
 H^2(AC_{(3)}(\O)) \to \K 3 3 {\O} 
\,.
\]

\begin{remark}
If $ \K 3 2 {\kk} = 0 $, or more generally, the map $ \K 4 3 F \to \K 3 2 {\kappa} $
is surjective, then from the exact localization sequence
\[
\dots \to \K 4 3 F \to \K 3 2 {\kk} \to \K 3 3 {\O} \to \K 3 3 F \to
\K 2 2 {\kk} \to \dots
\,,
\]
Proposition~\ref{milnork} and~eqref{scriptM2maps},
we see that
the map $ \coprod_{t \in \SPO} \K 2 2 {\O} \to \K 3 3 {\O} $,
and hence the map $ H^2( AC_{(3)}(\O) ) \to \K 3 3 {\O} $, are
surjective. 
\end{remark}

\begin{remark}
Because $ \K 1 2 F $ and $ \K 2 1 {\kk} $ are zero,
and $ \K 2 2 F \to \K 1 1 {\kappa} $ is surjective,
from~\eqref{of-ss} we get that there is an exact sequence
\[
\Ker\left( \coprod_{t \in \SPO}  \K 2 2 F \to \coprod_{t \in \SPkk} \K 1 1 {\kk} \right)
\to
\K 2 3 {X_{\O};\bbox} 
\to
\K 2 3 {X_{\O}^\loc;\bbox} 
\to
0
\,.
\]
If $ \K 3 2 {\kappa} $ is zero, or, more generally, the map $ \K 4 3 F \to \K 3 2 {\kappa} $
surjective, then Proposition~\ref{milnork} tells us that $ \coprod_{t \in \SPO} \K 2 2 {\O} $
surjects onto $ \K 2 3 {X_\O;\bbox} \iso \K 3 3 {\O} $, and we can
conclude that $ \K 2 3 {X_{\O,\loc};\bbox} $ is zero.
\end{remark}

Now we consider the acyclic subcomplex
\[
\IO \cup \K 2 2 {\O} \to \dd(\dots)
\]
of~\eqref{Bcomplex}, and quotient out to find a complex
\[
\Ccomp \bullet \O :
\Ccomp 1 \O \to \Ccomp 2 \O
\,,
\]
where 
\begin{equation}\label{C1}
\Ccomp 1 \O = \frac{B}{\IO \cup \K 2 2 {\O} } 
\end{equation}
and
\[
\Ccomp 2 \O =  \K 2 2 {\O} \tensor \O_\Q^* 
\,.
\]
We still have an isomorphism
\begin{equation}
\label{H1auxC}
H^1(\Ccomp \bullet \O ) \iso \K 4 3 {\O} / \K 3 2 {\O} \cup \OQ
\end{equation}
and a map
\begin{equation*}
H^2(\Ccomp \bullet \O ) \to \K 3 3 {\O} 
\,,
\end{equation*}
which by Proposition~\ref{milnork} and~\eqref{scriptM2maps}
is a surjection if $ \K 4 3 F \to \K 3 2 {\kappa} $ is
surjective, e.g., if $ \K 3 2 {\kappa} = 0 $.

Observe that if $ g $ is in $ \SPO $, and $ f $ is in $ \OQ $, then $ \symb g f  $
is in $ \Ccomp 1 \O $, and has boundary $ \{(1-g)^{-1}, f\} \tensor g = - \{ (1-g) , f \} \tensor g $
in $ \Ccomp 2 \O $.  The condition for $ \sum_i \symb {g_i} {f_i} $
to be in $ H^1(\Ccomp \bullet \O ) $ is therefore that
\begin{equation*}
\sum_i \{ 1-g_i , f_i \} \tensor g_i = 0
\end{equation*}
in  $ \Ccomp 2 \O = \K 2 2 {\O} \tensor \OQ $.

Note that because the construction of the spectral sequence in
\eqref{of-ss} is compatible with localizing the base from $ \O $ to $ F $
and enlarging the coproduct from being over $  \SPO $ to $ \SPF $
(in which case it becomes the localization sequence in~\eqref{secondbasicFlocalization}),
and that $ \IO $ is contained in $ \I $, and $ \K 2 2 {\O} \subseteq \K 2 2 F $,
we have an obvious map of complexes,
\[
\Ccomp \bullet \O \to \Ccomp \bullet F
\,,
\]
which fits into the commutative diagram
\begin{equation} \label{FOauxcd}
\begin{split}
\xymatrix{
H^1(\Ccomp \bullet \O ) \ar[r]\ar[d] & \K 4 3 {\O} / \K 3 2 {\O} \cup \OQ \ar[d]
\\
H^1(\Ccomp \bullet F ) \ar[r] & \K 4 3 {F} / \K 3 2 {F} \cup \FQ
\,,
}
\end{split}
\end{equation}
and similarly for $ H^2 $.

Finally, we have a commutative diagram
\begin{equation*}
\xymatrix{
M_3(\O) \ar[r]\ar[d]
&
M_2(\O) \tensor \OQ \ar[r]\ar[d]
&
\OQ \tensor \bigwedge^2\OQ\ar[d]
\\
0 \ar[r]
&
\Ccomp 1 \O \ar[r]
&
\Ccomp 2 \O 
}
\end{equation*}
as follows. We map $ [u]_2 \tensor v $ to $ [u]_2 \cup v $, and $ u\tensor v\wedge w $
to $ \{u, v \} \tensor w - \{u, w \} \tensor v $.  
This gives rise to a commutative diagram
\begin{equation}
\label{MCO}
\begin{split}
\xymatrix{
H^2(\scriptM 3 {\O} ) \ar[r]\ar[d] & \K 4 3 {\O} \ar[d]
\\
H^1(\Ccomp \bullet {\O} ) \ar[r] & \K 4 3 {\O} / \K 3 2 {\O} \cup \OQ 
\,,
}
\end{split}
\end{equation}
which is the bottom left square of~\eqref{MAP}.
Obviously, the two diagrams above are compatible with~\eqref{shiftedcomplexmap}
and~\eqref{auxdiagram} under the localization from $ \O $ to $ F $.

\subsubsection{Construction of the complexes $ \tildescriptM 2 {\O} $ and $ \tildescriptM 3 {\O} $.} \label{tM2O}
For $ n=2 $ and 3, let
$ N_n(\O) = \langle [u]_n +(-1)^n [u^{-1}]_n \text{ with } u \text{ in } \SPO \rangle_\Q \subseteq M_n(\O) $.
Consider the subcomplex of $ \scriptM 2 {\O} $ given by
\[
 N_2(\O) \to \dd(\dots)
\,.
\]
Because the corresponding subcomplex~\eqref{N2Fcomplex} of $ \scriptM 2 F $
is acyclic and the natural map $M_2(\O) \to M_2(F)  $ is an injection (see Remark~\ref{m2includes}),
this subcomplex is acyclic.
The second term is $ \Sym^2(\OQ) $, and the resulting quotient complex of
$ \scriptM 2 {\O} $ is
\begin{equation} \label{tildescriptM2O}
\tildescriptM 2 {\O} : \tM_2(\O) \to \bigwedge^2 \OQ
\,,
\end{equation}
with $ \tM_2(\O) = M_2(\O) / N_2(\O) $,
and $ \dd [u]_2 = (1-u) \wedge u $.

Because $ \tildescriptM 2 {\O} $ is quasi isomorphic to $ \scriptM 2 {\O} $
we have maps
\[
H^i(\tildescriptM 2 {\O} ) \to \K 4-i 2 {\O}
\,.
\]
For $ i=1 $ this is again an injection.
There is a map $ \tildescriptM 2 {\O} \to \tildescriptM 2 F $
obtained by localizing the construction from $ \O $ to $ F $,
and for $ i=1,2 $ a commutative diagram
\[
\xymatrix{
 H^i(\tildescriptM 2 {\O}) \ar[d]& H^i(\scriptM 2 {\O}) \ar[r] \ar[l]_-{\simeq}\ar[d] & \K 4-i 2 {\O} \ar[d]
\\
 H^i(\tildescriptM 2 F )         & H^i(\scriptM 2 F )   \ar[r] \ar[l]_-{\simeq}       & \K 4-i 2 F 
\,.
}
\]
In this diagram for $ i=1 $ the central vertical map is injective
by the discussion in Remark~\ref{m2includes}.
Hence the same holds for the map $ H^1(\tildescriptM 2 {\O}) \to H^1(\tildescriptM 2 F ) $,
the map $ \tM_2(\O) \to \tM_2(F) $ is an injection, and
$ \tildescriptM 2 {\O} $ is a subcomplex of $ \tildescriptM 2 F $.

By Remark~\ref{m2includes}, in the commutative diagram
\begin{equation*}
\begin{split}
\xymatrix{
 M_3(\O) \ar[r] \ar[d] & M_2(\O)\otimes\OQ \ar[r]\ar[d] & \OQ\otimes\bigwedge^2\OQ \ar[d]
\\
 M_3(F)  \ar[r]        & M_2(F) \otimes\FQ \ar[r]       & \FQ\otimes\bigwedge^2\FQ
}
\end{split}
\end{equation*}
the two right-most maps are injective. (If we knew (as part of
the rigidity conjecture) that $ H^1(\scriptM 3 {\O} ) \to  H^1(\scriptM 3 F ) $
were injective, then this would also hold for the left-most map.)
We can quotient out the complex $ \scriptM 3 {\O} $ in the first
row by the subcomplex
\begin{equation*}
\xymatrix{
 N_3(\O) \ar[r]
&
 N_2(\O) \otimes \OQ \ar[r]
&
\dd(\dots)
\,,}
\end{equation*}
which maps to the subcomplex~\eqref{N3Fcomplex} of the second row.
We saw earlier that $ d : N_2(\O) \to \Sym^2(\OQ) $
is an isomorphism, so as in the proof of~\cite[Corollary~3.22]{Jeu95}
one sees that this subcomplex is acyclic in degrees 2 and~3.
The quotient complex is
\begin{equation*}
 \tildescriptM 3 {\O} : \tM_3(\O) \to \tM_2(\O) \otimes \OQ \to \bigwedge^3\OQ
\,,
\end{equation*}
where $ \tM_3(\O) = M_3(\O) / N_3(\O) $,
and the natural map $ \tildescriptM 3 {\O} \to \tildescriptM 3 F $
is an injection in degrees~2 and~3 because, as we saw earlier,
$ \tM_2(\O) $ injects into $ \tM_2(F) $.
Still denoting the class of $ [x]_i $ with $ [x]_i $, the maps
are now given by 
\[
\dd [u]_3 = [u]_2 \tensor u
\]
and
\begin{equation}\label{otild5}
\dd [u]_2 \tensor v = (1-u) \wedge u \wedge v
.
\end{equation}
Using~\eqref{scriptMOmaps} we see that for $ i=2,3 $ we have a commutative diagram
\begin{equation*}
\begin{split}
\xymatrix{
H^i( \tildescriptM 3 {\O}) \ar[d] & H^i( \scriptM 3 {\O} ) \ar[d] \ar[l]_-{\simeq} \ar[r] & \K 6-i 3 {\O} \ar[d]
\\
H^i( \tildescriptM 3 F ) &  H^i( \scriptM 3 F ) \ar[l]_-{\simeq}\ar[r] & \K 6 4-i F 
\,.
}
\end{split}
\end{equation*}

\subsection{A diagram.}

For the convenience of the reader, we give a commutative diagram
summarizing the cohomology groups of  most of the complexes introduced, and the maps.  We have
kept the lay-out of the diagram in the same spirit as the relativity
in the plane.  Note that the outer square is only relevant in
the situation of Notation~\ref{case2}, and that we may replace
$ F $ and $ \O $ with $ F' $ and $ \O' $ in this case.

The top half of this diagram is the top of the one in~\eqref{MAP}.
The vertical maps correspond to the maps from
constructions over $ \O $ to the corresponding constructions
over $ F $.  The horizontal maps are the maps on cohomology of
complexes constructed in the previous subsections, and the 
diagonal maps correspond to the maps in 
\eqref{shiftedcomplexmap},
\eqref{FOM3cd}, \eqref{FOCcd} and~\eqref{FOauxcd}.
\begin{equation}
\label{BIGCD}\notag
\xymatrix
@*+<1pt>
@!=40pt
{
*[r]{H^2(\scriptM 3 {\CC'} )} \ar[rrrrr]\ar[ddddd]\ar[rd]&&&&& *[l]{\K 4 3 {\CC'} \oplus \K 3 2 k \cup \OpQ }\ar[ddddd]\ar[ld]
\\
& H^2(\scriptM 3 {\O} ) \ar[rrr]\ar[ddd]\ar[rd]&&& \K 4 3 {\O}  \ar[ddd]\ar[ld]
\\
& & H^1(\Ccomp \bullet \O ) \ar[r]\ar[d]\ar@{}[dr]|{{\textstyle{(\ref{BIGCD})}}}& \frac{\K 4 3 {\O} }{\K 3 2 {\O} \cup \OQ }\ar[d]
\\
& &  H^1(\Ccomp \bullet F ) \ar[r] & \frac{\K 4 3 F }{ \K 3 2 F \cup \FQ } 
\\
&  H^2(\scriptM 3 F ) \ar[rrr]\ar[ru] &&& \K 4 3 F \ar[lu]
\\
*[r]{H^2(\scriptM 3 C' ) } \ar[rrrrr]\ar[ru] &&&&& *[l]{\K 4 3 C' \oplus \K 3 2 k \cup \FpQ }\ar[lu]
}
\end{equation}
Note that by Remarks~\ref{K4CtoK4Finjective} and~\ref{sameproducts}
the rightmost vertical map is an isomorphism.

\section{The classical case} \label{sec:classical}

In Proposition~\ref{blueprop} below, we
rephrase the results in Theorem~4.2 and Remarks~4.3 and ~4.5 of \cite{Jeu96}
in a way that resembles the formulae in Theorems~\ref{main-thm3}
and~\ref{main-thm4}(1) (see Remark~\ref{compremark}
for some thoughts on this comparison).
In fact, Sections~\ref{sec:trip-loc} and~\ref{sec:trip-glob}
grew out of attempts to obtain syntomic analogues of those results
of loc.\ cit., but the resulting formulae
seem to be less flexible than the classical ones so we rephrase
the latter.

In the next proposition, we let
$ \Hdr^1(F, \R(2)) = \lim_{\to \atop U} \Hdr^1 (U, \R(2)) $
where the limit is over $ U $ with $ \Can\setminus U $ finite,
and similarly for $ H_\mathcal{D}^2(F, \R(3)) $. 
Here $ \R(m) = (2 \pi i)^m \R \subset \C $.
If $ \omega $ is holomorphic on $ \Can $, then by \cite[Proposition~4.6]{Jeu96}
one has a well-defined map $ \Hdr^1(F, \R(2)) \to \C $ by taking a representative
$ \beta $ of a class in $ \Hdr^1(F, \R(2)) $ satisfying~(9) of
loc.\ cit., and computing $ \int_\Can \beta \wedge \omega $.

\begin{proposition} \label{blueprop}
Let $ C $ be a smooth, proper, irreducible curve over $ \C $ with function field
$ F=\C(C) $, and let $ \Can $ be the analytic manifold associated to $ C(\C) $.
For a holomorphic 1-form $ \omega $ on $ \Can $,
the maps
\begin{alignat*}{1}
 \myphi :  M_2(F) \otimes \FQ & \to \C
\\
 [g]_2 \otimes f & \mapsto -4 \int_{C_{an}}  \log|f| \log|g| \dd \log|1-g|  \wedge \omega
\end{alignat*}
and
\begin{alignat*}{1}
 \tmyphi : \tM_2(F) \otimes \FQ & \to \C
\\
 [g]_2 \otimes f & \mapsto 
 -\frac{8}{3} \int_\Can \log|f| (\log|g|\dd\log|1-g| - \log|1-g|\dd\log|g|) \wedge \omega
\end{alignat*}
are well-defined, and induce maps $ H^2(\scriptM 3 F ) \to \C $
and $ H^2(\tildescriptM 3 F ) \to \C $ respectively.
Moreover, with $ \regC : \K 4 3 F \to H_\mathcal{D}^2(F, \R(3)) \simeq \Hdr^1(F, \R(2)) $ the Beilinson
regulator map, the compositions
\begin{equation*}
\xymatrixcolsep{4pc}
\xymatrix{
  H^2(\scriptM 3 F ) \ar[r]^-{\eqref{HM3Fmap}} & \K 4 3 F \ar[r]^-{\int_\Can \regC(\cdot)\wedge\omega} & \C
}
\end{equation*}
and
\begin{equation*}
\xymatrixcolsep{4pc}
\xymatrix{
  H^2(\tildescriptM 3 F ) \ar[r]^-{\eqref{HtildeM3Fmap}} & \K 4 3 F \ar[r]^-{\int_\Can \regC(\cdot)\wedge\omega} & \C
}
\end{equation*}
coincide with these induced maps.
\end{proposition}

\begin{proof}
Since $ d \otimes \id : M_2(F) \otimes \FQ \to \FQ \otimes \FQ \otimes \FQ $
maps $ [g]_2 \otimes f $ to $ (1-g) \otimes g \otimes f $, $ \myphi $ is well-defined.
That is induces the stated map on $ H^2(\scriptM 3 F ) $ and that this induced map
has the stated property 
follows from Proposition~3.2 and (the proof of) Theorem~4.2 of~\cite{Jeu96}.
(The condition in loc.\ cit.\ that $ C $ is defined over a number field
is not used in the proof of Theorem~4.2. The same holds
for the condition with respect to complex conjugation on $ \omega $,
which guaranteed only that the value of the integral was in $ \R(1) \subset \C $.)

Similarly, $ d \otimes \id : \tM_2(F) \otimes \FQ \to (\bigwedge^2 \FQ) \otimes \FQ $
maps $ [g]_2 \otimes f $ to $ (1-g)\wedge g \otimes f $, so $ \tmyphi $
exists.
Using a limit version of Stokes theorem we may add
$ 0= \int_\Can \dd ( \alpha \wedge \omega ) $
for $ \alpha = - \frac{4}{3} \log|g|\log|1-g|\log|f| $, so we map $ [g]_2 \otimes f $
to
\begin{equation*}
-\frac{4}{3} \int_\Can (3 \log|f| \log|g| \dd \log|1-g| 
         + \log|1-g| (\log|g| \dd \log|f| - \log|f| \dd \log|g| )) \wedge \omega
\,.
\end{equation*}
So $ \tmyphi $ and $ \myphi $ coincide on the kernel of the map $ M_2(F) \otimes \FQ \to \FQ \otimes (\bigwedge^2 \FQ) $
that maps $ [g]_2 \otimes f $ to $ (1-g) \otimes (g \wedge f) $.
That $ \tmyphi $ induces a map on 
$ H^2(\tildescriptM 3 F ) $ with the desired property then follows
from the corresponding statements for~$ \myphi $.
\end{proof}

\begin{remark}
The Bloch-Wigner dilogarithm $ D(z) : \P_\C^1 \setminus\{0,1,\infty\} \to (2 \pi i)\R \subset \C $
satisfies $ \dd D(z) = \log|z|\dd i \arg(1-z)-\log|1-z|\dd i \arg(z) $
and extends to a continuous function on $ \P_\C^1 $. It is the
function in the classical case that corresponds to $ \Lmod 2 (z) $
in the sense that they have similar functional equations, e.g.,
$ D(z) + D(z^{-1}) = 0 $.
Because $ \dd \log(g) \wedge \omega = \dd \log(1-g) \wedge \omega = 0$,
we find $  \dd(\Pzag 2 (g) \log|f|\omega) $ equals
\begin{equation*}
 \Pzag 2 (g) \dd\log|f| \wedge\omega +
\log|f| (\log|1-g|\dd\log|g|-\log|g|\dd \log|1-g| ) \wedge \omega
\,.
\end{equation*}
Hence $ \tmyphi $ is also given by mapping $ [g]_2 \otimes f $
to $ - \frac{8}{3} \int_\Can \log|f| D(g) \omega$.
\end{remark}

\section{Coleman integration}
\label{sec:coleman}

In this short section we briefly discuss Coleman's integration
theory in the one dimensional case only. The interested reader may
refer to~\cite{Bes98a} for more details.

Coleman theory is done on wide open spaces in the sense of
Coleman~\cite{Col-de88}. In general these are the overconvergent
spaces described in section~\ref{sec:regulators}. In the
one-dimensional case these can be described concretely in the
following way. Let $X$ be a curve over $\Cp$ with good reduction
(there is a minor assumption that it is obtained by extension
of coefficients from a curve over a complete discretely
valued subfield,
which will always be satisfied in our case). The rigid
analytic space $X(\Cp)$ is set-theoretically decomposed as the union
$X=\cup_x U_x $ where $x$ varies over the points in the reduction of $X$
and $U_x$ is the residue disc (tube in the language of Berthelot) of
points reducing to $x$. By the assumption of good reduction each
residue disc is isomorphic to a disc $|z|<1 $. A wide open space $U$ is
obtained from $X$ by fixing a finite set of points $S$ in the
reduction and throwing
away the discs inside the residue discs $U_x$, $x\in S$, isomorphic to $|z|<r$
for arbitrarily large $r<1 $. $U$ should be thought of as the inverse
limit of the correspondeing spaces $U_r$.

Coleman theory associates to $U$ the $\Cp$-algebra $\acol(U)$ and the
$\acol(U)$-modules $\ocol^i(U)$ with differentials forming a
complex. The key property is that this complex is exact at the
one and zero forms, i.e, there is an exact sequence
\begin{equation*}
  0\to \Cp \to \acol(U) \to \ocol^1(U) \to \ocol^2(U)\,.
\end{equation*}
The space $\ocol^1(U) $ contains the space $\Omega^1(U) $ of
overconvergent forms on $U$, i.e., those forms that
are rigid analytic on some $U_r$. Similarly, the space $\acol(U)  $
contains the space $A(U)$ of overconvergent functions. The
differential extends the usual differential on the subspaces.

The whole picture extends to higher dimensions.
We shall only need the case where $U$ is one-dimensional. In this case
the space $\ocol^2(U)$ is already $0$.

Coleman functions may be interpreted as locally analytic functions on
$U$. More precisely, again in the one-dimensional case, for $x\notin S$,
the intersection of the residue disc $U_x$ with $U$ is $U_x$, while for
$x\in S$ it is
an annulus $e_x$ isomorphic to an annulus of the form $r<|z|<1 $. A
Coleman function is analytic on each $U_x$ and is
a polynomial algebra $A(e_x)[\log(z)]$ where $z$ is a local parameter
on $U_x$ (here, there is an implicit global choice of a branch of the $p$-adic polylogarithm).

We define the space $\acola (U)$ to be the inverse image of
$\Omega^1(U)\subseteq \ocol^1(U)$ under the differential $d$. The space of
differentials $\Ocola(U)$ is the product $\acola(U)\cdot \Omega^1(U) $.

If $\omega\in \Omega^1(U_r)$ and $y,z\in U_r $ the integral $\int_z^y
\omega$ is clearly well-defined as $f(y)-f(z)$ where $f\in \acol(U_r) $ and $d
f= \omega$. It is a basic property of Coleman integration that if
$X,U,\omega,z,y$ are all defined over the complete subfield
$K$, then so is the integral $\int_z^y \omega$.

\section{Regulators} \label{sec:regulators}

In this section we compute the regulator on
$ \Ccomp {1} \O  $ in (modified) syntomic cohomology.
In case the element lies in the subspace $ H^1( \Ccomp {\bullet} \O ) $,
we also explain how we wish to interpret the cup product of this
regulator with the cohomology class of a form $\omega$ of the second
kind on $C$, and what are the obstacles for doing so, thus paving
the way for constructions in the next sections.

We first write down the relevant spaces and the (modified) syntomic
complexes computing their cohomology. For the full story the reader
should consult~\cite{Bes98a}.

We begin with a smooth proper relative curve $\CC/R$. Related to that
is the space $X_{\CC} := \PP_{\CC}^1\setminus\{t=1\}$.
The superscript loc will denote various
localizations, obtained by removing the image of a finite number of $R$-sections. We note that the computations in this section can be done after a finite base change, so we may easily get from more general localizations into this situation by further localization.
We shall use localizations $\CC^{\loc}$ of $\CC$ or $X_{\CC}^{\loc} $ of
$X_{\CC}$.
If the localization is non-trivial, and we may and do assume this,
then all localized schemes are affine.

Our goal is to compute the syntomic regulator $ \K 4 3 {\CC} \to
\hsyn^2(\CC,3)$. According to~\cite[Proposition~8.6.3]{Bes98a} there is an
isomorphism, commuting with the regulator, $\hsyn^2(\CC,3)
\xrightarrow{\simeq} \htms^2(\CC,3)$, where $\htms$ is the 
Gros style modified rigid syntomic cohomology, in the sense of
loc.~cit.
From now on we shall therefore concentrate on modified syntomic
cohomology. We shall refer to it simply as syntomic cohomology.

Let us recall one of the possible models for modified syntomic cohomology for
affine schemes. Let $A$ be an affine $R$-scheme. We assume we have an
open embedding $A\inject \overline{A}$, where $\overline{A}$ is proper.
From the embedding $A\inject \overline{A}$ one obtains the overconvergent
space $A^\dagger$. This space can be made sense out in
Grosse-Kl{\"o}nne's theory of overconvergent spaces~\cite{Grosse00} as
the space whose affine ring, $\O(A^\dagger)$, is the weak completion,
in the sense of
Monsky-Washnitzer, of $\O(A)$. However, here we shall simply think of
$A^\dagger$ formally as the inverse system of strict neighborhoods of
the special fiber of $A$ in that of $\overline{A}$

We further assume that we
have an $R$-linear endomorphism $\phi:A^\dagger \to A^\dagger$ whose reduction
is a power of Frobenius, say of degree $q=p^r$. We call $\phi$ a
Frobenius endomorphism. Standard results (\cite[Thm A-1]{Col85}
or~\cite[Thm 2.4.4.ii]{Put86}) imply
one always has such $\phi$.

With the above data, we have
\begin{equation*}
  \htms^n(A,j) = H^n(\textup{MF}(F^j \Omega^\bullet (A^\dagger)
  \xrightarrow{1-\phi^\ast/q^j}\Omega^\bullet (A^\dagger)))\,.
\end{equation*}
Here, the filtration is the stupid filtration on the space of
differentials and $\textup{MF}$ denotes the mapping fiber (Cone shifted
by $ -1 $).
To be more precise, one really needs to take the limit of these
cohomology groups with respect to powers of $\phi$, in a way explained
in~\cite{Bes98a}, but it is also explained there that one can ignore
this point.

The cohomomology groups $\htms$ are in fact functorial with respect to
arbitrary maps of schemes. This functoriality is not at all obvious
from the definition except in the case where the maps extend to the
dagger spaces and commute with $\phi$. Fortunately, this will
always be the case for us. In this situation, one may also construct
relative cohomology in the obvious way (the reader is advised to look
at~\cite[Section~\ref{sec:down}]{Bes-deJ98} 
for constructions of
complexes computing
relative syntomic cohomology).

To end this general review we recall that the corresponding syntomic
regulator is defined by the formula
\begin{equation}\label{eq:regfunc}
  f\in \O(A)^* \subset K_1(A) \mapsto (\dlog(f),\log(f_0)/q)\in
  \htms^1(A,1)\,,
\end{equation}
where $f_0 = f^q/\phi^\ast(f)$ and has the property
that $\log(f_0) $ is in $ \O(A^\dagger)$.
We also recall from\cite[Definition~6.5]{Bes98a} that the cup
product $\htms^\bullet (A,i) \times \htms^\bullet (A,j) \to
\htms^\bullet(A,i+j)$ is given by 
\begin{equation}
  \label{eq:cupprod}
  \begin{split}
    (\omega_1,\epsilon_1)\cup (\omega_2,\epsilon_2) = &\Big(\omega_1\wedge
    \omega_2,\\& 
    \epsilon_1 \wedge \left(\gamma +(1-\gamma)
     \frac{\phi^\ast}{q^j}\right) \omega_2\\
    +(-1)^{\deg(\omega_1)}
    &\left(\left((1-\gamma)+ \gamma
        \frac{\phi^\ast}{q^i}\right)\omega_1\right)
    \wedge \epsilon_2\Big)\,.
  \end{split}
\end{equation}
for some constant $\gamma$, which can be taken arbitrarily (producing
homotopic products).

We now write these constructions for the affine schemes we are
considering. To simplify notation we write $U$ for
$(\CC^\loc)^\dagger$, $U'$ for
$(X_{\CC}^{\loc})^\dagger$, and $X_U$ for
$(X_{\CC^{\loc}})^\dagger$. We may localize such that
$U'\subset X_U$. We fix a Frobenius endomorphism $\phi:U\to U$. We can
then take the Frobenius endomorphism for $X_U$ to be the product of
$\phi$ with the map $t\mapsto t^q$ and for $U'$ the restriction of
this endomorphism to $U'$. Since $t\mapsto t^q$ fixes $0$ and $\infty$
we can use the embedding of $U$ in $U'$ at $t=0$ and $t=\infty$. With
this we have the following models for syntomic cohomology.
\begin{equation}
  \label{eq:syntomicmodel}
  \htms^i(X_{\CC}^{\loc},i)=\frac{
    \{(\omega,\epsilon),\; \omega\in \Omega^1(U'),\; \epsilon\in
    \Omega^{i-1}(U'),\; d\omega=0,\; d \epsilon=
    \left(1-\frac{\phi^\ast}{q^i}\right)(\omega)\}
    }{
      \{(0,d \epsilon),\; \epsilon\in \Omega^{i-2}(U')\}
    }
\end{equation}
For $i=1,2$. Now, for the relative one we can write, by throwing away
terms which are forced to be $0$,
\begin{equation}
  \label{eq:reltwo}
    \htms^2(X_{\CC}^{\loc},\bbox,2)= \frac{
      \left\{(\omega,\epsilon,\epsilon_\infty,\epsilon_0),\;
        \begin{aligned}
          &\omega\in
          \Omega^2(U'),\; \epsilon\in \Omega^1(U'),\;\epsilon_s \in
          \O(U),s=0,\infty,\\
          &d\omega=0,\; d \epsilon=
          \left(1-\frac{\phi^\ast}{q^i}\right)(\omega),\; d \epsilon_s =
          \epsilon|_{\{t=s\}},\;s=0,\infty
      \end{aligned}
      \right\} }{ \left\{\left(0,d
      \epsilon,\epsilon|_{\{t=\infty\}},\epsilon|_{\{t=0\}}\right),\;
      \epsilon\in \O(U')\right\}}
\end{equation}
The map between $\htms^2(X_{\CC},\bbox,2)$ and $\htms^i(X_{\CC},i)$
remembers only $\omega$ and $\epsilon$. Since $U'$ is two dimensional
and therefore does not support forms of degree $3$, we also have
\begin{equation}
  \label{eq:relthree}
  \begin{split}
    &\htms^3(X_{\CC}^{\loc},\bbox,3)=\\ &\frac{
      \{(\epsilon,\epsilon_\infty,\epsilon_0) ,\; \epsilon\in
      \Omega^2(U'),\;\epsilon_s \in
      \Omega^1(U),s=0,\infty, \; d \epsilon=0,\; d \epsilon_s =
      \epsilon|_{\{t=s\}},\;s=0,\infty\} }{ \left\{\left(d
      \epsilon,\epsilon|_{\{t=\infty\}},\epsilon|_{\{t=0\}}\right),\;
      \epsilon\in \Omega^1(U')\right\}}
\end{split}
\end{equation}
If we replace $U'$ by $X_U$ we obtain a model for
$\htms^3(X_{\CC^{\loc}},\bbox,3)$

The last model is
\begin{equation}
  \label{eq:Cmodel}
  \htms^2(\CC^{\loc},3)=\frac{\{\epsilon\in \Omega^1(U),\;
  d\epsilon=0\}}{\{d \epsilon\;,\epsilon\in \O(U)\}}
\end{equation}
This is of course just the first de Rham cohomology of $U$. However,
the ``correct'' isomorphism with this cohomology is not the
obvious one but rather the one twisted by $1-\phi^\ast/q^3$, i.e.,
\begin{equation}
  \label{eq:twist}
  \Hdr^1(U) \to  \htms^2(\CC^{\loc},3),\;
  [\eta] \mapsto [(1-\phi^\ast/q^3) \eta]
\end{equation}
(for an explanation of this see~\cite[Proposition
10.1.3]{Bes98a}). Here, and in what
follows, we denote the cohomology class of an element in square brackets.

At this point, we are able to make more precise the definition of the
$p$-adic regulator for open curves that was hinted to in the
introduction before stating Theorem~\ref{main-thm2}. As explained
there, for each $U$ as above, one has a canonical
projection $\Hdr^1(U) \xrightarrow{\canproj}
\Hdr^1(C/K)$. This is the unique Frobenius equivariant splitting of
the natural restriction map in the other direction. These projections
are compatible in the obvious way when restricting to a smaller $U$.

\begin{definition}\label{regdefined}
The regulator map
\begin{equation*}
    \reg: \K 4 3 {\CC^\loc } \to \Hdr^1(C/K)
\end{equation*}
is the composition
\begin{equation*}
  \K 4 3 {\CC^\loc} \to \Hdr^1(U/K) \xrightarrow{\canproj} \Hdr^1(C/K)\,.
\end{equation*}
Using the compatibility of the maps $\canproj$ mentioned above
for all possible $ \CC^\loc $,
from $ \K 4 3 {\O} = \dirlim_{\CC^\loc} \K 4 3 {\CC^\loc} $
(see \cite[Proposition~2.2]{Qui67} or \cite[Lemma~5.9]{Sri96})
we also obtain a well defined regulator map
\begin{equation*}
    \reg: \K 4 3 {\O } \to \Hdr^1(C/K)
    \,.
\end{equation*}
\end{definition}

We need a formula for the cup product 
$\htms^2(X_{\CC}^{\loc},\bbox,2)\times
\htms^1(X_{\CC}^{\loc},1) \to
\htms^3(X_{\CC}^{\loc},\bbox,3)$ in terms of the
models~\eqref{eq:reltwo},~\eqref{eq:syntomicmodel} and~\eqref{eq:relthree}
respectively. Using the formula for a cup product between a cone and a
complex and~\eqref{eq:cupprod} with $\gamma=0$ we find the
following formula:
\begin{equation}
  \label{eq:relcup}
  (\omega,\epsilon,\epsilon_\infty,\epsilon_0) \cup
  (\eta,h) = ( h \omega + \epsilon\wedge \frac{\phi^\ast}{q} \eta ,
  \epsilon_\infty \eta, \epsilon_0 \eta)\,.
\end{equation}

Suppose now that $f$ and $g$ are in $\O^\ast(\CC^{\loc})$ (see
Subsection~\ref{CcompOconstruction}).
To compute the regulator of $[g]_2\cup (f)$
we start with $ [g]_2$ in $ \kbb(X_{\CC}^{\loc},\bbox)$.
It maps in $\kbb(X_{\CC}^{\loc})$ to 
$-\tmg \cup (1-g)$, by pulling back along g the corresponding result
for the universal elements~\cite[Proposition~\ref{was6.7}]{Bes-deJ98}.

\begin{lemma}\label{firstreg}
  We have in  $\htms^2(X_{\CC}^{\loc},2)$ that $- ch(\tmg \cup (1-g)) =
 (\omega_g,\epsilon_g)$, in the model~\eqref{eq:syntomicmodel} with
\begin{align*}
  \omega_g &= -\omg \\ \intertext{and} \epsilon_g &= \epsg
\end{align*}
\end{lemma}

\begin{proof}
  This follows from the formula~\eqref{eq:regfunc} for the regulators
  of functions, the
  compatibility of $ch$ with cup products and the cup product
  formula~\eqref{eq:cupprod}.
\end{proof}

  In what follows, the notation $[a_1,\dots,a_i]$ will denote
the class of $ (a_1,\dots,a_i) $ in~\eqref{eq:reltwo} or~\eqref{eq:relthree},
depending on if $ i=3 $ or~4.

\begin{proposition}\label{3.10}
  We have in $\htms^2(X_{\CC}^{\loc},\bbox,2)$, using the
  model~\eqref{eq:reltwo}, 
  \begin{equation*}
    ch([g]_2)=[\omega_g,\epsilon_g,0,\Theta(g)]
  \end{equation*}
  where
  \begin{equation}\label{Theta}
    d \Theta(g)= \epsilon_g|_{t=0} = \epsgi \,.
  \end{equation}
\end{proposition}

\begin{proof}
We are looking for a closed four-tuple, whose first two coordinates
represent the cohomology class of $(\omega_g,\epsilon_g)$.
It is easy to see that we may assume that the first two coordinates
are indeed $(\omega_g,\epsilon_g)$. Then the closedness condition
implies that the differentials of the next two coordinates
give the restriction to $t=\infty$ and $t=0$ respectively of
$\epsilon_g$. These are respectively $0$ and
$\epsilon_g|_{t=0}$, so the result is clear.  
\end{proof}

\begin{remark}
1. One can show that there exist a
function $\Theta$ on $\PP^1$ such that $\Theta(g)$ is indeed the
composition of $\Theta$ and $g$, but we shall not need to use this.\\
2. The determination of the regulator at this stage is incomplete,
since we have only determined $\Theta(g)$ up to a constant. It will
turn out that for the regulator computation this is irrelevant. For
the computation of the boundary this becomes much trickier. We in fact
failed to determine the boundary of the regulator directly. When we
need this towards the end of Section~\ref{sec:end} for the proof
of Theorem~\ref{main-thm1}, we shall use a trick to overcome this
difficulty, which in particular forces us to assume working over a
number field at that stage.
\end{remark}

\begin{proposition}
  The regulator of  $\symb{g}{f}$ in
  $\htms^3(X_{\CC}^{\loc},\bbox,3) $ is
  represented by the following element in the model~\eqref{eq:relthree}
  \begin{equation*}
   \epsilon(g,f):=\left(\ovq \log f_0 \omega_g + \ovq
    \epsilon_g\wedge \phi^\ast \dlog
    f,0,\ovq \Theta(g) \phi^\ast \dlog f\right)
  \end{equation*}

\end{proposition}
\begin{proof}
This follows again from the compatibility of the regulator with cup
products and from the formulas for the cup product in relative
syntomic cohomology~\eqref{eq:relcup}.
\end{proof}

Suppose now that $\alpha=\sum_i [g_i]_2 \cup (f_i)$ belongs to
\[
H^1(\Ccomp \bullet \O ) \iso \K 4 3 {\O} / \K 3 2 {\O} \cup \OQ 
\,,
\]
see~\eqref{H1auxC}.
Note that $ \alpha $ is only determined
up to an element in $ \IO \cup \OQ $, see~\eqref{Bcomplex} and~\eqref{C1}.
A term in the latter space consists explicitly of elements of the form
\begin{equation}\label{mydelta}
  \delta=\sum_j \delta_{1,j} \cup \delta_{2,j}
\end{equation}
 with $\delta_{1,j}\in
\K 1 1 {X_\CC^\loc,\bbox} $ and $\delta_{2,j} \in \K 2 2 {\CC^\loc} $,
for all possible localizations.
Therefore, for an appropriately
chosen $\CC^\loc$, there
exists $\beta\in \K 3 3 {X_{\CC^\loc},\bbox} $
whose restriction to $(X_{\CC}^\loc,\bbox)$ is $\alpha+\delta$, where 
$\delta $ is as in \eqref{mydelta}.
If we write $ch(\beta)=
[\epsilon,\epsilon_\infty,\epsilon_0]$, with the $\epsilon$'s living
on $X_U$, then we have
$[\epsilon,\epsilon_\infty,\epsilon_0]|_{(X_\CC^\loc,\bbox)}=\sum
[\epsilon(g_i,f_i)]+ch(\delta)$. Writing this explicitly this means that
\begin{equation*}
  (\epsilon,\epsilon_\infty,\epsilon_0)|_{(U',\bbox)}=\sum
\epsilon(g_i,f_i)+ ch(\delta) +
(d \lambda,\lambda|_{\{t=\infty\}},\lambda|_{\{t=0\}})
\end{equation*}
for some $\lambda\in \Omega^1(U')$ and where now $ch(\delta)$ means
any form representing this class.

The isomorphism $ \reliso : \htms^3(X_{\CC^\loc},\bbox,3) \isom \htms^2(\CC^\loc,3)$
is obtained by integration from $0$ to $\infty$. More precisely it is
given by 
\begin{equation}
  [\epsilon,\epsilon_\infty,\epsilon_0]\mapsto [(\int_0^\infty
  \epsilon ) - (\epsilon_\infty-\epsilon_0)]\,,\label{eq:integisom}
\end{equation}
 where the integration is only with
respect to the variable $t$. Note that we are integrating forms on $X_U$.

For forms on
$U'$ we may do Coleman integration instead
(Section~\ref{sec:coleman}). This
technique was introduced
in~\cite[Section~\ref{sec:down}]{Bes-deJ98}. Note that we only
discussed Coleman integration over $\Cp$. The extension of scalars of
$U$ and the fibers of $U'\to U$, to $\Cp$ are wide open space in the
sense of Coleman so
one can do Coleman integration on them. By abuse of notation we shall
continue to denote this extension of scalars by the same letters.
Coleman integration will be the same as ordinary integration if the
forms extend to $X_U$. The theory of Coleman integration is not
sufficiently developed yet to tell us that what we do makes sense in
general, so we must be careful to check that it makes sense for the
particular forms we are working with.

Now we check what happens to the term
$\epsilon(g,f)$ under this integration. The integral of the first term is
\begin{align*}
  &\int_0^\infty \ovq \log f_0 \omega_g + \ovq \epsilon_g\wedge \phi^\ast
  \dlog f =
  \ovq \log f_0 \int_0^\infty \omega_g + \ovq
  \left(\int_0^\infty\epsilon_g\right)\wedge \phi^\ast  \dlog f\\ =
  &\ovq \log f_0 \log g \dlog (1-g) - \ovqs
  \log (1-g)_0 \log g  \phi^\ast  \dlog f\,.
\end{align*}
The last equality follows because $\int_0^\infty \dlog \tmg  = -\log
g$ and the term involving $\log (\tmg)_0$ vanishes because it does not
involve a $dt$. Adding the term $a_0$ we obtain
\begin{equation}\label{basform}
  \begin{split}
    \int_0^\infty \epsilon(g,f)=
    &\ovq \log f_0 \log g \dlog (1-g)\\ - &\ovqs
    \log (1-g)_0 \log g  \phi^\ast  \dlog f\\ + &\ovq \Theta(g) \phi^\ast
    \dlog f\,.
  \end{split}
\end{equation}
Note that this integral belongs to $\Ocola(U)$, in the notation of
Section~\ref{sec:coleman}.

\begin{lemma}
For $\delta$ in $ \IO \cup \K 2 2 {\O} $ we have  $\reliso(ch(\delta)) = 0 $.
\end{lemma}

\begin{proof}
  As in \eqref{mydelta} $\delta$ is a sum of terms of the form
  $\delta_{1} \cup \delta_{2}$ with $\delta_{1}$ in
$ \K 1 1 {X_\CC^\loc,\bbox} $ and $\delta_{2} $ in $\K 2 2 {\CC^\loc} $.
  That $ \reliso $ vanishes on these elements follows from the proof
  of~\cite[Proposition~\ref{rregfactors}]{Bes-deJ98}.   
\end{proof}

Now we deal with the term
$(d \lambda,\lambda|_{\{t=\infty\}},\lambda|_{\{t=0\}})$.

\begin{proposition}\label{4.17}
  Suppose that $X_{\CC}^\loc$ is obtained from $X_{\CC^\loc}$ by
  removing the graphs of  $t=h_j(x)$ for $j=1,\ldots, n$.
  Assume further that the reductions
  of those graphs are either disjoint or identical (which we can achieve
 by shrinking $ \CC^\loc $).
 Then there are $a_j(x), a(x) \in \O(U)$ such that we have
  \begin{equation*}
   \reliso (d \lambda,\lambda|_{\{t=\infty\}},\lambda|_{\{t=0\}})
    =d (a+\sum_{j} a_j \log(h_j))\,,
  \end{equation*}
  where, if there are two $h_j$ with identical reduction, one may take
  just one of them.
  In particular, it belongs to  $\Ocola(U)$.
\end{proposition}
\begin{proof}
We have global coordinates $x$ and $t$ on $U'$ so we can write
$\lambda=f(x,t) dx + g(x,t)dt$. Then 
\begin{equation*}
  d \lambda= \left(\frac{\partial g}{\partial x}-
 \frac{\partial f}{\partial t}\right) dx\wedge dt\,.
\end{equation*}
Therefore
\begin{equation*}
   \int_{t=0}^{t=\infty} d \lambda =  \left(\int_{t=0}^{t=\infty}
   \frac{\partial
   g}{\partial x} dt\right)dx - (f(x,\infty)-f(x,0))dx\,.
\end{equation*}
But the second term is exactly $\lambda|_{t=\infty}-\lambda|_{t=0}$ so we
find
\begin{equation*}
 \reliso [d \lambda,\lambda|_{\{t=\infty\}},\lambda|_{\{t=0\}}]=
  d \left(\int_{t=0}^{t=\infty} g(x,t) dt\right)\,.
\end{equation*}
Consider now the two-form $\gamma=g(x,t) dx\wedge dt \in \Omega^2(U')$. This
is closed so represents a cohomology class in
$\hr^2((X_{\CC}^\loc)_\kappa/K)$. We have a short exact sequence
\begin{equation*}
  \hr^2((X_{\CC^\loc})_\kappa/K) \to \hr^2((X_{\CC}^\loc)_\kappa/K
  \xrightarrow{\res}
  \oplus_i \hr^1((\CC^\loc)_\kappa/K)\,,
\end{equation*}
where the map $\res=\oplus_j \res_j$ is the sum of the boundary maps
on the reductions
of $t=h_j(x)$, composed with the pullback under the isomorphisms of
these graphs with $(\CC^\loc)_\kappa$. Suppose that $\res_j(\gamma)$
is the cohomology class of $a_j(x) dx\in \Omega^1(U)$. Let
$\gamma_j:=a_j(x) dx \wedge \dlog(t-h_j(x))$. Clearly
$\res_l(\gamma_j)=0$ if $l\ne j$. We claim that
$\res_j(\gamma_j)=\res_j(\gamma)$. This can be seen easily by applying
the map $(x,t)\to (x,t-h_j(x))$, transforming $\gamma_j$ to $a_j(x)
dx\wedge \dlog(t)$. Thus, $\gamma- \sum_j \gamma_j$ extends to
$\hr^2((X_{\CC^\loc})_\kappa/K)$ and its integral is a holomorphic one
form on $U$. Let this form be $a(x) dx$. Since
$\int_{t=0}^{t=\infty}\gamma_j = \pm a_j(x) \log(h_j(x)) dx $ we find
$\pm \int_{t=0}^{t=\infty} \gamma =(a(x)+\sum a_j(x)
\log(h_j(x))) dx$ and dividing by $dx$ we find
$\int_{t=0}^{t=\infty} g(x,t) dt = \pm (a(x)+\sum a_j(x)
\log(h_j(x))) $. This completes the proof.
\end{proof}

These results give us a strategy for breaking the regulator into a sum
of terms, each depending on the pairs $(g_i,f_i)$, as follows. Suppose
that $\omega $ is a form of the second kind on $C$ and let $[\omega]$
be its cohomology class in $\Hdr^1(C/K) $.
\begin{definition}\label{proplist1}
  A functional $L_\omega:\Ocola(U)\to \Cp$ will be called good if it has
  the following properties:
\begin{itemize}
\item
   it kills terms of the forms $d a$ and $d(a \log f)$ for $a,f\in \O(U) $;
\item
   it kills all terms of the form $\reliso [d\lambda,\lambda|_{\{t=\infty\}},\lambda|_{\{t=0\}}]$;
\item 
   if $\eta$ is in $\Omega^1(U)$ then we have
  $L_\omega(\eta)=  \canproj([\eta]) \cup [\omega]$. 
\end{itemize}
\end{definition}

\begin{proposition}\label{proplist}
Suppose that an element $ \beta $ in $ \K 4 3 {\CC^\loc} $ maps to 
$ \sum_i [g_i]_2 \cup (f_i) $ in $ H^1(\Ccomp \bullet \O ) $
under the natural map
$ \K 4 3 {\CC^\loc} \to \K 4 3 {\O} \to \K 4 3 {\O} / \K 3 2 {\O} \cup \OQ $,
see~\eqref{H1auxC},
and that $ch(\beta)=[\eta_0]$ in the model~\eqref{eq:Cmodel}. Then
we have, for a good functional $L_\omega$,
\begin{equation*}
  \canproj([\eta_0])\cup [\omega] = \sum_i L_\omega\left(\int_0^\infty
  \epsilon(g_i,f_i)\right)\,.
\end{equation*}
\end{proposition}
\begin{proof}
We must first show that the map
\begin{equation*}
   \K 4 3 {\CC^\loc} \xrightarrow{ch} \htms^2(\CC^{\loc},3)
   \xrightarrow{\eta_0\mapsto L_\omega(\eta_0)} \Cp
\end{equation*}
factors via $ \K 4 3 {\O} / \K 3 2 {\O} \cup \OQ $. By further
localizing, it suffices
to show that the map above vanishes on elements of the form $\gamma
\cup f$ with $\gamma \in  \K 3 2 {\CC^\loc} $ and $f\in \O^*({\CC^\loc})$.
We have
\begin{equation}
  \htms^1(\CC^{\loc},2)=\{(0,\epsilon),\epsilon\in \O(U),\;d\epsilon=0\} =\{(0,\epsilon),\epsilon\in K\} 
\,.
\end{equation}

Thus $ch(\gamma) = (0,\alpha)$ for some $\alpha\in K$. On the other hand,
by~\eqref{eq:regfunc} we have $ch(f)= (\dlog f,\log(f_0)/q)$ (here
$f_0$) does not matter).  Using \eqref{eq:cupprod} we obtain, in the
model \eqref{eq:Cmodel}
\begin{equation*}
  ch(\gamma\cup f) = (0,\alpha)\cup  (\dlog f,\log(f_0)/q) = \alpha
  \dlog f \,.
\end{equation*}
The factorization thus follows from first property of the good
functional (with $a=1)$. Next, by Proposition~\ref{4.17} 
the first property also implies that $L_\omega $ kills all terms of
the form $ \reliso [d\lambda,\lambda|_{\{t=\infty\}},
\lambda|_{\{t=0\}}]$.
The result now follows immediately from the discussion above.
\end{proof}

\begin{remark}
  There is a final wrinkle here because of the
  normalization~\eqref{eq:twist}
  for the syntomic regulator. For $\beta$ as in the
  Corollary, the regulator of $\beta$ is in fact $[\eta]$ with
  $\left(1-\frac{\phi^\ast}{q^3}\right) \eta=\eta_0$ Thus, once we
  have the 
  functional $L_\omega$ we shall be able to compute $\canproj(
  \eta_0)\cup [\omega]$ but will in fact want $\canproj(
  \eta)\cup [\omega]$. Fortunately, it
  is easy to see (and will be explained) that if we know
  $\canproj(
  \eta_0)\cup [\omega]$ for \emph{all}
  $\omega$, then we also know  $\canproj(
  \eta)\cup [\omega]$ for all $\omega$. 
  In fact, as in previous computations, the result with $\eta$ is much
  simpler than with $\eta_0$, confirming the ``correctness'' of our
  normalization.
\end{remark}

\section{Wishes}
\label{sec:wishes}

This section is highly speculative. It contains no formal
proofs. Nevertheless, we feel it is vital for the understanding of a
significant portion of the computations to come. It also suggests
interesting research directions into a more canonical representation
of syntomic cohomology, one that would make the computations in the
syntomic case equivalent to the complex case.

We want to follow a strategy that proved very successful in computing
syntomic regulators on $K_2$ of curves (see the discussion after
Proposition~5.2 in~\cite{Bes98b}). We argue
heuristically, in some make believe world where syntomic cohomology
looks much more like Deligne cohomology from the computational
standpoint, and get a formula for the
regulator. Then we try to relate this formula with the formula we
obtained in the previous section and see what needs to be proved to
show that the two formulas are equivalent. That the make believe
formula turns out to be correct is a strong indication that one should
be able to turn the make believe computation into a rigorous one.

The make believe computation is based on the
following assumptions:
\begin{itemize}
\item The ``cohomology'' is given by the pairs $(\omega,h)$ where
  $\omega$ is an $i$-form and $h$ is an $i-1$ form with
  $dh=\omega$. Of course $h$ is not an actual form but something like
  a Coleman form, for example a Coleman function.
\item The ``regulator'' of a function $f$ is the pair $(\dlog(f),\log(f))$.
\item The cup product is given by $(\omega_1,h_1)\cup (\omega_2,h_2)=
  (\omega_1\wedge \omega_2,\omega_1\wedge h_2$ or $h_1\wedge \omega_2)$.
\end{itemize}

With these rules, we can redo the computation from the previous
section in this make believe language:
We have in $\htms^2(X_{\CC}^{\loc},2)$ that $- ch(\tmg \cup
(1-g))=(\omega_g,\varepsilon_g)$ with
$\omega_g$ as in Lemma~\ref{firstreg} and $\varepsilon_g=-\log (1-g) \dlog
(\tmg)$. Since the restriction of $\varepsilon_g$ to $t=0$ is
$-\log(1-g) \dlog(g)= d \Li_2(g)$ we have, following the proof of
Proposition~\ref{3.10}, that
$ch([g]_2)\in \htms^2(X_{\CC}^{\loc},\bbox,2)$ equals
$[\omega_g,\varepsilon_g,0,\Li_2(g)]$. Cupping with  $(\dlog(f),\log(f))$
we get
\begin{equation*}
 \tilde\epsilon(g,f):= ch([g]_2 \cup (f))=[-\log(f) \dlog (\tmg)\wedge \dlog (1-g)),0,0]\,.
\end{equation*}
Applying $ \reliso $ we find
$ \reliso (\tilde\epsilon(g,f)) = \log(f) \log (g) \dlog (1-g) $.

We now compare this with  $\int_0^\infty \epsilon(g,f)$ of
\eqref{basform}. Continuing to mimic the discussion of the $K_2$
in~\cite{Bes98b}, the former version should be an untwisted version of
the latter, i.e., without the ``twist'' by $ (1-\fdiv{3})$. To see this, we
use the formalism described in \cite[Remark~\ref{trick}]{Bes98b} to get
\begin{equation}\label{symbolic}
\begin{split}
  \left(1-\fdiv{3}\right) [\log(f) \log(g) \dlog(1-g)] =\phantom{+}
  &\ovq \log(f_0) \log(g) \dlog(1-g)\\ + &\ovqs \logf(f) \log(g)
  \dlog(1-g)_0\\ + &\ovqt \log(g_0) \logf(f) \phi^\ast \dlog (1-g) 
\end{split}
\end{equation}
This already begins to look similar to  $\int_0^\infty \epsilon(g,f)$
but there are differences. We want to argue that the difference is
``exact''. This can not be taken to simply mean being the differential
of something, since in Coleman's theory every form is
integrable. Experience has shown that things are exact if they are the
differential of a product of functions. We shall use two such
assertions. To each one will correspond a precise statement in the
following sections, which will be justified by the techniques we shall
introduce. To remind ourselves where these occured, we
shall call them ``Wishes'', and mark them explicitly. The first one is 
\begin{wish}
  We have in cohomology that $ \Theta(g) \dlogf(f) =  - \logf(f) \, \dd \Theta(g) $.
\end{wish}

Using this wish we can write the term $\ovq \Theta(g) \dlogf(f)$ in
\eqref{basform} as
\begin{align*}
  &\phantom{=}-\ovq d \Theta(g) \logf(f) \\ &= - \ovq
  \left(\epsgi\right) \logf(f)\\ &=
-\ovqs \log(1-g)_0 \dlog(g) \logf(f) + \ovqt \log(g_0) \dlog \phi^\ast (1-g) \logf(f) \,,
\end{align*}
so we obtain
\begin{align*}
\int_0^\infty \epsilon(g,f)=
    &\phantom{-}\ovq \log(f_0) \log(g) \dlog (1-g) - \ovqs
  \log (1-g)_0 \log(g)  \phi^\ast  \dlog(f)\\
  &-\ovqs \log(1-g)_0 \dlog(g) \logf(f) + \ovqt \log(g_0) \dlog \phi^\ast (1-g) \logf(f) \,.  
\end{align*}
Comparing this with $\left(1-\fdiv{3}\right) \left(\log(f) \log(g)
\dlog(1-g)\right)$ given in~\eqref{symbolic} we see that the first and last
terms are the same, and that therefore we get our desired equality,
``twisted'' by $1-\fdiv{3}$ if we get our second wish to come
true.

\begin{wish}
  We have in cohomology that
  \begin{equation*}
  \log (1-g)_0 \log(g)  \phi^\ast(\dlog(f))
  + \log(1-g)_0 \logf(f)  \dlog(g)
  +  \log(g) \logf(f) \dlog(1-g)_0 
 \end{equation*}
is trivial.
\end{wish}

In Sections~\ref{sec:trip-loc} and~\ref{sec:trip-glob} we shall introduce triple indices. The wishes
described above correspond to precise results stated in terms
of triple indices, which we can indeed prove.

\section{The triple index, local theory}
\label{sec:trip-loc}

We first briefly recall the theory of the ``local index''
from~\cite[Section~\ref{sec:recip}]{Bes98b}. In our new context this
should be called the double 
index. To make things slightly simpler, we work in an algebraic
context. The transition to working with annuli is straightforward.

Let $K$ be a field of characteristic $0$. 
We consider the algebra 
$\alog:=K((z))[\log(z)]$ of polynomials over the formal variable
$\log(z)$, over the field of finite
to the left Laurent power series in $z$. We further consider 
the module of differentials $\alog \cdot dz$. It is an easy excercise
in integration by parts to see that every form in $\alog \cdot dz$ has
an integral in $\alog$ in a unique way up to a constant. We
distinguish in $\alog$ the subfield $\mer:= K((z))$ of
meromorphic functions and the subspace $\aloga=\mer + K\cdot \log(z)$
consisting exactly of all functions whose differential is in $\mer
\cdot dz$. To $F\in \aloga$ we can associated the residue of its
differential $\res dF\in K$. If $F\in \aloga$, then $F\in \mer $ if
and only if $\res dF=0$.

\begin{definition}[{\cite[Proposition~\ref{double-index}]{Bes98b}}]
  The double index is the unique anti-symmetric bilinear form
  $\pair{~,~} : \aloga  \times \aloga \to K$ such that $\pair{F,G}= \res
  FdG$ whenever this last expression makes sense.
\end{definition}
We recall that the construction of this index is essentially trivial:
one notices that the anti-symmetry forces
$\pair{\log(z),\log(z)}=0$ and that $\pair{F,G}= -\res GdF$ whenever
this expression makes sense. Then one writes $F=\alpha\log(z)+f$, $G=
\beta\log(z)+g$ with $f,g\in \mer$ and then one uses the bilinearity
to write $\pair{F,G}$ as a sum of terms that can be computed.

The triple index turns out to be a bit more complicated. First of all
we need to explain on which data it is evaluated:
\begin{itemize}
\item three functions $F,G,H$ in $\aloga$;
\item for each two functions $R$ and $S$ out of $F,G,H$ a choice of
  $\int RdS$ (i.e., a function in $\alog$ whose differential is $RdS$)
  and of $\int SdR$ in such a way that
  \begin{equation} \label{integparts}
     \int RdS+ \int SdR = RS \,.
  \end{equation}
\end{itemize}
As it will turn out this information is a bit redundant: clearly
$\int RdS$ determines $\int SdR$. Also it will turn out that the index will
be independent of $\int FdG$. Still, these symmetric data are very
convenient. To not carry around too much notation,
we shall simply denoted these data by $(F,G;H)$, where the additional choices
should be understood from the context. In particular, any permutation
of $F,G,H$ induces an obvious permutation of the additional 
data. Also, if $(F_i,G;H)$, $i=1,2$ are given with all their
additional data then there is a natural choice of data for
$(F_1+F_2,G;H)$, and similary in the second and third positions. If we
do need to indicate a change in the auxiliary
data we shall write this as $(F,G;H|I_{FdG},\cdots)$, where the
subscript $FdG$ indicates that $I$ is an integral of $FdG$.

\begin{proposition}
  There exist a unique function from data as above to $K$, denoted
 $(F,G;H) \mapsto \pair{F,G;H}$, called the  triple index,
  such that the following conditions are satisfied.
  \begin{enumerate}
  \item Trilinearity - the triple index is linear in each of the
    three variables, which means that $\pair{\alpha_1 F_1 + \alpha_2
      F_2,G;H} = \alpha_1 \pair{F_1,G;H} +\alpha_2 \pair{F_2,G;H}$
    provided that all auxiliary data are chosen in the way indicated
    above, and similary for linearity in $G$ and $H$.
  \item Symmetry - we have $\pair{F,G;H}=\pair{G,F;H}$, again with the
    choice of auxiliary data indicated above.
  \item Triple identity - We have, again with the obvious additional
    choices, $$ \pair{F,G;H}+\pair{F,H;G}+\pair{G,H;F}=0 . $$
  \item Reduction to the double index - if $G\in \mer$ then $\pair{F,G;H}
    = \pair{F,\int GdH}$, where $\int GdH$ is taken from the auxiliary
    data and is in $\aloga$ because by assumption $GdH \in \mer \cdot dz$.
  \end{enumerate}
\end{proposition}

\begin{proof}
We first show that the dependency on the choices of integrals is
forced by the properties of the triple index.

\begin{lemma}
\label{tripexists}
   Suppose that the triple index exists. We then have the following
   change of constant formulae:
   \begin{enumerate}
   \item If $ C $ is a constant, then
     \begin{align*}    
       \pair{F,G;H|(I+C)_{GdH},(J-C)_{HdG}}&=
       \pair{F,G;H|I_{GdH},J_{HdG}}-C\cdot \res dF\\
       \pair{F,G;H|(I+C)_{FdH},(J-C)_{HdF}}&=
       \pair{F,G;H|I_{FdH},J_{HdF}}-C\cdot \res dG
     \end{align*}
   \item The triple index is independent of the integral $\int FdG $.
   \end{enumerate}
\end{lemma}

\begin{proof}
We use the trilinearity. Consider the data $(F,0;H)$, where the
additional data are the same for $F$ and $H$ but we take the integral
of $0 dH$ to be $C$, hence we are forced to take that of $H d0$ to be
$-C$. We take $\int 0 dF =0$. The trilinarity implied that
$\pair{F,G;H}$ and $\pair{F,0;H}$
gives the left-hand side of the formula. But reduction to the double
index means that $\pair{F,0;H}=\pair{F,C}=-\res C dF$. An identical
argument proves the second case. Finally, if in the above argument we
take instead $\int 0 dF =D$ and $\int 0 dH=0$, we see from exactly the
same argument that the integral is independent of the auxiliary choice
$\int F dG$.
\end{proof}

We now check that the triple index is uniquely defined on all data
where at least one of $F$, $G$, $H$ is in $\mer$. Clearly in this case
we can use Reduction to the double index together with symmetry and
the triple formula to compute the index, so it is clearly unique. The
following lemma gives existence.

\begin{lemma}\label{recepy}
  Consider the following recepy:
  \begin{enumerate}
  \item  if $G\in \mer$
    define $\pair{F,G;H}=\pair{F,\int GdH}$;
  \item  if $F\in \mer$ define
    $\pair{F,G;H}=\pair{G,F;H}$ where the last expression is defined
    as in (1);
  \item if $H\in\mer$ define
    $\pair{F,G;H}=-(\pair{F,H;G}+\pair{G,H;F})$ where each of these terms
    is defined as in 1.
  \end{enumerate}
  Then this recepy gives a well-defined $\pair{F,G;H}$ in all cases
  where at least one of $F$, $G$ and $H$ is in $\mer$ and restricted
  to this subset it satisfies all properties of the triple index.
\end{lemma}

\begin{proof}
To show that this expression is well-defined we need to consider what
happens when two 
of $F,G,H$ are in $\mer$: If $F,G\in \mer$ we check that $\pair{F,\int
  GdH}=\pair{G,\int FdH}$. This follows because by the definition of
the double index both expressions equal $\res FGdH$. Next we check that
if $G,H\in \mer$ then 
\begin{align*}
  &\pair{F,\int GdH} + \pair{F,\int HdG} + \pair{G,\int HdF}\\&=
  \pair{F,GH} +  \pair{G,\int HdF} \text{ by bilinearity of the double index and~\eqref{integparts}}\\
  &= -\res GHdF +  \res GHdF =0\,.
\end{align*}
Thus we find that we have a well-defined expression. We need to check
that all properties of the expected triple index hold in this
case. Trilinearity is essentially clear from the bilinearity of the
double index. Symmetry is also easy: if $F$ or $G$ are in $\mer$ then
symmetry follows from the first two rules. If $H $ is in $ \mer$ then the
expression in (3) is clearly symmetric in $F$ and $G$. The triple
identity is forced by (3) and the reduction to the double index is an
immediate consequence of our check that the triple index is well-defined.
\end{proof}

Note that the proof of Lemma~\ref{tripexists} applies verbatim for this
partial triple index, so we know the dependency on the choices of
integrals.

To extend the triple index to all $F$, $G$ and $H$ we first check the
case where $F=G=H=\log(z)$. Then we can arrange that all auxiliary data
equal $(1/2)\log^2(z)$. The triple formula implies immediately that
(with these data)
\begin{equation}
  \label{triple-log}
  \pair{\log(z),\log(z);\log(z)}=0
.
\end{equation}
We can now
demonstrate uniqueness for the triple index.  Suppose
$F_i=\alpha_i\log(z)+f_i$, $i=1,2,3$ where $\alpha_i\in K$ and $f_i\in
\mer$.  Choose some auxiliary data $\int RdS$ for any two $R$ and $S$
out of $f_i$ and $\alpha_i \log(z)$, where we continue to take $\int
\log(z)\dlog(z)= (1/2)\log^2(z)$. Using trilinearity and
\eqref{triple-log} we can write $\pair{F_1,F_2;F_3}$, with some choice
of auxiliary data, as the sum with some coefficients of triple
indices where at least one of the entries is in $\mer$ which are
therefore computable by previous considerations. Now we can use change
of constant to write $\pair{F_1,F_2;F_3}$ with arbitrary auxiliary
data. This shows uniqueness and gives a formula for the general
index. We need to check that this formula is well-defined, which given
the fact that all the summands are well-defined thanks to Lemma~\ref{recepy}
amounts to checking independence of the choices of the auxiliary
data. This is Just a tedious formal check: suppose for example that we
add $C$ to $\int \alpha_1\log(z)df_3$, and correspondingly subtract
$C$ from  $\int f_3\alpha_1\dlog z$. This will have the effect that
$\int F_1 d F_3 $ will be added a $C$ and $\int F_3 d F_1 $ will be
subtracted a $C$. This procedure will subtract $\alpha_2 C= C \res dF_2$ from
$\pair{\alpha_1 ,\alpha_2 \log(z);f_3} $ and will not change any of
the other indices. This shows that the change does not alter the index.

It remains to check that our formula satisfies all the properties for
the triple index. First the change of constant formula of
Lemma~\ref{tripexists} is clear
because we used it in the definition and we showed that the formula we
get is well-defined. Now given change of constant it easy to see that
it is enough to check trilinearity, symmetry and triple identity for
one choice of auxiliary data. The derivation of these three formulas is
then completely formal. Finally, reduction to the double index can
only occur if at least one $\alpha_i$ is $0$. But in this case we
clearly get the triple index for the case where $F_i\in \mer$ so we
know this formula already.
\end{proof}

To compute the triple index in some concrete situations, which will be
needed later, we introduce the notion of the constant term.
\begin{definition}
\label{constanttermdef}
  The constant term, with respect to the variable $z$ is the linear
  functional $c_z:\alog \to K$, first defined on $\mer$ by 
  \begin{align*}
    &c_z(\sum a_n z^n)=a_0\\
    \intertext{and then in general by}
    &c_z(\sum_{i=0}^\infty f_i(z) \log^i(z)) = c_z(f_0) \,.
  \end{align*}
\end{definition}

Note that the unlike the triple index, the constant term definitely
depends on the choice of the local parameter $z$. For example, for
$\alpha\in K$ and the function $f(z)=\log(z) = \log(\alpha z)-
\log(\alpha)$ we have $c_z(f)=0$ but $c_{\alpha z}(f)= - \log(\alpha)$.

\begin{proposition}\label{const-term-expr}
  Let $F$, $G$ and $H$ be three functions
  in $\aloga$ whose
  differentials (which are in $\mer dz$) have at most simple poles at
  $0$. The choice of integrals $\int FdH$ and $\int GdH$ gives 
  auxiliary data for the computation of $\pair{F,G;H}$ and with
  respect to this choice we have
  \begin{equation*}
    \pair{F,G;H}= c_z(F)\cdot c_z(G) \cdot \res dH - \res dF\cdot c_z(\int
    GdH) - \res dG\cdot c_z(\int FdH)
  \end{equation*}
\end{proposition}
\begin{proof}
We have a bilinear map
\begin{equation*}
  (F,H) \to \pint FdH:=\text{ unique } \int FdH \text{ with } c_z(\int
  FdH)=0\,.
\end{equation*}
Therefore, we see that the map
\begin{equation*}
  (F,G,H) \to \ppair{F,G;H}:= \pair{F,G;H \Big |\pint FdH_{FdH}, \pint
  GdH_{GdH}}
\end{equation*}
is trilinear and symmetric in $F$ and $G$. By Lemma~\ref{tripexists}
it suffices to prove that
\begin{equation}\label{triconst}
  \ppair{F,G,H} = c_z(F)\cdot c_z(G) \cdot \res dH
\end{equation}
and as both sides are
trilinear and symmetric in $F$ and $G$, and as $F= a\log(z) + f(z)$
with $f(z)$ holomorphic and similary for $G$ and $H$, it suffices to
treat the following cases:\\
(1) When $f$, $g$ and $h$ are holomorphic we have
  \begin{equation*}
    \ppair{f,g,h} = \res fgdh = 0 = c_z(f) c_z(g) \res dh
  \end{equation*}
  since $\res dh=0$.\\
(2) Suppose $F=G=H=\log(z)$. Since $c_z(\log^2(z)/2)=0$ we see that
  the local index computed with all auxiliary data set equal to $\log^2(z)/2$
  is given by $\ppair{\log(z),\log(z);\log(z)}$, and this we know is $0$
  by~\eqref{triple-log}. On the other hand, the
  right-hand side of~\eqref{triconst} is also zero since
  $c_z(\log(z))=0$.\\
(3) If $g$ and $h$ are holomorphic we have
  \begin{equation*}
    \ppair{\log(z),g;h} = \pair{\log(z),\pint gdh}=-\res (\pint gdh)
    \dlog z = (\pint gdh)(0) = 0 \,,
  \end{equation*}
  which equals $c_z(\log(z)) c_z(g) \res dh$ as required.\\
(4) if $f$ and $g$ are holomorphic we find
  \begin{equation*}
    \ppair{f,g;\log(z)} = \res fg \dlog z = fg(0) = c_z(f) c_z(g) \res
    \dlog z \,.
  \end{equation*}
(5) If $g$ is holomorphic and $a=c_z(g) $ we see that $\pint (g-a)
    \dlog z= \pint g \dlog z - a\log(z)$. Using this we find
    \begin{align*}
      \pair{\log(z),g;\log(z)} &= \pair{\log(z),\pint g \dlog z} =
      \pair{\log(z),\pint (g-a) \dlog z}\\
      &=-\res \left(\pint (g-a) \dlog z\right) \dlog z=0
    \end{align*}
since $\pint (g-a) \dlog z $ is holomorphic and has constant term
$0$. This again equals the right-hand side.\\
(6) The final case is for $\pair{\log(z),\log(z);h} $ with $h$
holomorphic. Now $c_z(h\log(z))=0$, so we have the equation
$\pint h \dlog z + \pint \log(z) dh = h\log(z)$. We therefore
immediately deduce this last case from the previous one and the
triple identity.
\end{proof}

\section{The triple index, global theory}
\label{sec:trip-glob}

At this point we shall switch for convenience to assuming that our
ground field is $\Cp $.
Suppose now that we consider an open annulus $V\isom \{r<|z|<s\}$ with a
parameter $z$. Then exactly the same analysis gives us a triple index
on $V$.

The uniqueness of the triple index immediately
implies (compare~\cite[Lemma~\ref{functoriality}]{Bes98b}) the
following result.

\begin{lemma}
  If $\phi:V \to V$ is an endomorphism of degree $n$, let $\phi^\ast
  (F,G;H)$ be defined in the obvious way, pulling back by $\phi$ all
  the auxiliary data. Denote these data simply by $(\phi^\ast
  F,\phi^\ast G;\phi^\ast H)$. Then we have the formula
  \begin{equation*}
    \pair{\phi^\ast F,\phi^\ast G;\phi^\ast H} = n\pair{F,G;H}.
  \end{equation*}
\end{lemma}

Consider now a wide open space $U$ over $\Cp$ with annuli ends set
$\End(U)$. We
shall denote the triple index with respect to the end $e$ by the
subscript $e$. When we are given 3 Coleman functions $F$, $G$ and $H$
on $U$, such that their differentials are in $A(U)$, we may choose
Coleman integrals for all forms $RdS$ when $R$ and $S$ are among $F$,
$G$ and $H$, and we may do so in such a way that $\int RdS + \int
SdR=RS$ globally. This allows us to compute $\pair{F,G;H}_e$ at each
end $e$ and we may consider the global triple index
\begin{equation*}
  \pair{F,G;H}_\gl =\sum_{e\in \End(U)} \pair{F,G;H}_e
\end{equation*}

\begin{lemma}
  The expression $\pair{F,G;H}_\gl$ is independent of the auxiliary
  choices, so depends only on $F$, $G$ and $H$.
\end{lemma}

\begin{proof}
  Since the possible integrals differ from one another by a global
  constant, if we change for example $\int GdH$ by a constant $C$, the
  change of constant formula implies that the global triple index
  changes by $\sum_e C \res_e dF=C\sum_e \res_e dF= C\cdot 0=0$.
\end{proof}

Unlike the global double index, the global triple index does not
depend solely on the cohomology classes of $dF,\cdots$, and not even
just on the differentials of the functions. For example, if $C$ is a
constant we have the
formula $\pair{F,C;H}_\gl=\sum_e \pair{F, \int CdH}_e =C\sum_e
\pair{F,H}_e$. However, we do have the following.
\begin{lemma}\label{constinthird}
  If $C$ is a constant then $\pair{F,G;C}_\gl=0$.
\end{lemma}

\begin{proof}
Indeed,
\begin{align*}
  \pair{F,G;1}_\gl&=-\pair{F,1,G}_\gl - \pair{G,1,F}_\gl \text{ by the triple identity}\\
  &= -\pair{F,\int dG}_\gl - \pair{G,\int dF}_\gl \text{by reduction to the double index}\\
  &= -\pair{F,G}_\gl - \pair{G,F}_\gl = 0 \,,
\end{align*}
where the last two equalities follow because the global double index
is independent of the choice of the integral and by the anti-symmetry
of the double index.
\end{proof}

The lemma suggests that the global triple index is quite an
interesting creature. It deserves further study. For our purposes we
only need the following results:
\begin{proposition}\label{trip-recip}
  Let $F$, $G$, $H$ in $\acol(U)$ have $dF,dG,dH$ in 
  $\Omega^1(U)$, and suppose that $[dF]$ and $[dG]$ are eigenvectors for
  Frobenius with eigenvalue $q$. Then
  $\pair{F,G;H}_\gl = 0$.
\end{proposition}

\begin{proof}
We begin by establishing the following formulae.
If $r\in A(U)$ then 
\begin{equation}
  \label{log+hol}
  \pair{F,r,H}_\gl = \sum_e
  \pair{F,\int rdH}_e = 0
\end{equation}
where the last equality follows
from~\cite[Corollary~\ref{reciprocity}]{Bes98b}.
Similarly we find that if also $s\in
A(U)$ then $\pair{s,G,H}_\gl=0$. Now if $h\in A(U)$, then
\begin{equation*}
\pair{F,G;h}_\gl = -\pair{F,h;G}_\gl -
\pair{G,h;F}_\gl = 0 \,,  
\end{equation*}
by application of~\eqref{log+hol}. This last formula shows that for
fixed $F$ and $G$ the function $H \mapsto \pair{F,G;H}_\gl$
depends only on the cohomology class of $dH$, $[dH]\in H^1(U)$. Let
$\phi$ be a Frobenius lift on $U$. The assumption on $F$ and $G$ implies
the existence of $r,s\in A(U)$ such that $\phi^\ast
F=qF+r$ and
$\phi^\ast G= qG +s$. Using this we can compute
\begin{align*}
  &q\pair{F,G;H}_\gl =   \pair{\phi^\ast F,\phi^\ast
  G;\phi^\ast H}_\gl \\&=    \pair{qF+r ,qG+s;\phi^\ast
  H}_\gl =  q^2   \pair{F,G;\phi^\ast H}_\gl \,, 
\end{align*}
using bilinearity and~\eqref{log+hol}.
This shows that the functional $[dH] \mapsto
\pair{F,G;H}_\gl$ is an eigenvector for the action of
$\phi^\ast$ with eigenvalue $1/q$. Such a functional must be $0$
because the eigenvalues of $\phi^\ast$ on $H^1(U)$ are either $q$ or
Weil numbers of weight~1.
\end{proof}

Note that this proposition applies in particular when $F$ and $G$ are of
the form $r+\log(f)$ where $r,f\in A(U)$. This follows since by~\cite[Lemma~2.5.1]{Col-de88},
$\log(f^q/\phi^\ast(f)) $ is in $ A(U)$.

\begin{proposition}\label{needfortild}
  Suppose $\omega\in \Omega^1(U)$ has trivial residues on all residue
  ends, so that its Coleman integral $F_\omega$ is in fact analytic on
  the ends. Let $F,G,H$ be Coleman functions on $U$ whose
  differentials are holomorphic and represent eigenvectors for
  Frobenius with eigenvalue $q$. Then
\begin{equation}\label{stokes}
  \sum_e \pair{F,G;\int F_\omega dH}_e  + \sum_e \pair{F,H;\int F_\omega dG}_e + \sum_e \pair{G,H;\int F_\omega dF}_e
\end{equation}
equals zero.
\end{proposition}

\begin{proof}
Note that the expression above makes sense since on each residue end
$e$ the form $F_\omega dH$ is analytic, so the corresponding
triple index is defined, and similarly with $ H $ replaced by $ F
$ and $ F $. Note also that this is of course not a global index in the sens
of this section, since $F_\omega dH$ is not holomorphic. The
strategy for the proof is the same as for
Proposition~\ref{trip-recip}. First we notice that if $F_\omega$ is in
fact holomorphic, then the identity holds by
Proposition~\ref{trip-recip}. It follows that the expression factors
via the cohomology class $[\omega]$. Suppose now that we replace
$F$ by a holomorphic function $u$. We then have
\begin{align*}
   \sum_e \pair{u,G;\int F_\omega dH}_e &=
   \sum_e\pair{G,\int F_\omega u dH}_e\\
   \sum_e \pair{u,H;\int F_\omega dG}_e &=
   \sum_e\pair{H,\int F_\omega u dG}_e
 \end{align*}
by reduction to the double index, and
\begin{align*}
  &\phantom{=} \sum_e \pair{G,H;\int F_\omega du}_e =
   \sum_e \pair{G,H; F_\omega u-\int u \omega}_e\\
  &= \sum_e \pair{G,H; F_\omega u}_e \quad
   \text{by Proposition~\ref{trip-recip}}\\
   &=-\sum_e\pair{G,F_\omega u;H}
   -\sum_e\pair{H,F_\omega u;G}\quad\text{by the triple identity}\\
   &=-\sum_e\pair{G,\int F_\omega u dH}_e-
   \sum_e\pair{H,\int F_\omega u dG}_e
\end{align*}
by reducing to the double index again as $ F_\omega $ is analytic. This shows that if we replace
$F$ by $u$ in the formula to be proved we indeed get
$0$. Similarly we get the same result if we replace $G$ by a
holomorphic $v$, $H$ by a holomorphic $w$, or if we do $2$ or
$3$ of these replacements at the same time. Now, exactly as in the proof of
Proposition~\ref{trip-recip}, writing the right-hand side of~\eqref{stokes}
as $T(F,G,H,\omega)$, we easily get from the previous computation that
\begin{align*}
  q T(F,G,H,\omega)&=T(\phi^\ast F,
  \phi^\ast G,\phi^\ast H,\phi^\ast \omega)\\ &=
   q^3 T(F,G,H,\phi^\ast\omega)
\end{align*}
which shows that the functional $\gamma([\omega]):=
T(F,G,H,\omega)$ satisfies $ \gamma( \phi^* [\omega] ) = q^{-2} \gamma(
[\omega] ) $,
so that $ \gamma(q^2 \phi^* - \id) [\omega]) = 0 $.  By the theory of
Weil numbers, it follows that $ \gamma=0$. This proves what we want.
\end{proof}

\section{A formula for the regulator} \label{sec:begin}

In this section we obtain our first explicit regulator formula,
Theorem~\ref{aux-thm}, using the theory of the triple index.
For technical reasons, the syntomic regulator itself must be
developed over a discretely values field. However, since we have
formulas for the regulator that make sense over $\Cp $ as well, we
work from now until the end of this paper over $\Cp$.

Now that we have at our disposal the triple index, we can interpret
our make believe computation of Section~\ref{sec:wishes} in such a
way that it will become true. We continue with the notation of the
previous section, so $U$ is a wide open space over $\Cp$.

The first thing that the triple index allows us to do is to extend the
cup product to some Coleman differential forms.
We first need a lemma.

\begin{lemma} \label{7.1}
  The map $\Ocola(U) \to H^1(U) \otimes \Omega^1(U)$ given by 
  \begin{equation*}
    \sum F_{\omega_i} \eta_i \to \sum [\omega_i] \otimes \eta_i
  \end{equation*}
  is well-defined.
\end{lemma}

\begin{proof}
This is~\cite[Corollary~6.2]{Bes99}.
\end{proof}

\begin{proposition}
  There is a unique bilinear map
$$
\dpair{~,~}: \acola(U) \otimes   \Ocola(U) \to \Cp
$$ 
  such that we have, for any $F$, $G$, $H$ in $\acola(U)$,
  \begin{equation} \label{dpairdef}
    \dpair{F,G \dd H}=\pair{F,G;H}_\gl \,.
  \end{equation}
\end{proposition}

\begin{proof}
By definition, $\Ocola(U)$ is generated by forms like $G \dd H$ so
uniqueness is clear. To show the existence we first note that by
Lemma~\ref{constinthird} the right-hand side depends only on
$\dd H$. This shows that
$\dpair{~,~}$ is well-defined as a map $\acola(U)\otimes \acola(U)
\otimes \Omega^1(U) \to \Cp$, where the tensors are taken over
$\Cp$. Lemma~\ref{7.1} shows that the kernel of the map $G\otimes \dd H \to
G\dd H$ from $\acola(U) \otimes \Omega^1(U)$ to $\Ocola(U)$ is contained
in $A(U) \otimes \Omega^1(U)$ so it is enough to observe that that if
$g $ in $  A(U)$ then $\pair{F,g;H}_\gl = \pair{F,\int g\dd H}_\gl$ indeed
depends only on the form $g\dd H$.
\end{proof}

The interest in the pairing $ \dpair{~,~} $ lies in the fact that
its restriction to
$\acola(U) \otimes \Omega^1(U)$ is given by 
\[
  \dpair{F,\dd G}= \pair{F,G}_\gl
\,.  
\]
The pairing on the right was studied in~\cite{Bes98b}. 
It is
known to depend only on $\dd F$, and if $\dd F$, $\dd G$ give cohomology classes
that extend to $C$ it is simply given by the cup product. This proves
part of the following result.

\begin{proposition} \label{premaketrue}
  Let $\omega $ in $  \Omega^1(U)$, such that $[\omega]$ extends to
  $C$, and
  let $F=F_\omega $ in $  \acola(U)$ be a Coleman integral of $\omega$. The
  functional $L_\omega(\eta)=\dpair{F,\eta}$ on $\Ocola(U)$ is good in
  the sense of Definition~\ref{proplist1}.
\end{proposition}

\begin{proof}
Note that we are not claiming that this functional is independent of
the choice of the constant of integration. The only property we need
to prove is that $L_\omega$ vanishes on forms of type $\dd (a\log f)$,
with $ a $ and $ f $ in $ A(U) $. This is easily established:
\begin{align*}
  \dpair{F,\dd (a\log f)}&=\dpair{F,a \dlog f}+\dpair{F, \log f \dd a}\\&=\pair{F,a;\log f}_\gl +\pair{F,a;\log f}_\gl \\
  &= \pair{a,\log f;F}_\gl =0
\end{align*}
by Proposition~\ref{trip-recip}.
\end{proof}

\begin{corollary} \label{factorcor}
The composition
$ \K 4 3 {\O} \xrightarrow{\reg} \Hdr^1(C)  $
factorizes through the quotient map
$ \K 4 3 {\O} \to \K 4 3 {\O} / \K 3 2 {\O} \cup \OQ $.
\end{corollary}
\begin{proof}
By Proposition~\ref{proplist} and the normalization~\eqref{eq:twist}, the fact that a good
functional for $\omega$ exists implies that the composition
$$ \K 4 3 {\O} \xrightarrow{\reg} \Hdr^1(C)
\xrightarrow{1-\phi^\ast/q^3}  \Hdr^1(C) \xrightarrow{\cup [\omega]}
  K  $$ factors. As this is true for any $\omega$ it follows that  $
  \K 4 3 {\O} \xrightarrow{\reg} \Hdr^1(C)
  \xrightarrow{1-\phi^\ast/q^3}  \Hdr^1(C)$ factors, but
  $1-\phi^\ast/q^3 $ is invertible on $\Hdr^1(C)$ so the result follows.
\end{proof}

Propositions~\ref{premaketrue} and~\ref{proplist} suggest
that we need to compute $\dpair{F,\int_0^\infty \epsilon(g,f)}$. 
We shall manipulate this, by ``making our wishes come true'', in the form of the following proposition.

\begin{proposition} \label{7.6}
  Let $F$ be as in Proposition~\ref{premaketrue} and let $g,f\in
  \O^\ast(\CC^{\loc})$ with $g \neq 1$. Let $\int_0^\infty \epsilon(g,f)$ be as in~\eqref{basform}. Then we have 
  \begin{equation} \label{basic-eq}
  \dpair{F,\int_0^\infty \epsilon(g,f)}=
      \sum_e \someterm(g,f,F)_e
      \,,
  \end{equation}
  where
  \begin{equation} \label{basic-eq1}
    \begin{split}
      \someterm(g,f,F)_e= &\ovq \pair{\log f_0,\log g;\int F \dlog
        (1-g)}_e \\+ &\ovqs \pair{\logf(f),\log(g); \int
        F\dlog(1-g)_0}_e \\ + &\ovqt \pair{\logf(f),\log(g_0); \int F
        \phi^\ast \dlog (1-g)}_e \,.
    \end{split}
  \end{equation}
\end{proposition}

\begin{proof}
We have by~\eqref{basform} and~\eqref{dpairdef}
\begin{align*}
  \dpair{F,\int_0^\infty \epsilon(g,f)}=
  \sum_e\Big(&\ovq \pair{F,\log g;\int \log f_0 \dlog
    (1-g)}_e\\
  - &\ovqs \pair{F,\log g; \int \log(1-g)_0 \dlogf(f)}_e\\
      +&\ovq \pair{F,\Theta(g);\dlogf(f)}_e \Big)\,.
\end{align*}
Note that $d F$ is in $\Omega^1(U)$ and has trivial residues along all
annuli ends. It follows that $F$ is holomorphic on each annuli end.

At every annulus $e$ we obtain the identities
\begin{align*}
  \pair{F,\log g;\int \log f_0 \dlog(1-g)}_e &=
  \pair{\log(g),\int F \log f_0 \dlog(1-g)}_e\\&=\pair{\log f_0,\log g;\int
    F \dlog (1-g)}_e\\
  \pair{F,\log g; \int \log(1-g)_0 \dlogf(f)}_e &=
  \pair{\log g, \int F \log(1-g)_0 \dlogf(f)}_e\\&=
  \pair{\log g, F\log(1-g)_0;\logf(f)}_e\\
\intertext{and}
  \pair{F,\Theta(g);\dlogf(f)}_e &=
  \res_e F\Theta(g)\dlogf(f)=
  \pair{\logf(f),\Theta(g)F}_e
\end{align*}
so we obtain
\begin{align*}
  \dpair{F,\int_0^\infty \epsilon(g,f)}=
  \sum_e\Big(&\ovq \pair{\log f_0,\log g;\int F \dlog
    (1-g)}_e\\
  - &\ovqs \pair{\log g, F\log(1-g)_0;\logf(f)}_e\\
      -&\ovq \pair{\logf(f),\Theta(g)F}_e \Big)\,.
\end{align*}
To equate this with the right-hand side of~\eqref{basic-eq} we now
realize our wishes one by one. First we notice that
the first summands in each expression are identical. The realization
of the first wish corresponds to the formula
\begin{align*}
    &\phantom{\,=\,} \sum_e \pair{\logf(f),\Theta(g)F}_e\\ &=
     \sum_e \pair{\logf(f),\int Fd\Theta(g)}_e+
     \sum_e \pair{\logf(f),\int\Theta(g)\dd F}_e\\ &=
     \sum_e \pair{\logf(f),\int Fd\Theta(g)}_e \,,
\end{align*}
as the second sum on the second line vanishes
by~\cite[Corollary~\ref{reciprocity}]{Bes98b}.
Now we may use the formula~\eqref{Theta} for $d\Theta(g)$ to write this as
\begin{align*}
  \sum_e\Big(&\ovq \pair{\logf(f),F\log(1-g)_0;\log(g)}_e\\
  -&\ovqs \pair{\logf(f),\log(g_0); \int F \phi^\ast \dlog (1-g)}_e\Big)\,,
\end{align*}
so the left-hand side of~\eqref{basic-eq} becomes
\begin{align*}
  \sum_e\Big(&\ovq \pair{\log f_0,\log g;\int F \dlog
    (1-g)}_e\\
  - &\ovqs \pair{\log g, F\log(1-g)_0;\logf(f)}_e\\
  -&\ovqs \pair{\logf(f),F\log(1-g)_0;\log(g)}_e\\
  +&\ovqt \pair{\logf(f),\log(g_0); \int F \phi^\ast \dlog (1-g)}_e
  \Big)\,.
\end{align*}
Now the last term also agrees with the last term of the right-hand side of 
\eqref{basic-eq} and we are left with verifying the realization of the
second wish in the form of
\begin{align*}
   \sum_e\Big(&\pair{\log g, F\log(1-g)_0;\logf(f)}_e\\+
  &\pair{\logf(f),F\log(1-g)_0;\log(g)}_e\\ +
  &\pair{\logf(f),\log(g); \int F\dlog(1-g)_0}_e\Big) =0 \,.
\end{align*}
If the last triple index is replaced by $\pair{\logf(f),\log(g);
F\log(1-g)_0}_e$ the result is an immediate consequence of the
triple identity, and indeed we have
\begin{align*}
  &\sum_e\pair{\logf(f),\log(g); \int F\dlog(1-g)_0}_e\\
  = \phantom{-} &\sum_e\pair{\logf(f),\log(g); F\log(1-g)_0}_e\\
              - &\sum_e\pair{\logf(f),\log(g); \int \log(1-g)_0 \dd F}_e \,,
\end{align*}
and the last sum is $0$ by Proposition~\ref{trip-recip}.
\end{proof}

\begin{proposition} \label{G-expr}
  Let $G$ be such that $d G$ lies in $ \Omega^1(U)$ and $G$ is holomorphic on
  annuli ends. Then, with the notation of Proposition~\ref{7.6}, we
  have  
\begin{equation*}
   \someterm(g,f,\phi^\ast G)_e
 = \pair{\log(f),\log(g);\int (\phi^\ast-\ovqs) G
    \dlog(1-g)}_e\,.
\end{equation*}
\end{proposition}

\begin{proof}
Let $F=\phi^\ast G$. We replace in~\eqref{basic-eq1} each term of the
  form $h_0$ by $q\log(h) -\logf(h)$. Then we get
\begin{alignat*}{1}
  \someterm(g,f,F)_e = & \, \phantom{\,+\,\,\,} \ovq \pair{q\log(f) -\logf(f),\log g;\int F \dlog (1-g)}_e
\\
 & + \ovqs \pair{\logf(f),\log(g); q\int F\dlog(1-g)-\int F\dlogf(1-g)}_e
\\
 & + \ovqt \pair{\logf(f),q\log(g)-\logf(g); \int F \phi^\ast \dlog (1-g)}_e
\,,
\end{alignat*}
which after some cancelations equals
\begin{equation*}
  \pair{\log(f),\log(g);\int F \dlog(1-g)}_e - \ovqt \pair{\logf(f),\logf(g);\int F \dlogf(1-g)}_e
\,.
\end{equation*}
After substituting $\phi^\ast G$ for $F$ this becomes
\begin{alignat*}{1}
  &  \pair{\log(f),\log(g);\int \phi^\ast G \dlog(1-g)}_e
  - \ovqs \pair{\log(f),\log(g);\int G \dlog(1-g)}_e
\\
  = \, & \pair{\log(f),\log(g);\int (\phi^\ast-\ovqs) G \dlog(1-g)}_e
\end{alignat*}
as required.
\end{proof}

We now proceed to apply this theory to elements in $ K $-theory.

\begin{theorem} \label{aux-thm}
1. Suppose that an element $ \beta \in  \K 4 3 {\CC^\loc} $ maps to 
$ \sum_i [g_i]_2 \cup f_i $ in $ H^1(\Ccomp \bullet \O ) $
under the composition (with the last isomorphism from \eqref{H1auxC})
\begin{equation} \label{mainAmap1}
    \K 4 3 {\CC^\loc} \to \K 4 3 {\O} \to \K 4 3 {\O} / \K 3 2 {\O}
    \cup \OQ \xrightarrow{\simeq}  H^1(\Ccomp \bullet \O )
\,,
\end{equation}
and that $ch(\beta)\in  \htms^2(\CC^{\loc},3)$ is the image of
$[\eta]\in \Hdr^1(U)$ under the map
\eqref{eq:twist}.
Let $\omega $ in $ \Omega^1(U)$ have trivial residues along all
  annuli ends of $U$. Then
  \begin{equation} \label{keyregform}
    \pair{F_\omega, F_\eta}_{\gl} = \sum_i \sum_e
  \pair{\log(f_i),\log(g_i);\int F_\omega \dlog(1-g_i)}_e\,,
  \end{equation}
  where $F_\omega$ and $F_\eta$ are any Coleman integrals of $\omega$ and 
  $\eta$ respectively.

  2. In particular, the composition
  \begin{equation*}
    \K 4 3 {\CC^\loc} \xrightarrow{ch} \htms^2(\CC^{\loc},3)
    \xrightarrow{[\eta] \mapsto \pair{F_\omega, F_\eta}_{\gl}} \Cp
  \end{equation*}
  factors via \eqref{mainAmap1}.
\end{theorem}

\begin{proof}
First one easily checks that the validity of the formula depends only
on the cohomology class of $\omega$. Since the operator $\phi^\ast
-1/q^2$ is invertible on $H^1(U)$ we can assume that $\omega =
(\phi^\ast -1/q^2)\mu$ with $ \mu $ in $ \Omega^1(U)$ and that
$F_\omega=(\phi^\ast -1/q^2)G$ with
$G$ a Coleman integral of $\mu$. Notice that $G$ satisfies the
condition of Proposition~\ref{G-expr}. Let $\eta_0$ be $ch(\beta)$
$\in \htms^2(\CC^{\loc},3)$ in the model~\eqref{eq:Cmodel} so that by
\eqref{eq:twist} we have $\eta_0 = (1-\phi^\ast/q^3) \eta$. We can
take the Coleman integral of $\eta_0$ to be $F_{\eta_0} =
(1-\phi^\ast/q^3) F_\eta$. Let $F=\phi^\ast G$. By
Proposition~\ref{premaketrue} the functional
$L_\omega(\eta)=\dpair{F,\eta}$ is good in the sense of
Definition~\ref{proplist1}. It follows that we may apply
Proposition~\ref{proplist} to obtain
\begin{align*}
  \dpair{F,\eta_0} &= \sum_i \dpair{F, \int_0^\infty
  \epsilon(g_i,f_i)} \\
  &= \sum_i \sum_e \someterm(g_i,f_i,F)_e \quad
  \text{by Proposition~\ref{7.6}}\\
  &= \sum_i \sum_e \pair{\log(f),\log(g);\int (\phi^\ast-\ovqs) G
    \dlog(1-g)}_e \quad
\end{align*}
by Proposition~\ref{G-expr}.  On the other hand, we have
\begin{align*}
  \dpair{F,\eta_0} & = \pair{F,F_{\eta_0}}_\gl\\
  & = \pair{F,(1-\fdiv{3})F_\eta}_\gl \\
  & = \pair{\phi^\ast G,F_\eta-\fdiv{3}F_\eta}_\gl\\
  & = \pair{\phi^\ast G,F_\eta}_\gl 
       - \pair{\ovqs G,F_\eta}_\gl \\
  & = \pair{(\phi^\ast -\ovqs) G,F_\eta}_\gl \,,
\end{align*}
so our formula was proved with $(\phi^\ast -\ovqs) G$ as required.
\end{proof}

We can restate the first part of Theorem~\ref{aux-thm} in a form that is more convenient
for the rest
of this paper. As explained in the introduction, one has a canonical
projection $\Hdr^1(U) \xrightarrow{\canproj}
\Hdr^1(C/K)$. This is the unique Frobenius equivariant splitting of
the natural restriction map in the other direction.

Recall now the Definition~\ref{regdefined} of the regulator map
$\reg$, using the projection map $\canproj$. 
 It follows from
~\cite[Prop~4.10]{Bes98b} that $\canproj$ can be described in the
following way. It is the unique map such that for any $\eta \in
\Omega^1(U)$ and for any form of the second kind $\omega$ on $C$,
which is holomorphic on $U$, one has
\begin{equation} \label{eq:projform}
   (\canproj\eta) \cup [\omega] = \pair{F_\eta,F_\omega,}_\gl 
 \,.
\end{equation}

\begin{corollary} \label{companioncor}
 Suppose that an element $ \beta \in  \K 4 3 {\CC^\loc} $ maps to 
$ \sum_i [g_i]_2 \cup f_i $ in $ H^1(\Ccomp \bullet \O ) $
under \eqref{mainAmap1}. Let $\omega$ be a form of the second kind on $C$
that is holomorphic on $U$. Then $  \reg(\beta) \cup [\omega]$ is
given by the right-hand side of~\eqref{keyregform}.
\end{corollary}

\section{End of the proofs} \label{sec:end}

In this section we prove our main theorems. These will all follow from
manipulations of Theorem~\ref{aux-thm} and Corollary~\ref{companioncor}.

Fix a form $ \omega $ of the second kind on $ C $ and a Coleman integral
$F_\omega$ of $\omega$.
We begin with the proof of Theorem~\ref{main-thm3}.
\begin{lemma} \label{factb}
 The assignment
  \begin{equation*}
    [g]_2 \tensor f \mapsto \sum_e
  \pair{\log(f),\log(g);\int F_\omega \dlog(1-g)}_e
  \end{equation*}
 extends to a well-defined map $
  \regmapc: M_2(\O) \tensor \OQ \to K$.
\end{lemma}

\begin{proof}
For functions $f,g,h \in \O$ the association
  \begin{equation} \label{auxmapa}
    (h,g,f) \mapsto \sum_e
  \pair{\log(f),\log(g);\int F_\omega \dlog(h)}_e
  \end{equation}
  is trilinear by the properties of the triple index, hence defines a map
  $ \OQ^{\otimes 3} \to K$. Recall
that the complex $\scriptM 2 {\O}  $ from \eqref{scriptM2O} has a
differential $\dd: M_2(\O) \to \OQ^{\otimes 2} $ given by $\dd [g]_2 =
(1-g)\otimes g$. The required map is just the composition of $\dd \otimes
  \operatorname{id} $ with the map~\eqref{auxmapa}
\end{proof}

\begin{lemma} \label{facta}
The restriction of $\regmapc$ to $(M_2(\O) \tensor \OQ)^{\dd = 0}  $
coincides with the composition
$$
\xymatrix{
(M_2(\O) \tensor \OQ)^{\dd = 0} \ar[r] &H^2(\scriptM 3 {\O} ) \ar[r] & \K 4 3 {\O} \ar[r]^-{\reg} &
\Hdr^1(C/K) \ar[r]^{\cup \omega} &  K.
}
$$
\end{lemma}

\begin{proof}
This is an immediate consequence of diagram~\eqref{MCO}, noting the
the vertical map on the left there is $[g]_2\otimes f \mapsto [g]_2\cup f $, and of Corollary~\ref{companioncor}.
\end{proof}

\begin{proof}[Proof of Theorem~\ref{main-thm3}]
The only part of the theorem not proven already in Lemmas~\ref{factb}
and~\ref{facta} is that the map $\regmapc$ factors via $H^2(\scriptM 3 {\O} )$,
but this follows immediately from Lemma~\ref{facta}.
\end{proof}

\begin{proof}[Proof of part 1. of Theorem~\ref{main-thm4}]
Recall that the map in question is given by
\begin{align*}
  [g]_2 \tensor f \mapsto 
  &\frac{2}{3} \sum_e \pair{\log(f),\log(g);\int F_\omega
    \dlog(1-g)}_e \\ - &\frac{2}{3}
  \sum_e \pair{\log(f),\log(1-g);\int F_\omega \dlog(g)}_e
\,.
\end{align*}
This is clearly trilinear in $f,g,1-g$ and anti-symmetric in $g$ and
$1-g$, so we can proceed as in the proof of Lemma~\ref{factb}, using now the
differential $\dd:  \tM_2(\O) \to \bigwedge^2 \OQ $
from~\eqref{tildescriptM2O} given by $\dd [g]_2 =
(1-g)\wedge g $.
Clearly, the same formula also gives a well-defined map on $M_2(\O) $ and it will
suffice to show that this map coincides with $\regmapc $ on $(M_2(\O)
\tensor \OQ)^{\dd = 0}  $, as the
composition of maps in Theorem~\ref{main-thm4} factors by definition via
$\scriptM 3 {\O} \to \tildescriptM 3 {\O} $.

Suppose then that $ \sum_i [g_i]_2 \tensor f_i $ is in
$ (M_2(\O) \tensor \OQ)^{\dd = 0} $, so that
$ \sum_i (1-g_i) \tensor (g_i \wedge f_i) = 0 $ by~\eqref{Ocond}.
By~Proposition~\ref{needfortild} we have
 \begin{align*}
   0= \sum_i\sum_e \bigg(&\pair{\log(f_i),\log(g_i);\int F_\omega \dlog(1-g_i)}_e\\
    +&\pair{\log(f_i),\log(1-g_i);\int F_\omega \dlog(g_i)}_e\\
    +&\pair{\log(g_i),\log(1-g_i);\int F_\omega \dlog(f_i)}_e \bigg)\\
  = \sum_i\sum_e \bigg(&\pair{\log(f_i),\log(g_i);\int F_\omega \dlog(1-g_i)}_e\\
    +2&\pair{\log(f_i),\log(1-g_i);\int F_\omega \dlog(g_i)}_e\bigg)
\,,
 \end{align*}
where the last equality follows because
$ \sum_i (1-g_i) \tensor (g_i \wedge f_i) = 0 $,
 so that the sum over $i$ and $e$
 of each of the last two summands is the same. The difference of the
 two maps in question on $ \sum_i [g_i]_2 \tensor f_i $ is $1/3$ of
 the expression above.
\end{proof}

For the proofs of Theorems \ref{main-thm2}~and~\ref{main-thm1}, as
well as part 2 of Theorem~\ref{main-thm4}, we now assume that $\omega$
is a holomorphic form on $C$.

\begin{lemma} \label{threecompat}
  The associations
  \begin{align*}
    [g]_2 \otimes f &\mapsto   \int_{(1-g)} \log(g) F_\omega \dlog
        (f)
      -\int_{(g)} \log(1-g) F_\omega \dlog (f)\\
    [g]_2 \otimes f &\mapsto  \int_{(f)} \Ltwo(g) \omega \\
    [g]_2 \otimes f &\mapsto  \sum_y \ord_y (f) F_\omega(y) \Lmod 2 (g(y)) 
  \end{align*}
  induce well-defined maps
 on $\tM_2(\O)\otimes \OQ$ (first) and
  $M_2(\O)\otimes \OQ$ (last
  two).
\end{lemma}

\begin{proof}
This is essentially clear for the first association, following the
proofs of Lemma~\ref{factb} and of the first part of
Theorem~\ref{main-thm4}.
For the second association, observe that $\dd \Ltwo = \log(z)
\dlog(1-z)$ by~\eqref{dilogfuncs}. Consider the association 
\begin{equation*}
  (h,g,f) \mapsto \int_{(f)} \left( \omega \cdot \int \log(g)
    \dlog(h)\right).
\end{equation*}
Here, the integral $ \int \log(g) \dlog(h) $ is a Coleman integral
defined only up to a constant. However, if the constant changes, the
entire expression changes by the same constant multiplied by
$ \int_{(f)}  \omega $, which equals 0 as it is the $p$-adic Abel-Jacobi map applied to
the principal divisor $(f)$ (see~\cite{Bes97}). This
association is therefore well-defined, clearly trilinear, and we
obtain the required result again as in the proof of
Lemma~\ref{factb}. For the third association, one first needs to note
that $ \Lmod 2 (g(y)) $ is the value of $\Lmod 2 (g) $ at $y$ (this is
not obvious in general because we are using the generalized way of
assigning values to Coleman functions by taking constant terms,
discussed in the introduction) as we shall see in
Corollary~\ref{lmodindep}, so the entire expression can be written
as 
$F_\omega\cdot \Lmod 2 (g) $ evaluated at the divisor of $f$. It is
now possible to proceed as in the previous case, given that
$\dd \Lmod 2 (g) = (\log(g)\dlog(1-g) - \log(1-g)\dlog(g))/2 $, by associating to $f,g,h$ the value at $(f)$ of
$F_\omega\cdot \int \left(\log(g)
  \dlog(1-g)-\log(1-g)\dlog(g))\right)$, where the constant of
integration does not matter for exactly the same reason it did not in
the previous case.
\end{proof}
Thus, in all the remaining theorems to prove, the association does
extend to a map as claimed. We shall next derive the formulas for the
regulator. In all cases, we already have a formula
for the regulator, expressed in terms of a sum of local indices on
annuli. We can use the argument in the proof of
\cite[Proposition~\ref{local-comp}]{Bes98b} using Proposition~\ref{trip-recip}
to replace the sum over annuli ends by a sum over points.

Let $ \alpha = \sum_i [g_i]_2 \tensor f_i $ be an element of
$(M_2(\O) \tensor \OQ)^{\dd = 0}  $. By the above we have
\begin{equation*}
 \regmapc(\alpha) = \sum_i \sum_{y\in C}
  \pair{\log(f_i),\log(g_i);\int F_\omega \dlog(1-g_i)}_y
\,.
\end{equation*}

We again extend scalars to $\Cp$, so in particular points are $\C_p$
valued.

Fix a local parameter at each point $y$, which we shall call
$z_y$, or, whenever there is no risk of confusion, simply
$z$. Consider a single point $y$ in $  C$. We recall that with respect to
the local parameter $z$ at $y$ we define, for a rational function
$f$, $\fbar(y) = (f/z^{\ord_y (f)})(y)$. For such a function $f$  we
have $c_z(\log(f))= \log  (\fbar(y))$.  We also have $\res_y(F_\omega
\dlog(f))= \ord_y (f) \cdot F_\omega(y)$. Thus, using
Proposition~\ref{const-term-expr}, 
we obtain
\begin{equation} \label{trip-with-const}
\begin{split}
   \regmapc(\alpha) = \sum_i \sum_{y\in C}
  \Big[
  &\ord_y(1-g_i) F_\omega(y) \log \fbar_i(y) \log \gbar_i(y)\\ -
  &\ord_y (f_i) c_z\left(\int \log(g_i) F_\omega \dlog(1-g_i)\right)\\ -
  &\ord_y (g_i) c_z\left(\int \log(f_i) F_\omega
  \dlog(1-g_i)\right)\Big]
\,.
\end{split}
\end{equation}

Let $A$ (respectively $B$) be the subgroup of $ k(C)^* $ generated by the
$f_i$ and $g_i$ (respectively by the $1-g_i$). By choosing bases for $A$
and $B$ and then choosing appropriate integrals we can arrange it so
that for each $f $ in $  A$ and $h $ in $  B$ an integral $\int \log(f) F_\omega
\dlog h$ is chosen such that the map $(f,h) \to \int \log(f) F_\omega
\dlog h$ is bilinear. Since the overall sum in~\eqref{trip-with-const}
is independent
of the choice of integrals, we may and do assume from now on that the
integrals there are chosen as above.

\begin{lemma} \label{twotermseq}
 If $ \sum_i [g_i]_2 \tensor f_i $ is in
$(M_2(\O) \tensor \OQ)^{\dd = 0}  $, then for every $ y $ in $ C $ we have
\begin{alignat*}{1}
&    \sum_i 
    \ord_y (f_i) \, c_z\left(\int \log(g_i) F_\omega \dlog(1-g_i)\right)
\\
 = \, &
    \sum_i 
    \ord_y (g_i) \, c_z\left(\int \log(f_i) F_\omega \dlog(1-g_i)\right) 
\,.
  \end{alignat*}
\end{lemma}

\begin{proof}
With the choices above the map 
\begin{equation*}
  (f,g,h) \to 
  \ord_y (f) c_z(\int \log(g) F_\omega \dlog(h)) -
  \ord_y (g) c_z(\int \log(f) F_\omega \dlog(h))
\end{equation*}
is trilinear and anti-symmetric
with respect to $f$ and $g$. The lemma
follows since $ \sum (1-g_i) \tensor (g_i \wedge f_i) = 0 $ by~\eqref{Ocond}. 
\end{proof}

We recall that the function $\Ltwo(z)$ is defined by
$\Ltwo(z) = \Li_2(z) + \log(z)\log(1-z)$ and that we have
$\dd\Ltwo(z) = \log(z)\dlog(1-z)$. Note that this last form is holomorphic in the residue
disc of $1$ and as a consequence so is $\Ltwo(z) $.

\begin{lemma} \label{L2const}
  Let $g$ be a rational function. The constant term at $y$ of $\Ltwo(g)$
  equals $\Ltwo(g(y))$ if $g(y)\ne 0,\infty$, equals $0$ if $g(y)=0$ and
  equals $\log^2(\gbar(y))/2$ if $g(y)=\infty$, where $\gbar$ is
  computed with respect to the same local parameter as the constant
  term. In addition, the expansion of $\Ltwo(g)$ with respect to any
  local parameter $ z $ contains no summands of the form $Const \cdot z^{n}$
  with $n<0$.
\end{lemma}

\begin{proof}
This is clear if $g(y)\ne 0,\infty$. Suppose $g(y)=0$. Since $\Li_2$
is holomorphic near $0$ and has value $0$ there, we see that the
constant term and terms of the form $z^n$ for $n<0$ are the same as in
$\log(g)\log(1-g)$. Near $y$
we have $\log(g(z)) = \ord_y (g) \log(z) +$ a holomorphic function in
$z$. Also, $\log(1-g)$ is holomorphic near $y$ with value $0$
there. Thus the result is clear. Finally, by~\cite[Proposition~6.4]{Col82},
we have $\Ltwo(g) + \Ltwo(1/g) = \log^2(g)/2$ so the result at $g(y)=\infty$ is deduced from that of
$1/g$ when $ g(y) = \infty $.
\end{proof}

\begin{corollary} \label{lmodindep}
The constant term of $\Lmod 2 (z) $ at $\infty$ is $0$ regardless of
parameter. Furthermore, for any rational function $g$ the constant
term of $\Lmod 2 (g) $ at any point $y$ equals $\Lmod 2 (g(y))$
\end{corollary}

\begin{proof}
Since $ \Lmod 2 (z) = \Ltwo(z)-\log(z)\log(1-z)/2 $ it is easy to check
that the constant term of $\Lmod 2 (g) $ is $0$ at either $g(y)=
0,\infty$, from which the result follows.
\end{proof}

\begin{lemma} \label{czindependence}
  For any point $y $ in $  C$ and for any choice of a Coleman integral
  $\int \Ltwo(g)\omega$ the quantity $c_z (\int \Ltwo(g) \omega)$ is
  independent of the choice of the local parameter $z$ at $y$.
\end{lemma}

\begin{proof}
Let $f_\omega$ be the unique Coleman integral of $\omega$ that
vanishes at $y$. We may choose a Coleman integral $\int f_\omega d
\Ltwo(g)$ in such a way that the integration by parts formula
\begin{equation*}
  \int \Ltwo(g) \omega = \Ltwo(g) f_\omega - \int f_\omega \dd \Ltwo(g)
\end{equation*}
holds. It is therefore sufficient to show that the constant term of
each of the summands on the right is independent of the parameter. From
the last assertion in Lemma~\ref{L2const} and the fact that $f_\omega(y)=0$
it is easy to see that the constant term of the first summand is $0$. 
For the second summand we have
\begin{align*}
  \int f_\omega \dd \Ltwo(g) & = \int f_\omega \log(g) \dlog (1-g) \\
       & =  \log(g) \int f_\omega \dlog(1-g) - \int (\int f_\omega \dlog(1-g))
  \dlog(g)
\end{align*}
for appropriate choices of integrals. As $f_\omega \dlog(1-g)$ is
holomorphic at $y$, we may arrange it so that $\int f_\omega \dlog(1-g)$
vanishes at $y$. Then in the last formula the first term has constant
term $0$ while the second term is holomorphic at $y$ hence its
constant term is independent of $z$.
\end{proof}

Using the last lemma we may set
\begin{equation*}
  \int \Ltwo(g) \omega|_y := c_z \left(\int \Ltwo(g) \omega\right)
\end{equation*}
with respect to any parameter $z$ at $y$.  Using this we can define
$\int_D \Ltwo(g)\omega$ for any divisor $D$ of degree zero.
If we change $\int \Ltwo(g)
\omega$ by a constant, its value at $y$ in the above sense will
change by the same constant. Thus when $D$ has degree $0$ the integral
$\int_D \Ltwo(g) \omega$ does not depend on the constant of integration
even if $D$ and the divisor of $g$ have a common support. This
explains the general definition of the integral in Theorem~\ref{main-thm1}.

\begin{lemma} \label{intlconst}
  Choose integrals such that the integration by parts formula 
  \begin{equation*}
  \int \log(g)F_\omega \dlog(1-g) = F_\omega \Ltwo(g) - \int \Ltwo(g)\omega
  \end{equation*}
  is satisfied. Then we have at a point $y$ and with respect to the local
  parameter $z$,
  \begin{equation*}
    c_z( \int \log(g)F_\omega \dlog(1-g)) = F_\omega(y) c_z(\Ltwo(g))
    - \int \Ltwo(g)\omega|_y
  \,.
  \end{equation*}
\end{lemma}

\begin{proof}
One just applies $c_z$ to the integration by parts formula and
observes that by Lemma~\ref{L2const} we have 
$ c_z(F_\omega \Ltwo(g))= F_\omega(y) c_z(\Ltwo(g)) $.
\end{proof}

\begin{proof}[Proof of Theorem~\textup{\ref{main-thm2}}]
We already saw that the association gives a well-defined map on
$M_2(\O)\otimes \OQ$. It therefore suffices to show that it gives
the same map on $(M_2(\O) \tensor \OQ)^{\dd = 0}  $ as  $\regmapc$
in Theorem~\ref{main-thm3}.
Consider~\eqref{trip-with-const}. By Lemma~\ref{twotermseq} we can
choose our integrals such that for each 
point $y$ the sum over $i$ of each of the last two terms is
identical. The term $\ord_y (f_i) c_z(\int \log(g_i) F_\omega
\dlog(1-g_i))$ is computed in Lemmas~\ref{intlconst} and~\ref{L2const}.
Substituting the results we see that we have the equation
\begin{equation*}
\begin{split}
  \regmapc(\alpha)= \sum_i \Big[\sum_{y\in C}
  &(\ord_y(1-g_i) F_\omega(y) \log \fbar_i(y) \log \gbar_i(y))\\ +
  &2 \int_{(f_i)} \Ltwo(g_i) \omega\\ - &\sum_{y\in C}
  \ord_y (f_i) F_\omega(y) \times 
  \begin{cases}
    0 & g_i(y)=0\\
    2\Ltwo(g_i(y))& g_i(y)\ne 0,\infty\\
    \log^2(\gbar_i(y)) & g_i(y)=\infty
  \end{cases}
  \Big]\,.
\end{split}
\end{equation*}
In the first sum over $y$ only terms with $g_i(y)=\infty$ can be non-zero.
Thus neither sum over $y$ contributes for $g_i(y)=0$, and
the right-hand side becomes
\begin{equation} \label{almostend}
  \sum_i \left[
  2 \int_{(f_i)} \Ltwo(g_i)\omega - 2\sum_{g_i(y)\ne 0,\infty} 
  \ord_y (f_i) F_\omega(y) \Ltwo(g_i(y))
  +\sum_{g_i(y)=\infty} F_\omega(y) \alpha_y(f_i,g_i)\,\right] 
\end{equation}
with
\begin{align*}
  \alpha_y(f,g) &= \ord_y(1-g) \log \fbar(y) \log \gbar(y) - \ord_y (f) \log^2 \gbar(y)
\\
  &= \log \gbar(y) \left(\ord_y(1-g) \log \fbar(y)  - \ord_y (f) \log \gbar(y)\right)
\\
  &= \log \overline{1-g}(y) \left(\ord_y (g) \log \fbar(y) - \ord_y (f) \log \gbar(y)\right) 
\end{align*}
because $g(y)=\infty$ implies $\ord_y (1-g)=\ord_y (g)$
and $\gbar(y)=-\overline{1-g}(y)$.

For $y $ in $  C$, the function $\beta_y(f,g,h):=\log \overline{h}(y)
\left(\ord_y (g) 
\log \fbar(y)  -   \ord_y (f) \log \gbar(y)\right)$ is trilinear in
$f$, $g$ and $h$ and anti-symmetric in $f$ and $g$. 
As $\sum_i (1-g_i) \tensor (g_i \wedge f_i)= 0 $
by~\eqref{Ocond}, we find
\begin{equation} \label{sumbeta}
  \sum_i \beta_y(f_i,g_i,1-g_i)=0
\,.
\end{equation}
If $g_i(y)=0$ then
$\beta_y(f_i,g_i,1-g_i) =0$ while if $g(y)\ne 0,\infty$ then
$\beta_y(f_i,g_i,1-g_i)=-\ord_y (f_i)\log g_i(y) \log (1-g_i(y))$,
where we set the value of $\log(y)\log(1-y) $ at $1$ to be $0$, which
is consistent with taking limits and with what follows.
Thus, summing~\eqref{sumbeta} multiplied by $F_\omega(y)$
over all $y $ in $  C$ we see that 
\begin{equation*}
  \sum_i \sum_{g_i(y)=\infty} F_\omega(y) \alpha_y(f_i,g_i) =
  \sum_i \sum_{g_i(y)\ne 0,\infty}\ord_y (f_i) F_\omega(y) \log g_i(y) \log
  (1-g_i(y))
\,.
\end{equation*}
Substituting this into~\eqref{almostend},
and using that $\Ltwo(z)- \log(z)\log(1-z)/2 = \Lmod 2 (z) $ by definition,
we obtain
\[
 \sum_e \pair{F_\eta, F_\omega}_e = 
  2 \sum_i \int_{(f_i)} \Ltwo(g_i)\omega - 2 \sum_i\sum_{g_i(y)\ne 0,\infty} \ord_y (f_i) F_\omega(y) \Lmod 2 (g_i(y))
\,.
\]
This formula finishes the proof of Theorem~\ref{main-thm2} as
$\Lmod 2 (z) $ vanishes at $0$ and $\infty $.
\end{proof}

\begin{proof}[Proof of Theorem~\textup{\ref{main-thm1}}]
That the assignment is well-defined is part of Lemma~\ref{threecompat}.
In order to see that it vanishes on $ [f]_2 \otimes f $, we note
that we already know this is true for the assignment in Theorem~\ref{main-thm2},
and that the second term in that assignment
is trivial on such terms because $ \Lmod 2 (z) $ vanishes at
$ 0 $ and $ \infty $.

For part~(2), consider~\eqref{MAP}.
That $ \partial_1(\alpha') = 0 $ means that $ \alpha' $
satisfies~\eqref{CCcond}, which is equivalent with $ \alpha' $ being
in $ H^2(\scriptM 3 {\CC'}) $ inside $ H^2(\scriptM 3 {\O'}) $
(recall from Section~\ref{M3CCconstruction} that the two vertical
maps at the top in this diagram are injections if we use $ \O' $
instead of $ \O $ everywhere).
The existence and uniqueness of $ \beta' $ was therefore proven just after~\eqref{FOCcd}.
In fact, $ \beta' $ is the $ \K 4 3 {\CC'} $ component of the image of $ \alpha' $
in $ \K 4 3 {\CC'} \oplus \K 3 2 k \cup \OpQ $,
and the images of $ \alpha' $ and $ \beta' $ in 
$ \K 4 3 {\O'} $ differ by some $ \gamma' $ in the image of $ \K 3 2 k \cup \OpQ $.
But $ \reg(\gamma')\cup\omega = 0$ by the commutativity of the
bottom right square, so that, after extending from $ \O' $ to
$ \O $, we have $ \reg(\beta)\cup \omega = \regmapb(\alpha) $
by Theorem~\ref{main-thm2}.
It therefore suffices to show that the contribution of
each $ \ord_y(f) F_\omega(g(y)) \Lmod2 (g(y)) $
in $ \regmapb(\alpha) $ is trivial.

Note that in Theorem~\ref{main-thm2} this sum has to
be computed after a suitable finite extension $ \tilde K $ of $ K $ that
makes the relevant $ y $ rational, but that further extending
the field to $ \Cp $ as we are using here gives the same result.
In fact, because we start over the number field $ k $, the relevant
$ y $ become rational over some number field $ L \subset \tilde K $ containing
$ k $.  The $ \tM_2(\cdot) $ are compatible 
with field extensions, and clearly the same holds for $ \partial_1 $.
Therefore~\eqref{CCcond} gives us that for each closed point $ y $
of $ C_L' $, $ \partial_{1,y}(\alpha') $ is trivial in $ \tM_2(L) $.
Because $ F_\omega(y) $ is just a constant, comparing with
the definition of $ \partial_{1,y} $ in Section~\ref{scriptM3Cconstruction},
we see that it suffices to show that
 the map 
\begin{align*}
 H^1(\tildescriptM 2 L )  & \to \tilde K
\\
\sum_i [a_i]_2 & \mapsto \sum_i \Lmod 2 (a_i)
\end{align*}
is well-defined.
It is conjectured in \cite[Conjecture~1.14]{Bes-deJ98}
that this map is the syntomic regulator map on  as composition
(with $ \O_L $ the ring of integers in $ L $)
\[
H^1(\tildescriptM 2 L ) \to \K 3 2 L \iso \K 3 2 \O_L \to \hsyn^1(\O_L,2) \iso K
\,,
\]
which would imply what we need.
However, extending the domain of the map, we can show by more
basic means that the map
\[
\begin{aligned}
\tM_2(L) & \to \tilde K
\\
[a]_2 & \mapsto \Lmod 2 (a)
\end{aligned}
\]
is well-defined, which will prove what we want.

Namely for any field $ L $ of characteristic zero,
let $ B_2'(L) $ be the free $ \Q $-vector space on elements
$ \{ b \}_2 $ with $ b $ in $ F $, $ b \neq 0,1 $, modulo the
five term relation
\begin{equation} \label{fivetermrelation}
 \def\gscmb #1 #2 {\left\{ #1 \right\}_{#2}}
   \gscmb b 2 + \gscmb c 2 + \gscmb \frac{1- b }{1-b c } 2
   +\gscmb {1-b c} 2
   +\gscmb \frac{ 1-c }{ 1-b c } 2  = 0
\,.
\end{equation}
It is shown in \cite[Lemma~5.2]{Jeu00} that
there is a map $ B_2'(L) \to \tM_2(L) $, given by sending $ \{ b \}_2 $
to $ [b]_2 $.
In case $ L $ is a number field, this was
already done on page~240 of \cite{Jeu95} (where the relations
were not made explicit and the group was called $ B_2(L) $),
and the map was shown to be an isomorphism in that case.
Finally, in \cite[Corollaries~6.4(ii),(iii) and~6.5b]{Col82} Coleman shows that $ \Lmod 2 $
(which is called $ D $ there)
satisfies
\[
\begin{aligned}
\Lmod 2 (z^{-1}) & = - \Lmod 2 (z)
\cr
\Lmod 2 (1-z) & = - \Lmod 2 (z)
\end{aligned}
\]
as well as (with signs corrected)
\[
  \Lmod 2 (z_1z_2) = \Lmod 2 (z_1) + \Lmod 2 (z_2)
  + \Lmod 2 \left( \frac{z_1(1-z_2)}{z_1-1} \right)
  + \Lmod 2 \left(\frac{z_2(1-z_1)}{z_2-1}  \right)
\,.
\]
Substituting $ z_1 = (bc)^{-1} $, $ z_2=c $ in the last relation
and using the first two, one sees
that $ \Lmod 2 $ satisfies
the relation corresponding to~\eqref{fivetermrelation}.  Therefore
it induces a map
\[
\tM_2(L) \iso B_2'(L) \to K
\]
mapping $ [b]_2 $ to $ \Lmod 2 (b) $.
This finishes the proof of~Theorem~\ref{main-thm1}.
\end{proof}

\begin{proof}[Proof of Theorem~\textup{\ref{main-thm4}} part 2]
We already saw in Lemma~\ref{threecompat} that the formula gives a
well-defined map on $\tM_2(\O)\otimes \OQ$, so it remains as usual to
derive it
from the corresponding formula in the first part. Suppose that$ \sum_i [g_i]_2 \tensor f_i \in (\tM_2(\O)
\tensor \OQ)^{\dd = 0}  $. Looking
at~\eqref{trip-with-const} and replacing the $g_i$'s and $1-g_i$'s we
get the
following formula for the regulator.
\begin{equation} \label{sixterm}
\begin{split}
  \sum_i \sum_{y\in C}
  \Big[
  &\ord_y(1-g_i) F_\omega(y) \log \fbar_i(y) \log \gbar_i(y)
  -\ord_y(g_i) F_\omega(y) \log \fbar_i(y) \log \overline{1-g_i}(y)\\ +
  &\ord_y (f_i) c_z\left(\int \log(1-g_i) F_\omega \dlog(g_i)\right) -
  \ord_y (f_i) c_z\left(\int \log(g_i) F_\omega \dlog(1-g_i)\right)\\ +
  &\ord_y (1-g_i) c_z\left(\int \log(f_i) F_\omega
    \dlog(g_i)\right)-
  \ord_y (g_i) c_z\left(\int \log(f_i) F_\omega \dlog(1-g_i)\right)\Big]
\end{split}
\end{equation}
For a given $i$ and $y$ one observes that the first two terms in the
inner term add up to $0$. Indeed, there can be contributions only if
$g_i(y)$ is either $0$, $1$ or $\infty$ but in the first two cases
either the order or the logarithm will make the two terms vanish. If
$g_i(y)=\infty$ the orders are the same so we get a multiple of
$\log(\overline{g_i/(1-g_i)}(y))=\log(-1)=0$.

For $3$ functions $f$, $g$ and $h$ consider the trilinear expression
\[I(f,g,h):=\ord_y(g) 
c_z\left(\int \log(f) F_\omega \dlog(h)\right)\] and the expression
$\sum_\sigma \sgn(\sigma) I(f,g,h)^\sigma$ where $I(f,g,h)^\sigma$ means permuting
the order of terms according to $\sigma$. This
expression is alternating and since, 
by~\eqref{otild5} we have  $\sum_i (1-g_i) \wedge g_i \wedge f_i = \dd(
\sum_i [g_i]_2 \tensor f_i) = 0 $,
we have $\sum_i \sum_\sigma 
\sgn(\sigma) I(f_i,g_i,1-g_i)^\sigma = 0$. This implies that the expression
in~\eqref{sixterm} equals
\begin{align*}
   \sum_i \sum_{y\in C}
  \Big[
  &\ord_y (g_i) c_z\left(\int \log(1-g_i) F_\omega \dlog(f_i)\right)\\-
  &\ord_y (1-g_i) c_z\left(\int \log(g_i) F_\omega \dlog(f_i)\right)\Big]
\end{align*}
and so this gives the formula in the theorem if one is willing to use
constant terms to evaluate at the singular points, but it is easy to
see that the difference of the constant terms at infinity is as
interpreted in the theorem.
\end{proof}

\begin{remark} \label{compremark}
We would like to explain a bit of the heuristics suggesting that
Theorem~\ref{main-thm4} gives a formula which is the $p$-adic analogue
of the complex analytic formula for the regulator in Section~\ref{sec:classical}.

Experience has taught us that complex surface integrals translate
in the $p$-adic world to a similar formula involving local
indices. For example, the complex analytic formula for the
regulator of the symbol $\{f,g\}$ in $ K_2(F)$,
\begin{equation*}
  \int_C \log |g| \overline{\dlog f} \wedge \omega
\end{equation*}
translates in the $p$-adic world, for $ \omega $ of the second kind,
 into the formula
\begin{equation*}
  \pair{\log f,F_\omega;\log g}_\gl
\,.
\end{equation*}

Note that, using the rules for the triple index, 
this is the same
as the formula $\sum_e \pair{\log f,\int( F_\omega\dlog(g))}_e $
obtained in~\cite[Propositon~5.1]{Bes98b}. This corresponds to the
regulator on an open curve using the same projection on $\Hdr^1(C/K)$
 we have been using in this paper. 
For a sum $ \{f_i,g_i\} $ in the kernel of the tame symbol, we
may, for every pair $ (f,g) = (f_i,g_i) $, replace 
$ \pair{\log f,F_\omega;\log g}_\gl $ with  $\int_{(f)} \log(g) \cdot \omega$,
obtaining the formula of Coleman and de Shalit.
This is similar to
Theorem~\ref{main-thm2} specializing to Theorem~\ref{main-thm1}.

Relying on these considerations, we might expect that, for $ \omega $
of the second kind, the regulator formula for $ H^2(\scriptM 3 {\O} ) $
will be given by mapping $ [g]_2 \otimes f $
\begin{equation*}
  \frac{8}{3} \, \pair{\log(f),F_\omega;\Lmod 2 (g)}_\gl
\,.
\end{equation*}
But this is not well-defined because $\dd \Lmod 2 (g) =
(\log(g)\dlog(1-g) - \log(1-g)\dlog(g))/2$ is
not holomorphic. However, we can make the following interpretaton:
\begin{align*}
  &
  \phantom{\,=\,}  \pair{\log(f),\int F_\omega \dd \Lmod 2 (g) }_\gl
\\
&=
  \frac{1}{2}
   \pair{\log(f), \int  F_\omega (\log(g)\dlog(1-g) -
     \log(1-g)\dlog(g))}_\gl \\ &=
 \frac{1}{2} \pair{\log(f), \log(g); \int  F_\omega \dlog(1-g)}_\gl
 -\frac{1}{2} \pair{\log(f), \log(1-g); \int  F_\omega \dlog(g))}_\gl
\end{align*}
by Lemma~\ref{recepy}(1).
Thus, the resulting heuristic formula differs from the correct formula
in Theorem~\ref{main-thm4} by a factor of $4$. Factors which are powers of 2 appear in comparison of other regulator formulas; see for example the introduction of~\cite{Bes10}.
\end{remark}

\end{document}